\documentclass[letterpaper,11pt]{article}
\usepackage{graphicx} 
\usepackage{amsfonts, amsmath, amssymb, color,  algorithm2e,amsthm}

\usepackage[margin=1in,footskip=0.25in]{geometry}
\usepackage{hyperref}
\usepackage{caption}
\usepackage{authblk}

\newtheorem{definition}{Definition}
\newtheorem{thm}{Theorem}[section]
\newtheorem{rem}{Remark}
\newtheorem{prop}[thm]{Proposition}
\newtheorem{lem}[thm]{Lemma} 
\newtheorem{cor}[thm]{Corollary}

\newcommand{\argmin}{\mathop{\mathrm{argmin}}}
\newcommand{\argmax}{\mathop{\mathrm{argmax}}}

\def\cA{\mathcal{A}}
\def\cB{\mathcal{B}}
\def\cC{\mathcal{C}}

\def\cG{\mathcal{G}}

\def\cL{\mathcal{L}}

\def\cN{\mathcal{N}}

\def\cP{\mathcal{P}}

\def\cR{\mathcal{R}}
\def\cS{\mathcal{S}}
\def\cT{\mathcal{T}}

\def\cV{\mathcal{V}}
\def\cW{\mathcal{W}}

\def\bW{\boldsymbol{A}}

\def\bG{\boldsymbol{G}}

\def\bR{\boldsymbol{R}}
\def\bS{\boldsymbol{S}}

\def\bV{\boldsymbol{V}}
\def\bW{\boldsymbol{W}}

\def\bbM{\mathbb{M}}

\def\bbR{\mathbb{R}}

\newcommand{\E}{\operatorname{\mathbb{E}}}
\renewcommand{\P}{\operatorname{\mathbb{P}}}
\newcommand{\tr}{\mathrm{Tr}}
\newcommand{\pa}[1]{\left(#1\right)}
\newcommand{\ac}[1]{\left\{#1\right\}}
\newcommand{\cro}[1]{\left[#1\right]}
\newcommand{\<}{\langle}
\renewcommand{\>}{\rangle}
\newcommand{\eps}{\varepsilon}

\newcommand{\R}{\mathbb{R}}
\newcommand{\1}{\mathbf{1}}
\newcommand{\N}{\mathbb{N}}

\newcommand{\cumul}{\mathop{cumul}}

\usepackage{comment}
\excludecomment{mute}

\usepackage[textwidth=2.0cm, textsize=tiny]{todonotes} 

\title{Computational barriers for permutation-based problems, and cumulants of weakly dependent random variables}

\author{Bertrand Even\footnote{Laboratoire de Math\'ematiques d'Orsay, Universit\'e Paris-Saclay, CNRS, France. bertrand.even@universite-paris-saclay.fr} \  and Christophe Giraud\footnote{Laboratoire de Math\'ematiques d'Orsay, Universit\'e Paris-Saclay, CNRS, France. Christophe.Giraud@universite-paris-saclay.fr} \ and Nicolas Verzelen\footnote{INRAE, Montpellier SupAgro, MISTEA, Univ. Montpellier, France. Nicolas.Verzelen@inrae.fr} }

\date{}

\begin{document}

\maketitle

\begin{abstract}
In many high-dimensional problems, polynomial-time algorithms fall short of achieving the statistical limits attainable without computational constraints. A powerful approach to probe the limits of polynomial-time algorithms is to study the performance of low-degree polynomials. The seminal work of \cite{SchrammWein22} connects low-degree lower bounds to multivariate cumulants. Prior works \cite{luo2023computational, Even25a} leverage independence among latent variables to bound cumulants. However, such approaches break down for problems with latent structure lacking independence, such as those involving random permutations. To address this important restriction, we develop a technique to upper-bound cumulants under weak dependencies -- such as those arising from sampling without replacement or random permutations. To show-case the effectiveness of our approach, we uncover evidence of statistical--computational gaps in multiple feature matching and in seriation problems. 

\end{abstract}

\section{Introduction}
In high-dimensional statistics, the main objective is to design computationally efficient estimation procedures with optimal statistical performance. However, in problems like sparse PCA, planted clique, and clustering, the best known polynomial-time algorithms often fall short of the performance achievable by   estimators computationally unconstrained. This has led to conjectures about statistical--computational gap discrepancies between statistically optimal solutions, and those achievable in polynomial time.
To fairly evaluate a computationally efficient algorithm, its performance should be compared not to the unconstrained optimum, but to the best performance achievable in polynomial time. This shift motivates the study of lower bounds for polynomial-time algorithms across various problems. Since high-dimensional problems are inherently probabilistic, classical worst-case complexity, such as  P vs NP, is less relevant. Instead, researchers focus on specific computational models such as Sum-of-Squares (SoS)  \cite{HopkinsFOCS17,Barak19}, statistical queries~\cite{kearns1998efficient,brennan2020statistical}, approximate message passing~\cite{MezardMontanari09,DMM09,montanari2024equivalence}, and low-degree polynomials~\cite{hopkins2018statistical,KuniskyWeinBandeira,SchrammWein22,SurveyWein2025}.

Among these, low-degree polynomial (LD) bounds have gained attention for establishing state-of-the-art lower bounds in tasks like community detection~\cite{Hopkins17},  spiked tensor models~\cite{Hopkins17,KuniskyWeinBandeira}, and sparse PCA~\cite{ding2024subexponential}. In this model, only estimators that are multivariate polynomials of degree at most $D$ are considered. The premise is that for many problems, degree $D = O(\log n)$ polynomials are as powerful as any polynomial-time algorithm~\cite{KuniskyWeinBandeira}. Thus, failure of such polynomials suggests no efficient algorithm exists for the task. We refer to~\cite{SurveyWein2025} for a comprehensive discussion of this conjecture.
The LD framework connects to other models like statistical queries~\cite{brennan2020statistical}, approximate message passing~\cite{montanari2024equivalence}, and free-energy landscapes~\cite{bandeira2022franz}. 
It was originally developed for detection problems, but this perspective fails to characterize the complexity of estimation problems when detection-recovery gaps exist -- see e.g.~\cite{mao2023detection}. 
To overcome this limitation, the framework has been extended to  estimation tasks, notably in the work of Schramm and Wein~\cite{SchrammWein22}. Their framework has since been applied to submatrix estimation~\cite{SchrammWein22}, stochastic block models~\cite{luo2023computational, SohnWein25}, Gaussian mixture models~\cite{Even24, Even25a}, dense cycles recovery~\cite{mao2023detection}, and graph coloring~\cite{kothari2023planted}.

Let $Y \in \mathbb{R}^{n \times p}$ be the observed matrix, assumed to satisfy the model $Y = X + E$, where $E$ is independent of $X$ and has entries drawn i.i.d. from $\mathcal{N}(0, \sigma^2)$. We consider a  recovery problem consisting in optimally predicting a real-valued random variable $x$, which is $\sigma(X)$-measurable, using degree-$D$ polynomials in $Y$. This is formalized via the minimum mean squared error
\begin{equation}\label{eq:MMSE}
\mathrm{MMSE}_{\leq D} := \inf_{f\in \R_{D}[Y]} \E\cro{\pa{f(Y)-x}^2},
\end{equation}
where $\R_{D}[Y]$ denotes the space of real-valued multivariate polynomials in $Y$ of degree at most $D$.

A key result of Schramm and Wein~\cite{SchrammWein22} is to lower-bound the $\mathrm{MMSE}_{\leq D}$ in terms  of multivariate cumulants
\begin{equation}\label{eq:SW22}
\mathrm{MMSE}_{\leq D} \geq \E[x^2]- \sum_{\alpha\in\N^{n \times p},\ |\alpha|\leq D}\frac{\kappa_{x,\alpha}^{2}}{\alpha!}\enspace,
 \end{equation}
 where $\alpha!=\prod_{(i,j)\in [n] \times [p] }\alpha_{ij}!$, and  for each multi-index $\alpha\in \N^{n \times p}$, the cumulant $\kappa_{x,\alpha}$ is defined as 
\begin{equation}\label{eq:def:kappa}
    \kappa_{x,\alpha}:=\cumul\pa{x, X_{\alpha}}=\cumul\pa{x, \ac{X_{ij}}_{(i,j)\in \alpha}}\enspace,
\end{equation}
 with  $X_{\alpha}=\ac{X_{ij}}_{(ij)\in \alpha}$ denoting the multiset containing $\alpha_{ij}$ copies of each $X_{ij}$ for $(i,j)\in [n]  \times [p]$. Definitions and properties of cumulants are provided in Appendix~\ref{sec:cumulant:background}.

 In certain latent variable models, the signal matrix $X$ depends on some latent variables $\pi^* \in \Pi\subset [K]^n$, and  $x = x(\pi^*)$ is $\sigma(\pi^*)$-measurable. Theorem 2.5 in~\cite{Even25a} shows that bounding $\kappa_{x,\alpha}$ can be reduced to bounding cumulants of the form
 \begin{equation}\label{eq:generalformcumulant}
\cumul\pa{\pa{\delta\pa{\pi^*(i_s), \pi^*(i'_s)}}_{s\in [0,l]}}\enspace,
\end{equation}
for some function $\delta:[K]^2\to \ac{0,1}$.  
When the latent variables $\pi^*(1),\ldots,\pi^*(n)$ are independent, such cumulants can be easily upper-bounded 
 using this independence~\cite{Even25a}. 
 However, the proof techniques of~\cite{Even25a} and~\cite{SchrammWein22} break down under weak dependence, such as when $\pi^*$ is a uniform random permutation of $[n]$, or is generated by sampling without replacement mechanisms.
 Nonetheless, this limitation can potentially be overcome, since multivariate cumulants play a central role in probability theory for deriving central limit theorems (CLTs) in weakly dependent settings -- such as Markov chains, random permutations, and models in statistical physics.
 A key framework for proving such CLTs is the weighted dependency graph theory developed by F\'eray~\cite{feray2018}.
\medskip

\noindent {\bf Our main contributions}\smallskip

1- We develop a quantitative and non-asymptotic extension of F\'eray's theory to bound cumulants of weakly dependent variables. We establish a new upper-bound on the cumulants of certain classes of weakly dependent random variables.
Then, we explain how to derive from this result some bounds on the cumulant $\kappa_{x,\alpha}$ in scenarios where the latent variable $\pi^*$  exhibits weak dependences.\smallskip

2- We instantiate our methodology to derive LD bounds for three problems: multiple feature matching, seriation and clustering with prescribed group sizes. Our LD bounds are tight up to poly-logarithmic factors, and match the performance of state-of-the-art polynomial-time algorithms for these problems, except in a few narrow parameter regimes.\smallskip

3- We derive the corresponding information-theoretic limits, revealing a statistical-computational gap for these three problems.
\smallskip

\noindent Before delving into the technical details, we summarize our key results for these problems.
\medskip

\noindent{\bf Multiple feature matching.}
Feature matching problems arise in many fields, ranging from image analysis \cite{lowe2004distinctive}, to computational biology~\cite{singh2008global}.  
We observe $M$ datasets, each of them gathering (noisy) observations of $K$ feature vectors randomly permuted. 
In the specific case where $M=2$, a probabilistic formulation of this model was introduced by Collier and Dalalyan~\cite{Collier16}. We extend this model as follows. The dataset indexed by $m\in[M]$ stores $K$ feature vectors $Y^{(m)}_{1},\ldots,Y^{(m)}_{K}\in\R^p$. 
For some unknown random permutations 
  $\pi^*_1,\ldots,\pi^*_m$ on $[K]$, and unknown means $\mu_{1},\ldots,\mu_{K}\in\R^{p}$,   the $p$-dimensional feature vectors $Y^{(m)}_{k}$ are assumed to be generated as follow
  \begin{equation}\label{eq:MFM}
  Y_1^{(m)}=\mu_{\pi^*_{m}(1)}+\eps_{1}^{(m)},\ \ldots \ ,Y_K^{(m)}=\mu_{\pi^*_{m}(K)}+\eps_{K}^{(m)}\, ,\quad \text{for $m\in [M]$,} 
  \end{equation}
with $\eps_{k}^{(m)}\stackrel{\text{i.i.d.}}{\sim}\mathcal{N}\pa{0, \sigma^2 I_p}$. 
The goal is to recover the permutations $\pi^*_1,\ldots,\pi^*_m$, i.e.\ to match the features across the datasets. This problem is closely related to the problems of database alignement~\cite{dai2019database,dai2023gaussian,tamir2022correlation}, geometric planted matching~\cite{kunisky2022strong}, particle tracking~\cite{chertkov2010inference} and fair clustering~\cite{FairClustering19,FairClustering21}.  
In particular, our LD bounds have implications in these frameworks.
 All those works are restricted to $M=2$, to the notable exceptions of~\cite{chertkov2010inference}, which considers the specific case where the $M$ datasets are snapshots of a time series.
We provide a complete characterization of the problem in terms of the separation
$\Delta^2=\min_{k\neq l}\frac{\|\mu_k-\mu_l\|^2}{2\sigma^2}$, 
identifying both the information-theoretic threshold and a computational barrier within the LD framework.  We also identify when it is beneficial to match the $M$ datasets simultaneously, rather than matching them pairwisely. 
Exact recovery is information-theoretic  possible, with high probability, as soon as $\Delta^2\gtrsim \log(KM)+\sqrt{\frac{p\log(KM)}{M}}$, which shows benefit of the multiplicity of datasets as soon as $p\gtrsim \log(KM)$. We prove a LD lower bound suggesting that the separation needed for recovering $\pi^*$ in polynomial time is $\Delta^2\geq_{\log} 1+\min\pa{\sqrt{p}, \sqrt{\frac{pK}{K}}}$, and so,  there is a benefit of having access to multiple datasets as soon as $M\geq_{\log} K$. This computational barrier is matched, up to logarithmic factors, in a wide range of regimes by applying a clustering procedure on the set $\big\{Y_{k}^{(m)}:(k,m)\in [K]\times [M]\big\}$. All those results provide evidences of a statistical-computational gap for large $K,M$ and for $p\geq_{\log} M/K$.

\medskip

\noindent{\bf Seriation.}  The seriation problem consists in ordering $n$ items from noisy pairwise  observations. Originally introduced for chronological dating of  graves in archeology~\cite{Robinson51,petrie1899sequences}, it has since found a wide range of applications, including genome alignement~\cite{garriga2011banded,bioinfo17}, time synchronization in distributed networks~\cite{Clock-Synchro04,Clock-Synchro06}, and envelope reduction for sparse matrices~\cite{barnard1995spectral}. This problem has  attracted a lot of attention both in the noiseless case~\cite{atkins1998spectral,fogel2013convex} 
and in the noisy case. The optimal statistical rates have been derived for various loss functions in~\cite{flammarion2019optimal,cai2023matrix,berenfeld2024seriation}. However, the best known polynomial-time algorithmes either achieve much slower rates~\cite{cai2023matrix,berenfeld2024seriation}, or require strong additional assumptions on the model~\cite{janssen2022reconstruction,issartel2024minimax}. For this reason, 
statistical-computational gaps have been conjectured e.g. in~\cite{cai2023matrix,berenfeld2024seriation}. In this work, we provide strong evidence of this gap 
by establishing a LD lower bound for the most emblematic seriation problem. 

For some random permutation $\pi^*=(\pi^*(1),\ldots, \pi^*(n))$, the signal matrix $X\in \mathbb{R}^{n\times n}$ is defined by $X_{ij}= \lambda \1\{|\pi^*(i)- \pi^*(j)|\leq \rho\}$ where the positive integer $\rho$ is a scale parameter, and $\lambda>0$ is the signal strength. 
Given $Y= X +E$, where $E$ is made of iid standard Gaussian variables, we focus on the problem of reconstructing the matrix $X$. We establish in Section~\ref{sec:seriation} the low-degree hardness of weak reconstruction in the regime where $\lambda\preccurlyeq_{\log} 1 \wedge  \frac{\sqrt{n}}{\rho}$, thereby matching the performances of polynomial-time procedures available in the literature. As a comparison, weak reconstruction of $X$ is statistically possible in the much wider regime $\lambda \succcurlyeq \log(n)/\sqrt{\rho}$. As a straightforward consequence, we also confirm the existence of statistical-computational gap for related seriation problems--see Appendix~\ref{sec:seriation:complementary}.
Although seriation shares similarities with one-dimensional random geometric graphs~\cite{penrose2003random}, and in particular to planted dense cycles problems~\cite{mao2023detection,mao2024information}, an important difference is that, in our case, $\pi^*$ is a permutation of $[n]$, whereas in related works~\cite{mao2023detection,berenfeld2024seriation}, $\pi^*$ is a random element of $[n]^n$, so that the $\pi^*(i)$ are independent. The latter property  considerably simplifies the arguments as explained in Section~\ref{sec:globalmethod}.

\medskip

\noindent{\bf Clustering with prescribed group sizes.} Gaussian mixture clustering is a central problem in unsupervised learning~\cite{Dasgupta99,VEMPALA2004,LuZhou2016,Regev2017,giraud2019partial,LiuLi2022,diakonikolasCOLT23b}.  For some unknown vectors $\mu_1,\ldots,\mu_K\in\R^p$, noise level $\sigma>0$, and partition $G^*=\ac{G^*_1,\ldots,G^*_K}$ of $[n]$, the observations $Y_{ij}$ are sampled independently with distribution
$Y_{ij} \sim \mathcal{N}(\mu_{kj},\sigma^2),\quad \text{for}\ i\in G^*_k\quad \text{and}\ j\in[p]$.
When  the clusters  are made of independent assignments, i.e.  $G_{k}^*=\ac{i:k^*_{i}=k}$ with $k^*_{1},\ldots, k^*_{n}$ i.i.d. uniform on $[K]$,
 Even {\it et al.}~\cite{Even25a} provide the computational threshold for poly-time clustering
 in terms of the  separation 
$\Delta^2=\min_{k\neq l}\frac{\|\mu_k-\mu_l\|^2}{2\sigma^2}$,
 except in a narrow parameter regime.
 As a simple example for explaining our theory, we focus here on the case of perfectly balanced clusters, where $G^*$ is sampled uniformly among all partitions fulfilling $|G^*_{k}|=n/K$ for all $k\in[K]$. While being close to the independent assignment case, the perfectly balanced case cannot be handled with the technique of~\cite{Even25a}, due to weak dependencies in the assignment process. Our results prove that the perfect knowledge of the cluster sizes does not help (significantly) for cluster recovery, in the sense that we recover the same  computational threshold as in~\cite{Even25a}.
\medskip

 \noindent{\bf Limitations and perspectives.}
 As in the original work of~\cite{SchrammWein22},  our techniques only allow to characterize hardness regimes up to poly-logarithmic factors. This limitation has been recently tackled by Sohn and Wein~\cite{SohnWein25}, who provide an approach to derive sharp thresholds. However, the latter requires to solve  large overcomplete linear systems, which limits its applicability in complex models. Let us also mention that, for some permutation-based models with no detection-estimation gap -- such as matching correlated 
 Erd\H{o}s-R\'enyi graphs -- one can establish LD bounds by reduction to a detection problem~\cite{ding2023low}.   

Still, we believe that our techniques can be instantiated to provide LD lower bounds in other permutation-based models, for which both statistical-computational, and detection-estimation gaps are conjectured, such as rectangular seriation problems~\cite{flammarion2019optimal}, tensor estimation with unknown permutations~\cite{lee2024statistical}, or constrained clustering~\cite{FairClustering19,FairClustering21,ConstrainedClustering25}.

\paragraph{Notation and organisation of the manuscript.}
Given a vector $v$, we write $\|v\|$ for its Euclidean norm. For a matrix $A$, we denote $\|A\|_F$  for its Frobenius norm, and $\|A\|_{op}$ for its operator norm. For an integer $n$, we write $\cS_n$ as the set of permutation on $[n]$. 
For two functions $u$ and $v$, we write $u\lesssim v$ if there a exists a numerical constant such that $u\leq c v$. For two functions that may depend on some parameters $n,p,K$ or $M$, we write $u\leq_{\log} v$, if there exist numerical constants $c$ and $c'$ such that $u\leq c \log^{c'}(npKM) v$. For a finite subset $S$, we write $|S|$ for its cardinality. 

We identify a matrix $\alpha\in \N^{n\times p}$ (resp. a tensor $\alpha\in \N^{K\times M\times p}$) with the multiset of $[n]\times[p]$ (resp. $[K]\times [M]\times [p]$) containing $\alpha_{ij}$ (resp. $\alpha_{kmj}$) copies of $(i,j)$ (resp. $kmj$). For $i\in [n]$, we write $\alpha_{i:}$ the $i$-th row of $\alpha$. Similarly, for $j\in [p]$, we write $\alpha_{:j}$ the $j$-th column of $\alpha$. We shall write $|\alpha|$ the $l_1$-norm of $\alpha$, which is the cardinality of $\alpha$ viewed as a multiset. Finally, $\alpha!$ stands for $\prod_{ij}\alpha_{ij}!$ (resp. $\prod_{kmj} \alpha_{kmj}$). 
For $W_1, \ldots, W_l$ random variables on the same space, we write $\cumul\pa{W_1,\ldots,W_l}$ for their joint cumulant.

First, we show-case in Section~\ref{sec:multiple_feature} the benefit of our general approach, by characterizing the hardness regimes in multiple feature matching. In Section~\ref{sec:globalmethod}, we explain our general strategy for bounding cumulants of the form~\eqref{eq:generalformcumulant}, and underline the key steps of the approach. We then describe another application of this method in Section~\ref{sec:seriation} by proving a nearly tight LD lower bounds for seriation problems. In Section~\ref{sec:clustering_constant}, we state a tight LD bound  for Gaussian mixture models with group sizes known to be eaxactly $n/K$. Although less significant than the other applications, this result is interesting for pedagogical purpose as this is one of the simplest models that requires our new proof techniques, and its proof is therefore more transparent. All the proofs and related results are postponed to the appendices.

\section{Multiple Feature Matching}\label{sec:multiple_feature}

A first problem that we investigate is the problem of Multiple Feature Matching, where one seeks to estimate multiple permutations from the observation of multiple datasets. Using our general techniques for dealing with weak dependencies -- see Section \ref{sec:globalmethod} -- we characterize a computational barrier for this problem in the LD framework. This barrier is matched in almost all regimes with a clustering procedure. In Appendix~\ref{sec:informationalfeatureMatching}, we also characterize the information-theoretic thresholds for this problem, giving evidence of a statistical-computational gap for Multiple Feature Matching.

\paragraph{Setup:} We observe $M\geq 2$ datasets, and for $m\in [M]$, the dataset indexed by $m$ is made of $Y^{(m)}_{1},\ldots, Y^{(m)}_{K}\in \R^{p}$ generated according to~\eqref{eq:MFM}.
Our aim is to determine the minimum separation $\Delta^2=\min_{k\neq l\in [K]}\frac{\|\mu_k-\mu_l\|^2}{2\sigma^2}$ required for recovering $\pi^*_1,\ldots, \pi^*_M$, up to some global permutation of $[K]$. To quantify the error, we consider the proportion of mismatched points
\begin{equation}\label{eq:errorperm}
    err(\pi,\pi^{*})= \min_{\psi \in\mathcal{S}_{K}}\frac{1}{KM}\sum_{k=1}^{K}\sum_{m=1}^{M}\1_{\psi(\pi_{m}(k))\neq \pi^{*}_{m}(k)}\enspace ,
\end{equation} 
where
$\mathcal{S}_K$ stands for the set of all permutations on $[K]$. The infimum over $\psi \in\mathcal{S}_{K}$ accounts for the fact that $\pi^*$ is identifiable only up to a global permutation $\psi$.
This model is a generalization of the feature matching model considered in~\cite{Collier16} to the case where the number $M$ of datasets  can be larger than $2$. We wonder to what extent the presence of multiple datasets can help to recover the permutations : Is it easier to match $M>2$ datasets than to match two datasets? In principle, the presence of several datasets could be helpful to estimate the unknown $\mu_k$'s, and thereby to improve the accuracy of the matchings. 
We emphasize that our model differs from that of \cite{kunisky2022strong}, where the second dataset is a noisy and permuted version of the first dataset. In the latter setup, having access to multiple datasets would indeed not be helpful for matching two specific datasets.

 In this section, our main contribution is a LD lower-bound for multiple feature matching. Alongside, we also exhibit polynomial-time procedures that nearly achieve this bound in most regimes. Finally, the information-theoretical threshold of the problem is also characterized in Appendix~\ref{sec:informationalfeatureMatching}, so that we provide a full characterization of the problem. To give a glimpse of our results, let us recall that, for feature matching with $M=2$, Collier and Dalalyan~\cite{Collier16} have established  exact recovery of the matching, i.e. $err(\hat{\pi},\pi^*)=0$, as soon as $\Delta^2 \gtrsim \log(K)+ \sqrt{p\log(K)}$. 
 This recovery can be efficiently achieved by an 
 Hungarian algorithm. For multiple feature matching, we prove that exact recovery is statistically possible as soon as $\Delta^2 \gtrsim \log(KM) + \sqrt{p\log(KM)/M}$ -- see Proposition~\ref{prop:upperboundinformationalfeature} and Theorem~\ref{thm:lowerboundinformationalfeature}. This result shows the benefit of permuting simultaneously multiple  samples, when $p$ is large compared to $\log(KM)$. If we restrict our attention to polynomial-time algorithms, we prove in Theorem~\ref{thm:lowdegreefeature} that $\Delta^2\succcurlyeq_{\log} 1+ \max(\sqrt{p}, \sqrt{\frac{pK}{M}})$ is required for low-degree polynomials.
 Corollary~\ref{cor:featurematchingupperbound} shows that successful recovery is possible in polynomial-time above this separation, in most regime. 
 Hence, for computationally efficient procedures, the benefit of multiple permuted samples only arises for $M\geq_{\log} K$, but this benefit is quite strong for large $M$.

\subsection{LD lower-bound for Multiple feature matching}

Applying the general techniques of Section \ref{sec:globalmethod} to the particular case of Multiple Feature Matching, we are able to characterize the computational barrier in the LD framework. Low-degree polynomials are not well suited for directly outputting permutations $\hat{\pi}_1,\ldots, \hat{\pi}_M$. Instead, we focus on the problem of estimating the matrix $\Gamma^*\in \R^{(K\times M) \times (K\times M)}$ defined by $\Gamma^*_{(k,m), (k',m')}=\1\ac{\pi^*_m(k)=\pi^*_{m'}(k')}$. 
Next proposition shows that proving computational hardness for estimating $\Gamma^*$ implies computational hardness for estimating $\pi^*_1,\ldots, \pi^*_M$. We refer to Appendix \ref{prf:featurematchinghardness} for a proof. 

\begin{prop}\label{prop:featurematchinghardness}
Suppose we are in a regime where $K,M$ (and possibly $p$) go to infinity. For any prior on the choice of $\pi^*$ and $\mu_1,\ldots, \mu_K$, if $$MMSE_{poly}:=\inf_{\hat{\Gamma}\ poly-time} {1\over M(M-1)K^2}\E\cro{\|\hat \Gamma-\Gamma^*\|^2_{F}}=K^{-1}(1+o(1))\enspace,$$ then, all polynomial-time estimators $\hat \pi: \R^{K\times M\times p}\to (\mathcal{S}_K)^M$ satisfy $\E\cro{err(\hat{\pi}, \pi^*)}=1+o(1)$. 
\end{prop}

In other words, it is impossible to reconstruct in polynomial time the permutations $\pi^*_1,\ldots, \pi^*_M$ better than random guessing in the regime where $MMSE_{poly}=K^{-1}(1+o(1))$. Thus, we focus in the following on the problem of estimating $\Gamma^*$. For proving a LD lower-bound, we 
consider a Bayesian setting with the following prior distribution.

\begin{definition}\label{def:priorfeaturematching}
We draw $\pi^*_1,\ldots,\pi^*_m$ uniformly and independently over $\mathcal{S}_{K}$. For a given $\bar{\Delta}>0$, independently from $(\pi_m)_{m\in[M]}$,  we draw $\mu_{1},\ldots,\mu_{K}\in\R^{p}$ iid with law  $\mathcal{N}\pa{0,\lambda^2I_p}$, with $\lambda^{2}=\frac{1}{p}\bar{\Delta}^{2}\sigma^2$. Conditionally on $\pi^*_1,\ldots,\pi^*_M$ and $\mu_1,\ldots,\mu_K$, we draw independently, for $k\in [K]$ and $m\in [M]$, $Y_k^{(m)}\sim \mathcal{N}\pa{\mu_{\pi^*_m(k)}, \sigma^2 I_p}$. The observed tensor is $Y\in \R^{K\times M\times p}$ defined by $Y_{k,m,d}=\pa{Y_k^{(m)}}_d$.
\end{definition}

We emphasize that,  with high probability under this prior, the separation satisfies $\Delta^2=\bar{\Delta}^2(1+o_p(1))$, when $p\gtrsim \log(K)$. 
As a proxy for $MMSE_{poly}$, we lower bound in Theorem~\ref{thm:lowdegreefeature} the low-degree mean square error $MMSE_{\leq D}:=\inf_{\hat{\Gamma}\in \R^{D}(Y)} {1\over M(M-1)K^2}\E\cro{\|\hat \Gamma-\Gamma^*\|^2_{F}}$.
The proof of Theorem \ref{thm:lowdegreefeature}, postponed to Appendix~\ref{prf:lowdegreefeature}, builds on the general technique of Section \ref{sec:globalmethod}.
\begin{thm}\label{thm:lowdegreefeature}
Let $D\in \N$ such that $K\geq 2\pa{D+2}^2$, and suppose $\zeta:=64D^{40}\lambda^4p\max\pa{1, \frac{M}{K}}<1$. Then, $$MMSE_{\leq D}\geq \frac{1}{K}-\frac{1}{K^2}\pa{1+64(D+1)D^{36}\frac{\zeta}{1-\sqrt{\zeta}}}\enspace.$$ 
\end{thm}

Taking $D=\pa{\log(n)}^{1+\eta}$, whenever $K\geq 2\pa{\log(n)^{1+\eta}+2}^2$, we deduce from Theorem \ref{thm:lowdegreefeature}, $\pa{\log(n)}^{1+\eta}$-degree hardness when $\bar{\Delta}^2\leq_{\log} \min(\sqrt{p}, \sqrt{\frac{pK}{M}})$. 
Combining this result with the information-theoretic threshold (Theorem~\ref{thm:lowerboundinformationalfeature}) for partial recovery, we get hardness of Multiple Feature Matching when 
\begin{equation}\label{eq:LD-MFM}
\Delta^2\leq_{\log}1+\min\pa{\sqrt{p}, \sqrt{\frac{pK}{M}}}\enspace.
\end{equation}

\subsection{Matching the LD lower-bound with clustering procedures}

 Let $n=KM$, and let us identify $(k,m)\in [K]\times [M]$ to the integer $q=K(m-1)+k$. Define $G^*$ the partition of $[n]$ induced by $\pi^*_1,\ldots, \pi^*_M$ with $G_k^*=\{\pa{\pi^*_m}^{-1}(k), \enspace m\in [M]\}$. There is a one-to-one correspondance between 
 $G^*$ and the partitions $\pi^*_1,\ldots, \pi^*_M$, up to a global permutation of $[K]$. 
 This correspondance suggests the following procedure to estimate $\pi^*$: first apply a clustering procedure $\hat G$ on the $n$ data points $\big\{Y_{k}^{(m)}:(k,m)\in [K]\times [M]\big\}$, and then construct $\hat\pi$ from $\hat G$. If the clustering procedure $\hat{G}$ outputs a partition that does not represent a $M$-tuple of permutations, we can  round it greedily  to output a $M$-tuple of permutations.

When the partition $\hat G$ recovers exactly $G^*$, the estimator $\hat \pi$  recovers exactly  the $M$-tuple of permutations $\pi^*=\pi^*_M,\ldots,\pi^*_M$, up to a global permutation.  
 Combining  Proposition 3.2 of \cite{Even25a}, Theorem 1 of \cite{giraud2019partial}, and Proposition 4 of \cite{Even24}, which provide guarantees for exact recovery in clustering, we obtain the following result, proved in Appendix~\ref{prf:featurematchingupperbound}.

\begin{cor}\label{cor:featurematchingupperbound}
There exists a numerical constant $c>0$ such that the following holds whenever $M\notin [K, K^c]$ or $p\notin [\log(KM), \frac{p}{KM}]$. Suppose that $\Delta^2\geq_{\log} 1+\min\pa{\sqrt{p}, \sqrt{\frac{pK}{M}}}$, then, there exists an estimator $\hat{\pi}=\hat{\pi}_1, \ldots, \hat{\pi}_M$ such that, with high probability, $err(\hat{\pi}, \pi^*)=0$.
\end{cor}

Corollary~\ref{cor:featurematchingupperbound} shows that feature matching is possible in polynomial time above the low-degree threshold \eqref{eq:LD-MFM} in almost all parameter regimes. In the remaining regime where both $M\in [K,K^c]$ and $p\in [\log(KM),\frac{p}{KM}]$, we have the next result, that we suspect to be suboptimal. 

\begin{rem}
When both $M\in [K,K^c]$ and $p\in [\log(KM),\frac{p}{KM}]$,  
we can derive from  Proposition G.7 of \cite{Even25a} perfect recovery of $\pi^*$, with high probability, whenever $\Delta^2\geq_{\log}\sqrt{K}+\sqrt{\frac{pK}{M}}$. 
\end{rem}

\section{Bounding cumulants of weakly dependent random variables} \label{sec:globalmethod}

In the class of models we are considering, to establish LD lower bounds --
either directly, or after applying Theorem 2.5 of \cite{Even25a} -- 
the problem reduces to upper-bounding cumulants of the form $\cumul\pa{\pa{\delta\pa{\pi^*(i_s), \pi^*(i'_s)}}_{s\in [0,l]}}$, where $\delta:[K]^2\to \ac{0,1}$, and $\pi^*(1),\pi^*(2),\ldots$ are weakly dependent random variables. The techniques of \cite{SchrammWein22, luo2023computational, Even25a} break down in such a setting. 
We describe in this section a methodology for upper-bounding $\cumul\pa{\pa{\delta\pa{\pi^*(i_s), \pi^*(i'_s)}}_{s\in [0,l]}}$ in such cases. 
\smallskip

\noindent{\bf Ex: Multiple feature matching.} The labels $\pi^*(k,m)$, for $(k,m)\in [K]\times [M]$ are defined by $\pi^*(k,m)=\pi^*_m(k)$ with $\pi^*_1,\ldots, \pi^*_M$ being permutations drawn independently and uniformly over $\cS_{K}$. After applying Theorem 2.5 of \cite{Even25a}, we are reduced to upper-bounding cumulants of the form $\cumul\pa{\pa{\pi^*(k_s,m_s)= \pi^*(k'_s,m'_s)}_{s\in [0,l]}}$, see Proposition~\ref{thm:LTC}, p.\pageref{thm:LTC}. We use the techniques of this section for dealing with the fact that, for $k,k',m$, $\pi^*(k,m)$ and $\pi^*(k',m)$ are not independent.

\subsection{Core bound on cumulants}
The core result on which relies our analysis is an upper-bound on the cumulants of a class of weakly dependent variables.
Let us consider three families of non-random disjoint (within a family) subsets of $[n]$, denoted  $I_{1},\ldots,I_{\ell}$, $A_{1},\ldots,A_{q}$, and $B_{1},\ldots,B_{r}$. For $\Delta \subset [\ell]$, we define $I_{\Delta}=\cup_{t\in\Delta} I_{t}$. The next lemma provides an upper-bound on the joint cumulant of random variables $Z_{1},\ldots,Z_{\ell}$ with mixed moments given by
\begin{align}\label{eq:formula:general_momodent}
\E\cro{\prod_{t\in \Delta}Z_{t}}=&\eta^{|I_\Delta|}\prod_{j=1}^q \frac{1}{1-x_0}\times\ldots\times \frac{1}{1-(|I_\Delta\cap A_{j}|-1)x_0}\\
&\times\prod_{i=1}^r (1-y_0)\times\ldots\times \pa{1-\pa{|I_\Delta \cap B_{i}|-1}y_0}\ ,\enspace \text{ for any $\Delta\subset [\ell]$.} \nonumber
\end{align}
Such random variables appear when  sampling without replacement and/or random permutations occurs, see Equations \eqref{eq:leveragingclustering}, \eqref{eq:leveragingfeature} and \eqref{eq:leveragingseriation}.

\begin{lem}\label{lem:upperboundgeneralcumulant}
Assume that $Z_{1},\ldots,Z_{\ell}$ have mixed moments given by \eqref{eq:formula:general_momodent} and let $L=|I_{[\ell]}|$. \vspace{-0.1cm}
\begin{enumerate}
\item When $r\geq 1$, $2L^2 y_0\leq 1$, and $2x_0\leq y_0$, we have
\begin{equation}\label{eq:core-bound}
\left|\cumul\pa{Z_{1},\ldots,Z_{\ell}}\right|\leq 4\ell^{2\ell}\eta^{L}\pa{L^2y_0}^{\ell-1}\pa{\frac{x_0}{y_0}}^{cc(\cB)-1}\enspace,
\end{equation}
 where $cc(\cB)$ is the number of connected components of the graph  $\cB$ with node set $[\ell]$ and with edges between $t$ and $t'$  if and only if there exists $i \in [r]$ such that both $I_{t}$ and $I_{t'}$ intersect $B_{i}$.
\item When $r=0$ and $2L^2x_{0}\leq 1$, we have
\begin{equation}\label{eq:upperboundgeneralcumulant2}
\left|\cumul\pa{Z_{1},\ldots,Z_{\ell}}\right|\leq 2\ell^{2\ell}\eta^{L}\pa{L^2x_0}^{\ell-1}.
\end{equation}
\end{enumerate}
\end{lem}

Lemma~\ref{lem:upperboundgeneralcumulant} is proved in Appendix~\ref{sec:quazifactorizationseries}. 
The proof builts on ideas from Féray~\cite{feray2018}, which are revisited in order to provide quantitative and non-asymptotic controls of the cumulants.

Random variables $Z_{1},\ldots, Z_{\ell}$ having mixed moments \eqref{eq:formula:general_momodent} are weakly correlated when $x_{0}$ and $y_{0}$ are small. When $x_{0}$ and $y_{0}$ go to zero, they become uncorrelated, and their joint cumulant $\cumul\pa{Z_{1},\ldots,Z_{\ell}}$ is zero, see Lemma~\ref{lem:independentcumulant}. Lemma~\ref{lem:upperboundgeneralcumulant} shows that this cumulant remains small when $x_{0}$ and $y_{0}$ are positive but small enough. 
While the random variables $\pa{\delta\pa{\pi^*(i_s), \pi^*(i'_s)}}_{s\in [0,l]}$ in~\eqref{eq:generalformcumulant}
 do not fulfill the mixed moment condition~\eqref{eq:formula:general_momodent} in the models that we are considering, for some fixed $\pi$ and some $I_{1},\ldots,I_{\ell} \subset [n]$, the random variables $Z_{t}=\prod_{i\in I_{t}} \1\ac{\pi^*(i)=\pi(i)}$ do fulfill~\eqref{eq:formula:general_momodent}, as highlighted by Equations \eqref{eq:leveragingclustering}, \eqref{eq:leveragingfeature} and \eqref{eq:leveragingseriation}. For example, if $\pi^*$ is drawn uniformly over all permutations on $[n]$, then $\E\cro{\prod_{t\in \Delta}Z_t}=\pa{\frac{1}{n}}^{|I_{\Delta}|}\pa{1-\frac{1}{n}}^{-1}\times \ldots\times \pa{1-{|I_{\Delta}|-1\over n}}^{-1}$,  for all $\Delta\subseteq [\ell]$.
 We explain below how we can reduce the problem of controlling the cumulant~\eqref{eq:generalformcumulant} to the problem of controlling cumulant of random variables like $Z_{t}$. This reduction is made in 4 steps.

\subsection{From cumulant \eqref{eq:generalformcumulant} to cumulants of variables fulfilling \eqref{eq:formula:general_momodent}}

\paragraph{Step 1: Reduction.}
We observe that the random variables $Z_{1},Z_{2},\ldots$ in~\eqref{eq:formula:general_momodent} are weakly correlated, while the  variables 
$\pa{\delta\pa{\pi^*(i_s), \pi^*(i'_s)}}_{s\in [0,l]}$ can have strong dependencies (for example if $i_s=i_{s'}$ for some $s\neq s'$). The first step of our analysis is to reduce the control of cumulants like \eqref{eq:generalformcumulant}, to the control of cumulants of weakly dependent variables, by grouping together variables presenting (possibly) strong dependencies. This reduction is based on Proposition 5.2 of \cite{feray2018}, recalled in Proposition~\ref{prop:feray} page \pageref{prop:feray}. 

To perform the reduction, we introduce 
 a weighted graph  $\cW$ with vertex set $[0,l]$, and with edges $w_{s,s'}$ defined by\smallskip
 
- $w_{s,s'}=1$, if $\ac{i_s,i'_s}\cap \ac{i_{s'}, i'_{s'}}\neq 0$,

- $w_{s,s'}=w\in (0,1)$, otherwise, with $w$ to be chosen later on.\smallskip \\*
In this graph, we have assigned edge weight $w_{s,s'}=1$, when there exists some strong dependences within the variables  $\pi^*\pa{i_s},\pi^*\pa{i'_s},\pi^*\pa{i_{s'}},\pi^*\pa{i'_{s'}}$. 

Let $S^{(0)}\subseteq [0,l]$. Write $\cW[S^{(0)}]$ for the weighted graph reduced to $S^{(0)}$, and write $\cW\<1\>[S^{(0)}]$ for the graph obtained by only keeping the edges of $\cW[S^{(0)}]$ of weight $1$.  We also denote by $\bS$ the set of connected components of $\cW\<1\>[S^{(0)}]$.
To upper-bound the cumulant \eqref{eq:generalformcumulant}, Proposition~\ref{prop:feray} ensures 
that it is enough to upper bound 
\begin{equation}\label{eq:reducedcumulantgeneral}
C_{\mathbf{S}}:=\cumul\pa{\pa{\prod_{s\in S}\delta\pa{\pi^*\pa{i_s}, \pi^*\pa{i'_s}}}_{S\in\mathbf{S}}}\enspace,
\end{equation}
in terms of some super-multiplicative function $\psi(S^{(0)})$ times the weight 
$$\bbM\pa{\cW[S^{(0)}]} := \max_{\cT\in \text{span-tree}(\cW[S^{(0)}])}\quad \prod_{e\in \cT}w_e\ ,$$
where the maximum runs over all spanning trees $\cT$ of $\cW[S^{(0)}]$, and where  the product is over all edges $e$ of $\cT$.
The benefit of  this first step is to reduce the control of the cumulant  \eqref{eq:generalformcumulant} to the control of the cumulant \eqref{eq:reducedcumulantgeneral}, where strongly dependent variables (linked by weight one) are grouped together.
The next lemma expressed the weight $\bbM(\cW[S^{(0)}])$ in terms of
 $w$ and of the number of connected components of $\cW\<1\>[S^{(0)}]$.   
\begin{lem}\label{lem:weigthgraphreduction}
For all $S^{(0)}\subseteq [0,l]$, we have $\displaystyle{\bbM\pa{\cW[S^{(0)}]}=w^{cc(\cW\<1\>[S^{(0)}])-1}=w^{|\mathbf{S}|-1}\enspace.}$
\end{lem}

The three next steps provide a path to upper bound $|C_{\mathbf{S}}|$ in terms of a super-multiplicative function $\psi(S^{(0)})$ times $w^{|\mathbf{S}|-1}$.

\paragraph{Step 2: Expansion.}

While the variables $\pa{\prod_{s\in S}\delta\pa{\pi^*\pa{i_s}, \pi^*\pa{i'_s}}}_{S\in\mathbf{S}}$ appearing in \eqref{eq:reducedcumulantgeneral} are  weakly dependent, their mixed moments are not of the form~\eqref{eq:formula:general_momodent}.  Yet, for a fixed $\pi\in\Pi$, writing $supp(S)=\cup_{s\in S}\ac{i_s, i'_s}$, the random variables $\pa{\prod_{i\in supp(S)}\1\ac{\pi^*(i)=\pi(i)}}_{S\in\mathbf{S}}$ typically have a distribution of the form~\eqref{eq:formula:general_momodent}, in the models  we are considering. In this second step, we leverage the multilinearity of cumulants in order to expand $C_{\mathbf{S}}$ as a sum of cumulants of such random variables.

For $S\in\mathbf{S}$, we have 
\begin{equation}\label{eq:expansion1}
\prod_{s\in S}\delta\pa{\pi^*\pa{i_s}, \pi^*\pa{i'_s}}=\sum_{\pi\in [K]^{supp(S)}}\ \prod_{i\in supp(S)} \1\ac{\pi^*(i)=\pi(i)} \prod_{s\in S} \delta\pa{\pi^*\pa{i_s}, \pi^*\pa{i'_s}}\enspace .
\end{equation}
For any $\pi=\pa{\pi(i)}_{i\in supp(S^{(0)})}\in [K]^{supp(S^{(0)})}$, we say that $\pi$ is active if and only if, for all $s\in S^{(0)}$, $\delta(\pi(i_s), \pi(i'_s))=1$ and if, for all $S\in \mathbf{S}$, the event $\forall i\in supp(S),\enspace \pi^*(i)=\pi(i)$ is of positive probability. 
By definition of $\mathbf{S}$, the subsets $supp(S)$, for $S\in \mathbf{S}$, form a partition of $supp(S^{(0)})$ so $[K]^{supp(S^{(0)})}=\prod_{S\in \mathbf{S}} [K]^{supp(S)}$. Hence from \eqref{eq:expansion1} and the multilinearity of cumulants, the cumulant \eqref{eq:reducedcumulantgeneral} can be expanded as
\begin{equation}\label{eq:pseudoconditionninggeneral}
C_{\mathbf{S}}=\sum_{\pi \enspace active} \underbrace{\cumul\pa{\pa{Z_{\pi,S}}_{S\in \mathbf{S}}}}_{=:C_{\mathbf{S}}(\pi)},\quad \text{where}\quad Z_{\pi,S}=\prod_{i\in supp(S)}\1\ac{\pi^*(i)=\pi(i)}\enspace .
\end{equation}
The cumulants $C_{\mathbf{S}}(\pi)$ involve some variables $Z_{\pi,S}$ having mixed moments like~\eqref{eq:formula:general_momodent} in the models we consider, unless some $Z_{\pi,S}$  are incompatible, in the sense that $\pi$ is not realizable and they cannot be jointly non-zero a.s.
In the latter case an additional pruning step is required before applying Lemma~\ref{lem:upperboundgeneralcumulant}. In the following, we write $\bS=S_1,\ldots, S_{|\bS|}$.

\paragraph{Step 3: Pruning incompatible configurations.}
This step is required when there exists disjoint subsets $U_1,\ldots, U_v$ of $[n]$ such that for $i,i'\in U_{v'}$, $\P\cro{\pi^*(i)=\pi^*(i')}=0$. This situation typically arises when the latent variable $\pi^*$ is related to some permutation.
For example,  this step will be implemented for feature matching and seriation, but not for clustering. 

Incompatible $Z_{\pi,S}$'s have some strong dependencies since they cannot be jointly non-zero. To cope with these dependencies we rely again on Proposition \ref{prop:feray}.
For some active $\pi\in [K]^{supp(S^{(0)})}$, we  define $\cV$ the weighted graph with vertex set $[|\mathbf{S}|]$ and with
\begin{itemize}
\item an edge of weight $1$ between $t,t'\in [|\mathbf{S}|]$ if and only if there exists $i\in supp(S_t)$ and $i'\in supp(S_{t'})$ such that $\pi(i)=\pi(i')$ and  $i,i'\in U_{v'}$ for some $v'\in [v]$,
\item edges of weight $m\in (0,1)$ to be chosen later on, otherwise. 
\end{itemize}
We underline that the weighted graph $\cV$ depends on the choice of $\pi$, and a weight $1$ between $t$ and $t'$ encodes incompatibility of $Z_{\pi,S_{t}}$ and $Z_{\pi,S_{t'}}$. 
Let us denote by $\bR$ the connected components of $\cV\<1\>[R]$. To apply Proposition~\ref{prop:feray}, we must upper bound the cumulant $\cumul\pa{\pa{\prod_{t\in R' }Z_{\pi,S_t}}_{R' \in \bR}}$ in terms of $\bbM\pa{\cV[R]}$.

A key feature of the construction of $\cV$ is that, when $R\subseteq [|\mathbf{S}|]$ is such that $\cV[R]$ has at least an edge of weight $1$, then  we have $\prod_{t\in R' }Z_{\pi,S_t}=0$ a.s. for  all $R' \in \bR$ of size at least $2$. As a consequence 
$\cumul\pa{\big(\prod_{t\in R' }Z_{\pi,S_t}\big)_{R' \in \bR}}=0$,
and Proposition~\ref{prop:feray} implicitely expurgates such incompatible $Z_{\pi,S}$.

When  $\cV[R]$ has no edge of weight $1$, then $\bR$ reduces to singletons, and $\cumul\pa{\pa{\prod_{t\in R' }Z_{\pi,S_t}}_{R' \in \bR}}$ is equal to $\cumul\pa{\pa{Z_{\pi, S_t}}_{t\in R}}$. Furthermore we have a simple formula for the weights of $\cV[R]$.
\begin{lem}\label{lem:weightedgraphimpossibleevents}
For  $R\subseteq [|\mathbf{S}|]$, we have 
$\bbM\pa{\cV[R]}=m^{cc\pa{\cV\<1\>[R]}-1}\enspace,$
where $cc\pa{\cV\<1\>[R]}$ stands for the number of connected components of $\cV\<1\>[R]$.
\end{lem}

Let us define $\bS[R]:=\pa{S_t}_{t\in R}$.
The above results show that, for applying Proposition \ref{prop:feray}, it is sufficient to upper-bound   the cumulants 
\begin{equation}\label{eq:CSR}
\cumul\pa{\pa{Z_{\pi, S_t}}_{t\in R}}=:C_{\bS[R]}(\pi)\enspace,
\end{equation}
in terms of $\bbM\pa{\cV[R]}=m^{|R|-1}$, for all $R\subseteq [|\mathbf{S}|]$ for which $\cV[R]$ does not have edges of weight $1$. Once this is done,  Proposition~\ref{prop:feray} provides a bound on the the cumulant $C_{\mathbf{S}}(\pi)$ in terms of $\bbM\pa{\cV}$.

\paragraph{Step 4: leveraging Lemma \ref{lem:upperboundgeneralcumulant}}
We have now reduced the problem of controlling the cumulants $\kappa_{x,\alpha}$ to the problem of controlling the cumulants \eqref{eq:CSR}
for $R\subseteq [|\mathbf{S}|]$ for which $\cV[R]$ does not have edges of weight $1$. We define $I_{t}=supp(S_{t})$ and $I_\Delta=\cup_{t\in \Delta}supp(S_t)$  for $\Delta \subseteq R$. In the models that we consider, we observe that the variables $\overline{Z}_{t}:=Z_{\pi,S_t}$ have mixed moments of the form \eqref{eq:formula:general_momodent}. Hence, Lemma~\ref{lem:upperboundgeneralcumulant} 
ensures that the cumulants $C_{\bS[R]}(\pi)$ defined in \eqref{eq:CSR} fulfill
\begin{itemize}
\item When $r\geq 1$, $|I_{R}|^2 y_0\leq \frac{1}{2}$ and $x_0\leq \frac{1}{2}y_0$,
$$|C_{\bS[R]}(\pi)|\leq 4|R|^{2|R|}\eta^{|I_{R}|}\pa{|I_{R}|^2y_0}^{|R|-1}\pa{\frac{x_0}{y_0}}^{cc(\cB)-1}.$$ 
\item When $r=0$,  and $|I_{R}|^2x_0\leq \frac{1}{2}$, 
$$|C_{\bS[R]}(\pi)|\leq 2|R|^{2|R|}\eta^{|I_{R}|}\pa{|I_{R}|^2x_0}^{|R|-1}.$$
 \end{itemize}
For example, assume that we are in the case $r=0$ (the case $r\geq 1$ can be treated similarly). Setting $m=|supp(S^{(0)})|^2x_0$, upper-bounding  $|R|^{2|R|}\leq |\bS|^{2|\bS|}$, and noticing that $\psi(R)=\eta^{|I_{R}|}$ is super-multiplicative for $\eta\in(0,1]$,  Proposition~\ref{prop:feray} and  Lemma~\ref{lem:subexponential} ensure that $$C_{\mathbf{S}}(\pi) \leq 2|\bS|^{8|\bS|}\eta^{supp(S^{(0)})}\pa{|supp(S^{(0)})|^2x_0}^{cc\pa{\cV\<1\>}-1}\enspace,$$ where we recall that $cc\pa{\cV\<1\>}$ stands for the number of connected components of the graph induced by the edges of weight $1$ of $\cV$.
Plugging this bound in \eqref{eq:pseudoconditionninggeneral}, we end-up with an upper bound of the form
$$C_{\bS}\leq D_{\infty}\, \tilde\psi(S^{(0)})\, w^{|\bS|-1},$$
for some super-multiplicative function $\tilde \psi$ and some $w\in (0,1)$, both of them dependending on the set of active $\pi$. We can then come back to Step 1, and proceed to the pending application of Proposition~\ref{prop:feray} and  Lemma~\ref{lem:subexponential}, to get a bound on the cumulant of interest
$$\cumul\pa{\pa{\delta\pa{\pi^*(i_s), \pi^*(i'_s)}}_{s\in [0,l]}} \leq 
D_{\infty} (l+1)^{6(l+1)}\, \tilde \psi([0,l])\, w^{cc(\cW\langle 1\rangle)-1}.$$
In Appendix~\ref{app:clustering}, we present  a simple instantiation of the proof technique outlined in this section, applied to the problem of clustering with prescribed group sizes. The more involved proofs for  multiple feature matching  and  seriation are provided in Appendix~\ref{prf:lowdegreefeature}  and~\ref{prf:lowdegreeseriation} respectively.

\section{Seriation}\label{sec:seriation}

The seriation problem amounts to order $n$ objects from pairwise noisy measurements $(Y_{ij})$ of interaction between items $i$ and $j$. The purpose of this section is to highlight a statistical-computational gap  for seriation through the prism of low-degree polynomials. 

For this reason, we focus here on the arguably simplest model, which is defined as follows. Henceforth, $\rho$ is a non-negative integer corresponding to a scale parameter, and $\lambda$ is positive and corresponds  to a signal strength. For some unknown permutation $\pi^*$ of $[n]$, the signal matrix $X$  is defined by
\begin{equation}\label{eq:definition:model:seriation}
X_{ij}=\left\{\begin{array}{cc}
      \lambda& \text{ whenever }|\pi^*(i)- \pi^*(j)|\leq \rho\\      
      0 & \text{ otherwise.}
    \end{array}
      \right.
\end{equation}
Given an observation\footnote{Actually, the diagonal of $Y$ is irrelevant, but we define it for the purpose of keeping unified notation through this manuscript.} of the matrix $Y= X + E$ where the entries of $E$ are i.i.d.\ and follow a centered Gaussian distribution with variance $1$, our objective is recover the underlying permutation $\pi^*$. In this section, we focus on the reconstruction problem, which amounts to estimate the matrix $X$ in Frobenius norm. In Appendix~\ref{sec:seriation:complementary}, we discuss other loss functions, that are pertaining to the estimation of $\pi^*$. 

When  $\pi^*$ is sampled uniformly at random,  for some given $i\neq j$,  the probability to have $|\pi^*(i)-\pi^*(j)|\leq \rho$ is $\varphi:=  \frac{2\rho}{n-1}\pa{1- \frac{\rho+1}{2n}}$. Hence, the best constant estimator in terms of the $L^2$-norm is  $X_{0}=\E\cro{X}$, which is  the matrix whose diagonal is equal to $\lambda$, and whose non-diagonal terms are equal to $\lambda\varphi$.
An estimator $\hat{X}$ is said to achieve weak reconstruction when $\frac{1}{n^2}\mathbb{E}[\|\hat{X}-X\|_F^2]$ is small compared to the risk  $\frac{1}{n^2}\|X- X_0\|_F^2= \frac{n}{n-1}\varphi(1-\varphi)$ of the best constant estimator.

We focus on finding conditions $(\lambda,\rho)$ such that weak-construction is possible in polynomial-time. Our main result is the following LD lower bound.  
Considering a Bayesian setting where $\pi^*$ is sampled uniformly at random among all permutations. For a  degree $D>0$, we define the low-degree $MMSE_{\leq D}$ by 
\begin{equation}\label{def:MMSE:seriation}
	MMSE_{\leq D}:=\inf_{\hat{X}\in\R_{D}[Y]}\frac{1}{n^2}\E\cro{\|\hat{X}-X\|_{F}^{2}}\enspace  . 
\end{equation}
With the proof technique of Section~\ref{sec:globalmethod}, we prove 
the following  LD lower bound, which provides evidence of hardness for recovering $X$ (and thereby the permutation $\pi^*$).
    \begin{thm}\label{thm:lowdegreeseriation}
Let us suppose that $n\geq 2(2D+2)^2$ and $\zeta:=2^{12}\lambda^2\pa{D+1}^{42}\max\pa{1,\frac{4\rho^2}{n}}<1$. Then, 
$$MMSE_{\leq D}\geq \lambda^2\frac{n}{n-1}\left[\varphi(1-\varphi)-2^{18}\frac{\rho^2}{(n-1)^2}
\pa{D+1}^{44}\frac{\max(\zeta,1/n)}{1-\max(\zeta,1/n)}
\right]\enspace.
$$
\end{thm}
We focus on the case where $D\asymp \log(n)$, which, according to the low-degree conjecture, is a strong indication to rule out the performances of polynomial-time algorithms~\cite{KuniskyWeinBandeira}. If $\lambda\preccurlyeq_{\log(n)} [1 \wedge \frac{\sqrt{n}}{\rho}]$, then $MMSE_{\leq D}$ is close to the risk  $MMSE_{\leq 0}$ of $X_{0}$, and it is therefore impossible to achieve weak reconstruction with low-degree polynomials.

Whereas statistical-computational gaps in seriation have been conjectured e.g. in~\cite{cai2023matrix}, we are only aware of the computational lower bounds of~\cite{berenfeld2024seriation}, where they consider different models, with  $\pi^*$ a uniform element of $[n]^n$ instead of a random permutation. The difficulty in establishing lower bounds for seriation stems from the fact that, when $\pi^*$ is an uniform permutation, $\pi^*(i)$ and $\pi^*(j)$ are not independent.

Conversely, whenever $\lambda\succcurlyeq_{\log(n)} [1 \wedge \sqrt{n}/\rho]$, there exists a polynomial-time estimator of $X$ achieving weak reconstruction.

\begin{prop}\label{prp:upper:polynomial:seriation}
Consider the seriation model~\eqref{eq:definition:model:seriation} with an unknown partition $\pi^*$. \begin{itemize}
    \item  Let $\hat{\pi}$ be the polynomial-time estimator \texttt{PINES} from~\cite{berenfeld2024seriation}. Then, 
\begin{equation}\label{eq:frobenius:risk}
\frac{1}{n^2}\mathbb{E}\left[\|\hat{X}_{\hat{\pi}}-X\|_F^2\right]\leq \frac{5\lambda^2}{n} + \frac{\lambda}{n^{1/2} } \sqrt{\log(n)} \ , 
\end{equation}
    \item Define the thresholding estimator $\hat{X}_{th}$ defined by $(\hat{X}_{th})_{ij}= \lambda \mathbf{1}\{Y_{ij}> \lambda/2\}$. If $\lambda\geq 4\sqrt{\log(n)}$, then 
    \[
\frac{1}{n^2}\mathbb{E}\left[\|\hat{X}_{\hat{\pi}}-X\|_F^2\right]\leq \frac{\lambda^2}{n^2}\ . 
    \]
\end{itemize}
\end{prop}
In particular, when $\lambda\geq c\sqrt{\log(n)}\big[1\wedge \sqrt{n}/\rho]$, then the risk of those estimator is significantly smaller than $2\lambda^2 \rho/n\asymp \varphi(1-\varphi)$, and those estimators therefore achieve weak reconstruction.

In contrast, weak reconstruction of the matrix $X$ is information-theoretically possible in the much wider regime where
 $\lambda\geq \sqrt{\log(n)/\rho}$, as shown by the next proposition. 
\begin{prop}\label{prp:upper:IT:seriation}
Consider  $\hat{\pi} \in \arg\min_{\pi} \|Y-X_{\pi}\|_F^2$ be the least-square estimator. Then, 
\[
\frac{1}{n^2}\mathbb{E}\left[\|X_{\hat{\pi}}-X\|_F^2\right]\leq \frac{8\log(n)}{n} \ . 
\]
As a consequence, $X_{\hat{\pi}}$ achieves weak reconstruction as long as $\lambda$ is large compared to $\log(n)/\sqrt{\rho}$. 
\end{prop}

In Appendix~\ref{sec:seriation:complementary}, we shall deduce from Theorem~\ref{thm:lowdegreeseriation} some hardness results for ordering Toeplitz-Robinson matrices, thereby answering an open question in~\cite{berenfeld2024seriation}.

\section{Clustering with constant size groups}\label{sec:clustering_constant}

The problem of clustering a mixture of isotropic Gaussians can be summarized as follows.  For some unknown vectors $\mu_1,\ldots,\mu_K\in\R^p$, noise level $\sigma>0$, and partition $G^*=\ac{G^*_1,\ldots,G^*_K}$ of $[n]$, we observe $Y\in \R^{n\times p}$ with entries  sampled independently with distribution
$$Y_{ij} \sim \mathcal{N}(\mu_{kj},\sigma^2),\quad \text{for}\ i\in G^*_k\quad \text{and}\ j\in[p].$$
The goal is to recover $G^*$.  
This problem  has received a lot of attention~\cite{Dasgupta99,VEMPALA2004,LuZhou2016,Regev2017,giraud2019partial,LiuLi2022,diakonikolasCOLT23b,Even24}.
When the partition originates from independent assignments,   i.e.  $G_{k}^*=\ac{i:k^*_{i}=k}$ with $k^*_{1},\ldots, k^*_{n}$ i.i.d. uniform on $[K]$, the
Information-Theoretic and Low-Degree thresholds are now well established~\cite{Even24,Even25a}, except in a narrow parameter regime.

We consider in this section the problem of clustering a  Gaussian mixture with groups $G^*_{k}$ of known size $n/K$. While close to the independent assignment setting, handling such a size constraint raises some  challenges due to the dependences inherited from the constraint.
Our main motivation for considering this problem is  to provide, in Appendix~\ref{app:clustering}, a simple instantiation of the proof technique of Section~\ref{sec:globalmethod}, before delving into  the more involved proofs of Theorems~\ref{thm:lowdegreefeature}  and~\ref{thm:lowdegreeseriation} in Appendices~\ref{prf:lowdegreefeature} and \ref{prf:lowdegreeseriation}.

To prove the LD bound, we consider a Bayesian setting with the following prior.

\begin{definition}\label{def:prior}
Assume that $K$ divides $n$. We draw $\pi^*(1),\ldots, \pi^*(n)$ uniformly among all $\pi \in[K]^n$ satisfying  $|\ac{i\in[1,n],\enspace \pi^*(i)=k}|=\frac{n}{K}$, for all $k\in[1,K]$. For a given $\bar{\Delta}>0$, independently from $(\pi^*(i))_{i\in[1,n]}$,  we draw $\mu_{1},\ldots,\mu_{K}\in\R^{p}$ i.i.d. with   $\mathcal{N}\pa{0,\lambda^2I_p}$ distribution, with $\lambda^{2}=\frac{1}{p}\bar{\Delta}^{2}\sigma^2$. 

Then, conditionally on $(\pi^*(i))_{i=1,\ldots,n}$ and $(\mu_k)_{k=1,\ldots,K}$, the $Y_i$'s are sampled independently with  $\mathcal{N}(\mu_{\pi^*(i)}, \sigma^2I_p)$ distribution. The partition $G^{*}$ is obtained from the $\pi^*(i)$'s with the canonical partitioning $G^{*}_{k}=\{i\in[1,n],\enspace \pi^*(i)=k\}$. 
\end{definition}

With high probability under this prior, the separation $\Delta^2=\min_{k\neq l}\frac{\|\mu_k-\mu_l\|^2}{2\sigma^2}$ satisfies $\Delta^2=\bar{\Delta}^2(1+o_p(1))$, when $p\gtrsim \log(K)$.

The partition $G^*$ and its estimator $\hat{G}$ are combinatorial objects and are not, as such, well suited to low-degree polynomials. As done previously (see e.g.~\cite{banks2018information}), we rather focus  on the problem of estimating the partnership matrix $\Gamma^*$ defined by $\Gamma^*_{ij}=\1\{\pi^*(i)=\pi^*(j) \}$, for $i,j\in [n]$. Given any partition $G=G_1,\ldots, G_k$, we define the partnership matrix $\Gamma^G$ by $\Gamma^G_{ij}=\sum_{k\in [K]}\1\{i,j\in G_k \}$. We know from~\cite{Even24} p.5 that 
\begin{equation}\label{eq:MtoG}
\frac{1}{n(n-1)}\|\Gamma^{G}-\Gamma^{*}\|_F^2 \leq \min_{\pi \in \mathcal{S}_{K}} {1\over n} \sum_{k=1}^K |G^*_{k}\triangle \hat G_{\pi(k)}|=:2\,err(\hat G,G^*)\ ,
\end{equation}
where $\triangle$ represents the symmetric difference,  $\mathcal{S}_{K}$ the permutation group on $[K]$, and where $err(\hat G,G^*)$ is the average proportion of misclassified points in $\hat G$. Hence, estimating $\Gamma^{*}$ in polynomial-time with small square-Frobenius distance is no harder than building a polynomial-time estimator $\hat{G}$ with a small error $err(\hat{G},G^*)$. Next proposition, proved in Appendix~\ref{prf:hardnessclustering}, shows that non trivial estimation of $\Gamma^*$ is as hard as non trivial recovery of $G^*$.

\begin{prop}\label{prop:hardnessclustering}
Suppose we are in a regime where $K,n$ (and possibly $p$) go to infinity, with $n/K=o(1)$. If $$\inf_{\hat{\Gamma}\ poly-time} {1\over n(n-1)}\E\cro{\|\hat \Gamma-\Gamma^*\|^2_{F}}=K^{-1}(1+o(1))\enspace,$$ then, all polynomial-time estimators $\hat G$ with prescribed group size $n/K$ satisfy $\E\cro{err(\hat{G}, G^*)}=1+o(1)$. 
\end{prop}

In other words, it is impossible to reconstruct in polynomial time the partition $G^*$ better than random guessing in the regime where $MMSE_{poly}=\frac{1}{K}(1+o(1))$. As a proxy for $MMSE_{poly}$, we analyse the Low-Degree $MMSE$ defined by 
\begin{equation}\label{def:MMSE}
	MMSE_{\leq D}:=\inf_{\hat{\Gamma}\in\R_{D}[Y]}\frac{1}{n(n-1)}\E\cro{\|\hat{\Gamma}-\Gamma^*\|_{F}^{2}}\enspace.
\end{equation}

\begin{thm}\label{thm:lowdegreeconstantsizeclustering}
We suppose $p\geq \frac{n}{K^2}$, $n\geq \max\pa{2(D+2)^2K, (D+2)^4}$, $K\geq D+2$ and $\zeta:= \frac{\overline{\Delta}^{4}}{p}D^{22}(D+2)^2\max\left(1, \frac{n}{K^2}\right)<1$. Then, under the prior of Definition \ref{def:prior}, we have 
$$MMSE_{\leq D}\geq \frac{1}{K}-\frac{1}{K^2}\pa{1+18D^{21}\frac{\zeta}{1-\sqrt{\zeta}}}\enspace.$$
\end{thm}

Theorem~\ref{thm:lowdegreeconstantsizeclustering} entails that, when the size of the groups is known to be exactly $n/K$, the clustering problem is computationally hard when, 
\begin{equation}\label{eq:computationalclustering}
\Delta^2\leq_{\log} \min\pa{\sqrt{\frac{pK^2}{n}},\sqrt{p}}\enspace.
\end{equation}
This threshold matches, up to possible logarithmic terms, the LD lower bounds established in~\cite{Even25a} for random-size partitions. 
Hence, the perfect knowledge of the cluster sizes does not help (significantly) for cluster recovery. 

Conversely,  most state-of-the-art procedures based on spectral projection~\cite{VEMPALA2004}, hierarchical-clustering~\cite{Even24}, SDP~\cite{giraud2019partial} or high-order tensor projection~\cite{LiuLi2022} are valid for partitions with fixed constant size --see e.g. Proposition 3.2 in~\cite{Even25a}, so that~\eqref{eq:computationalclustering} turns out to be matched by polynomial-time procedures in most regimes.

\subsection*{Acknowledgments}

 This work has been partially supported by ANR-21-CE23-0035 (ASCAI, ANR) and ANR-19-CHIA-0021-01 (BiSCottE, ANR).

\bibliographystyle{alpha}
\bibliography{biblio}

\appendix

\section{Information-theoretic thresholds for multiple feature matching}\label{sec:informationalfeatureMatching}

As discussed in Section~\ref{sec:multiple_feature}, the Multiple Feature Matching problem can be seen as the problem of clustering the  $n=KM$ data points $\big\{Y_{k}^{(m)}:(k,m)\in [K]\times [M]\big\}$ into $K$ groups, with the constraint that,  for each $m\in[M]$, each cluster $G^*_{k}$ contains one, and only one, observation from $\ac{Y_{k}^{(m)}:k\in [K]}$. In this section, we build on this observation for deriving a minimax optimal matching procedure, and we discuss thoroughly the link between clustering and multiple feature matching.

\paragraph{IT thresholds for multiple feature matching.}
A minimax near optimal algorithm for clustering a Gaussian Mixture Model is the Exact $K$-means estimator~\cite{Even24}. We adapt below this estimator to the Multiple Feature Matching problem.  
Given a $M$-tuple of permutations $\pi=(\pi_{m})_{m\in[M]}\in (\mathcal{S}_{K})^{M}$, we define 
\begin{equation}\label{eq:critpi}
    Crit(\pi)=\sum_{k\in[K]}\sum_{m\in[M]}\Big\|Y^{(m)}_{\pi_{m}^{-1}(k)}-\frac{1}{M}\sum_{m\in[M]}Y^{(m)}_{\pi_{m}^{-1}(k)}\Big\|^{2}\enspace .
\end{equation}
This criterion corresponds to the exact $K$-means criterion applied to the partition of $[M]\times [K]$ induced by $\pi$. Our estimator $\hat{\pi}$ will be any minimizer of this criterion over all  $M$-tuples of permutations 
\begin{equation}\label{eq:CKmeans}
\hat{\pi}\in \argmin_{\pi\in (\mathcal{S}_{K})^{M}} Crit(\pi)\enspace.
\end{equation}
The partition of $[K]\times [M]$ induced by $\hat{\pi}$ corresponds to the minimizer of the $K$-means criterion over the partitions that verify  $|G_{k}\cap ([K]\times \{m\})|=1$ for every $m\in[M]$ and $k\in[K]$. Adapting  the analysis of the exact $K$-means Criterion made in \cite{Even24}, we readily derive the following property, the proof of which is sketched in Appendix~\ref{prf:upperboundinformationalfeature}.

\begin{prop}\label{prop:upperboundinformationalfeature}
There exist positive numerical constants $c,c_1,c_2,c_3$ such that the following holds. If $\Delta^2\geq c_1\pa{\log(K)+\sqrt{\frac{p\log(K)}{M}}}$, then, with probability at least $1-\frac{c_2}{n^2}$, the estimator \eqref{eq:CKmeans} fulfills $$err(\hat{\pi}, \pi^*)\leq \exp\pa{-c_3\pa{\Delta^2\wedge \frac{\Delta^4 M}{p}}}\enspace.$$
In particular, if $\Delta^2\geq c\pa{\log(KM)+\sqrt{\frac{p\log(KM)}{M}}}$, then, with the same probability we have $err(\hat{\pi}, \pi^*)=0$. 
\end{prop}

Proposition \ref{prop:upperboundinformationalfeature} implies that partial recovery (i.e. a small $err(\hat{\pi},\pi^*)$) is possible when 
\begin{equation}\label{eq:IT-partial}
\Delta^2\gtrsim \log(K)+\sqrt{p\log(K)\over M}\ ,
\end{equation}
and perfect recovery (i.e. $err(\hat{\pi},\pi^*)=0$) is possible when 
\begin{equation}\label{eq:IT-exact}
\Delta^2\gtrsim \log(Km)+\sqrt{p\log(Km)\over M}\ .
\end{equation}
We now prove that these separation conditions are, up to some multiplicative constant, minimax optimal for multiple feature matching. 

Given $\bar{\Delta}>0$, we denote $\Theta_{\bar{\Delta}}$ the set of $K$-tuple $(\mu_{1},...,\mu_{K})\in (\R^{p})^{K}$ such that 
$$\frac{min_{k\neq l}\|\mu_{k}-\mu_{l}\|^2}{2\sigma^2}\geq \bar{\Delta}^2.$$
 Given a $K$-tuple $\mu=(\mu_{1},...,\mu_{K})\in (\R^{p})^{K}$, and a set of permutations $\pi=\ac{\pi_{m}:m\in[M]}$, we denote $\P_{\mu,\pi}$ the distribution of $Y=\pa{Y_k^{(m)}}_{k,m\in [K]\times [M]}$, for which all the $Y_{k}^{(m)}$'s are drawn independently and  $Y_{k}^{(m)}\sim \cN\pa{\mu_{\pi^*_m(k)}, \sigma^2I_p}$. Next theorem, proved in Appendices~\ref{prf:lowerboundexact} and~\ref{prf:lowerboundpartial}, provides minimax lower-bounds, matching the rates \eqref{eq:IT-partial} and \eqref{eq:IT-exact} derived for the constrained $K$-means estimator~\eqref{eq:CKmeans}.

\begin{thm}\label{thm:lowerboundinformationalfeature}
We have the following minimax lower bounds for the multiple feature matching problem.
\begin{enumerate}       
\item \textbf{Perfect recovery.} There exist positive numerical constants $c$, $C$, and $K_{0}$ such that the following holds. If $\bar{\Delta}^{2}\leq c \pa{\log(KM)+\sqrt{\frac{p}{M}\log(KM)}}$ and $K\geq K_{0}$, $$\inf_{\hat{\pi}:\R^{K\times M\times p}\to (\mathcal{S}_K)^{M}}\sup_{\mu\in\Theta_{\bar{\Delta}}}\sup_{\pi^*\in\pa{\mathcal{S}_K}^{M}}\P_{\mu,\pi^*}\pa{err(\hat{\pi}, \pi^*)\neq 0}\geq C\enspace.$$

\item \textbf{Partial recovery.} There exist positive numerical constants $c$, $c'$, $c''$, $C$, and $K_{0}$ such that the following holds. Suppose that $p\geq c\log(K)$, $K\geq K_0$ and $\bar{\Delta}^{2}\leq c' \pa{\log(K)+\sqrt{\frac{p}{M}\log(K)}}$, then, $$\inf_{\hat{\pi}:\R^{K\times M\times K}\to (\mathcal{S}_K)^{M}}\sup_{\mu\in\Theta_{\bar{\Delta}}}\sup_{\pi^*\in\pa{\mathcal{S}_K}^{M}}\E_{\mu,\pi^*}\pa{err(\hat{\pi}, \pi^*)}\geq C\enspace.$$
\end{enumerate}
\end{thm}

Combining Proposition \ref{prop:upperboundinformationalfeature} and Theorem \ref{thm:lowerboundinformationalfeature}, we conclude that the minimax rate for perfect recovery in multiple feature matching is  \eqref{eq:IT-partial} and \eqref{eq:IT-exact}, up to some multiplicative factor. Comparing these rates to  the LD rate \eqref{eq:LD-MFM} provides evidence for the existence of a statistical-computational gap for the multiple feature matching problem, when $p\geq_{\log}M/K$.

\paragraph{Multiple feature matching versus clustering.}
As mentioned before, multiple feature matching can be seen as the problem of clustering the  $n=KM$ data points $\big\{Y_{k}^{(m)}:(k,m)\in [K]\times [M]\big\}$ into $K$ groups, with the constraint that,  for each $m\in[M]$, each cluster $G^*_{k}$ contains one and only one observation from $\ac{Y_{k}^{(m)}:k\in [K]}$. Given this last constraint, we can expect that a lower separation $\Delta^2$ is required  for successful multiple feature matching, than for clustering $n=KM$ points in dimension $p$.
Yet, we observe that both the IT rates \eqref{eq:IT-partial}--\eqref{eq:IT-exact}, and the  LD rate \eqref{eq:LD-MFM} for multiple feature matching are the same than those for clustering $n=KM$ points in dimension $p$~\cite{Even24,Even25a}.  This means, that, in terms of rates (blind to constant factors), multiple feature matching is not easier than clustering.  Let us give some intuitions for explaining this observation.
 
For the IT threshold, the optimal algorithm for multiple feature matching and for clustering is, in both cases, a minimizer of the same $K$-means 
 criterion, the only difference is that the minimization is performed over $(\mathcal{S}_K)^{M}$ for multiple feature matching, and over the set $\cP_{K}(KM)$ of all partitions of $[KM]$ into $K$ groups for clustering. Since the performance of both minimizers of the $K$-means criterion is mainly driven by the concentration of the same quadratic functionals of $Y$, as a first approximation, only the size of the log-cardinalities 
 $\log\pa{|(\mathcal{S}_K)^{M}|}$  and  $\log\pa{|\cP_{K}(KM)|}$ can differentiate the rates.
 We observe that $\log\pa{|\cP_{K}(KM)|}\leq KM\log(K)$, while  $\log\pa{|(\mathcal{S}_K)^{M}|}\sim KM\log(K)$ for $K$ large, so both cardinalities are similar. While the constraint $|G_{k}\cap ([K]\times \{m\})|=1$ for every $m\in[M]$ and $k\in[K]$ reduces the number of admissible partitions, this reduction is not strong enough in order to significantly differentiate the minimax rates.

As for the LD threshold, we observe that the constraint $|G_{k}\cap ([K]\times \{m\})|=1$ for every $m\in[M]$ and $k\in[K]$ is complex and combinatorial in nature. \cite{Even25a} provides evidence that poly-time algorithms -- at least low-degree polynomials -- can struggle to fully exploit complex combinatorial structures, as for example in sparse clustering or biclustering. In light of this observation, it is not surprising that poly-time algorithms fail to exploit the  constraint $|G_{k}\cap ([K]\times \{m\})|=1$ for every $m\in[M]$ and $k\in[K]$, in order to perform significantly better than unconstrained clustering algorithms.
 
Finally, we mention that multiple feature matching can be viewed as a specific case of socially fair-clustering~\cite{FairClustering19,FairClustering21}. In fair-clustering, we observe $n$ data points $Y_{1},\ldots, Y_{n}$, each of them having an observed sensitive attribute (gender, social group, etc), say $m\in[M]$. One of the goal in fair-clustering is to provide a partition of  $Y_{1},\ldots, Y_{n}$ minimizing a criterion like $K$-means or $K$-median under the constraint that every demographic group is proportionally represented in each cluster~\cite{FairClustering19}.
Considering  $\big\{Y_{k}^{(m)}:(k,m)\in [K]\times [M]\big\}$ as data points to be clustered, and seeing the label $m$ as the sensitive attribute of $Y_{k}^{(m)}$, the multiple feature matching problem can be seen as a fair-clustering problem. Indeed, if we seek for a partition of $\big\{Y_{k}^{(m)}:(k,m)\in [K]\times [M]\big\}$ into $K$ clusters, each of them having an equal number of each label $m$, then, for each cluster, we are forced to peak one, and only one, observation from $\ac{Y_{k}^{(m)}:k\in [K]}$. Hence, a fair clustering of  $\big\{Y_{k}^{(m)}:(k,m)\in [K]\times [M]\big\}$  into $K$ clusters provide a multiple feature matching.

\section{Background on cumulants}\label{sec:cumulant:background}

In this section, we summarize results on mixed  cumulants of random variables that we shall use in our proof.

\subsection{Classical results on cumulants}

Let $Y_1,\ldots, Y_l$ be random variables on the same space $\mathcal{Y}$. Their cumulant generating function is $$K(t_1,\ldots, t_l):=\log\E\cro{\exp\pa{\sum_{l'=1}^lt_{l'}Y_{l'}}}\enspace,$$
and their joint cumulant is the value of the partial derivative of $K(t_1,\ldots, t_l)$ at 0
 $$\cumul\pa{Y_1,\ldots,Y_l}:=\pa{\frac{\partial^l\, K}{\partial t_{1}\cdots\partial t_{l}}(t_1,\ldots,t_l)}_{t_1,\ldots, t_l=0}\enspace.$$
In the proof, we use standard results on cumulants that we state in Lemma \ref{lem:mobiusformula} and \ref{lem:independentcumulant}.  We refer e.g. to \cite{novak2014three} for proofs and more details. The first lemma is the Möbius formula that expresses joint cumulant of random variables as a linear combination of their mixed moments.

\begin{lem}\label{lem:mobiusformula}
Let $Y_1,\ldots, Y_l$ be random variables on the same space $\mathcal{Y}$. For $G$ a partition of $[l]$, we write $m(G)=(-1)^{|G|-1}(|G|-1)!$ the Möbius function. Then $$\cumul\pa{Y_1,\ldots, Y_l}=\sum_{G\in\mathcal{P}([l])}m(G)\prod_{R\in G}\E\cro{\prod_{i\in R}Y_i} \ ,$$
where $\cP([l])$ is the collection of all partitions of $[l]$. 
In particular, cumulants are multilinear. 
\end{lem}

The next lemma gives a sufficient condition for the nullity of cumulants.

\begin{lem}\label{lem:independentcumulant}
Let $Y_1,\ldots, Y_l$ be random variables on the same space $\mathcal{Y}$. Suppose that there exists disjoint sets $I_1$ and $I_2$, non-empty and covering $[l]$, such that $(Y_i)_{i\in I_1}$ and $(Y_i)_{i\in I_2}$ are independent. Then, we have  $\cumul\pa{Y_1,\ldots, Y_l}=0$.
\end{lem}

\subsection{Expansion of cumulants in latent variable  models}\label{sec:law_total_cumulantce}

In some latent models, as  considered in Sections~\ref{sec:clustering_constant} and \ref{sec:multiple_feature},
the cumulants $\kappa_{x,\alpha}$ arising in the control~\eqref{eq:SW22} of the $MMSE_{\leq D}$ have intricate forms.  In this subsection, we recall some formulas derived in~\cite{Even25a}. Assume that the signal matrix $X$ is of the form
\begin{equation}\label{eq:latent-model}
X_{ij}=\mu_{\pi^*(i)j},\quad \textrm{for}\quad (i,j)\in [n]\times [p], 
\end{equation}
where $\pi^*: [n]\mapsto [K]$ is possibly randomly generated and the $\mu_{kj}$'s are  generated independently according to $\mathcal{N}(0,\lambda^2)$ with $\lambda>0$. 
In particular, this encompasses the prior distributions in Definitions~\ref{def:prior} and~\ref{def:priorfeaturematching}.

The following proposition is a corollary of Theorem 2.5 in~\cite{Even25a}. It shows that for bounding $\kappa_{x,\alpha}$, it  suffices to control mixed cumulants in terms of the latent variables $\pi^*(i)$. 
\begin{prop}\label{thm:LTC}[\cite{Even25a}]
Consider the model~\eqref{eq:latent-model}, and fix any $\alpha\in \N^{n\times p}$. If $|\alpha|$ is odd, then $\kappa_{x,\alpha}=0$. For $|\alpha|=2l$, suppose that, for all decomposition $\alpha=\beta_1+\ldots+\beta_l$, with $\beta_s$ of the form $\beta_s=\ac{(i_s, j_s), (i'_s, j_s)}$, we have $$\left|\cumul\left\{x,\1\ac{\pi^{*}(i_1)= \pi^{*}(i_2)},\ldots, \1\ac{\pi^{*}(i_l)= \pi^{*}(i'_l)} \right\}\right|\leq C_\alpha\enspace,$$
then, 
\begin{equation}\label{eq:LTCgeneral:thm}
|\kappa_{x,\alpha}|\leq \lambda^{|\alpha|}|\alpha|^{\frac{|\alpha|}{2}-1}C_\alpha\enspace.
\end{equation}
\end{prop}

\begin{proof}[Proof of Proposition \ref{thm:LTC}]
Let us remark that, in the model \eqref{eq:latent-model},  $(x,X)$ has the same law as $(x,-X)$. Thus, if $|\alpha|$ is odd, we deduce from the multilinearity of cumulants $$\kappa_{x,\alpha}=\cumul\pa{x, \pa{X_{ij}}_{ij\in \alpha}}=\cumul\pa{x, \pa{-X_{ij}}_{ij\in \alpha}}=\pa{-1}^{|\alpha|}\cumul\pa{x, \pa{X_{ij}}_{ij\in \alpha}}=-\kappa_{x,\alpha}\enspace,$$
and we directly have $\kappa_{x,\alpha}=0$. 
Suppose that $|\alpha|=2l$. We recall that we consider $\alpha\in \N^{n\times p}$ as a multiset containing $\alpha_{ij}$ copies of $(i,j)$. 
 Let $\mathcal{P}_2(\alpha)$ denote the set of all partitions of $\alpha$ into groups of size $2$. Given $G=\ac{G_1,\ldots,G_l}\in \cP_2\pa{\alpha}$ and $s\in [l]$, we define $\beta_s(G)\in \N^{n\times p}$ with entries $\pa{\beta_s(G)}_{ij}$ recording the number of copies of $(i,j)$ in $G_s$. Then, Theorem~2.5  in~\cite{Even25a} states that 
$$\kappa_{x,\alpha}=\lambda^{\alpha}\sum_{G\in \cP_2(\alpha)}\cumul\pa{x, \1_{\Omega_{\beta_1(G)}},\ldots, \1_{\Omega_{\beta_l(G)}}}\enspace,$$
where, for $\beta=\ac{(i,j), (i',j')}$, we write $\1_{\Omega_{\beta}}=\1\ac{j=j'}\1\ac{\pi^*(i)=\pi^*(i')}$. So, for $G\in \cP_2(\alpha)$, writing $\beta_s(G)=\ac{(i_s,j_s),(i'_s,j'_s)}$, if there exists $s$ such that $j_s\neq j'_s$, then $$\cumul\pa{x, \1_{\Omega_{\beta_1(G)}},\ldots, \1_{\Omega_{\beta_l(G)}}}=0\enspace,$$ and, if not, the hypothesis of the proposition leads to $$|\cumul\pa{x, \1_{\Omega_{\beta_1(G)}},\ldots, \1_{\Omega_{\beta_l(G)}}}|\leq C_\alpha\enspace.$$
Hence, $|\kappa_{x,\alpha}|\leq |\cP_2(\alpha)|\lambda^{|\alpha|}C_\alpha\leq \lambda^{|\alpha|}|\alpha|^{\frac{|\alpha|}{2}-1}C_\alpha$, where we use $|\cP_2(\alpha)|\leq |\alpha|^{\frac{|\alpha|}{2}-1}$.
\end{proof}

\subsection{Weighted dependency graph and cumulants}\label{sec:weighted-graph}

Our method for bounding cumulants builds upon the theory of weighted dependency graphs from \cite{feray2018}, on which we provide an account here. This section will be useful when we want to reduce random variables into product of random variables, as for example in Steps 1 and 3 of Section~\ref{sec:globalmethod}. For a set $S$, let  $(Y_i)_{i\in S}$ be random variables on the same space $\mathcal{Y}$, that are indexed by $S$.

\begin{definition}[Weighted Graph]
A weighted graph $\cL=(S,E)$ is an undirected graph with vertex set $S$ such that any edge $e$ has a weight $w_e\in (0,1]$.
\end{definition}

\begin{definition}[Weight of a graph]
Let $\cL=(S,E)$ a weighted graph with vertex set $S$, and weights $w_e$ for $e\in E$. Given a spanning tree $\cT$ of $\cL$, its weight is the product of the weights of its edges, that is $weight(\cT)=\prod_{e\in \cT}w_e$. We define $\bbM(\cL)$ as the maximum weight over all  spanning tree of $\cL$. Here, we use the convention that $\bbM[\cL]=0$ if $\cL$ is not connected.
\end{definition}

For $\cL$ a weighted graph, we define $\cL\<1\>$ the graph induced by the edges of weight $1$. Given $S'\subseteq S$, we define $\cL[S]$ the restriction of $\cL$ to the vertex set $S'$.

In the following, we seek to bound cumulants of subfamilies of $(Y_i)_{i\in S}$ in terms of the maximal weight of a  weighted graph, that keeps track of the low-dependencies between the random variables $Y_i$. To do so, we need to consider super-multiplicative functions on subsets of $S$.

\begin{definition}
Consider a non-negative function  $\psi$ mapping any subset of $S$ to $\mathbb{R}_+$. Then, $\psi$ is said to be super-multiplicative if for any  $S_1$,   $S_2\subseteq S$, we have $\psi\pa{S_1\cup S_2}\geq \psi(S_1)\psi(S_2)$.
\end{definition}

We are now in position to state Proposition~5.2 from \cite{feray2018}, that is very useful to reduce the problem of computing cumulants of random variables, to the problem of computing cumulants of products of random variables.

\begin{prop}\label{prop:feray}[Proposition 5.2 of \cite{feray2018}]
Let $\cL$ a weighted graph with vertex set $S$ and $\psi$ a super-multiplicative function on multisets of elements of $S$. Let $(D_r)_{r\geq 1}$ a sequence of positive numbers. Assume that, for any subset $B\subseteq S$, one has 
$$\left|\cumul\pa{\prod_{i\in B_1}Y_i,\ldots, \prod_{i\in B_l}Y_i}\right|\leq D_{|B|}\,\psi(B)\,\bbM\pa{\cL[B]}\enspace,$$
where $B_1,\ldots,B_l$ are the vertex sets of the connected components of $\cL\<1\>[B]$. Then, for all $B\subseteq S$, one has 
$$|\cumul\pa{(Y_i)_{i\in B}}|\leq C_{|B|}\psi(B)\bbM\pa{\cL[B]}\enspace,$$
with $C_r$ a sequence defined recursively by 
$$C_r=D_r+\underset{G\neq \{[r]\}}{\sum_{G\in\mathcal{P}([r])}}\prod_{R\in G}C_{|R|}\enspace ,$$
where  $\cP([r])$ stands for the collection of all partitions of $[r]$. 
\end{prop}

\medskip

\begin{rem}
 The proof of Proposition \ref{prop:feray} in \cite{feray2018} is a recursion on $B$. Thus, if one wants to upper-bound, for a specific subset $B\subseteq S$, the quantity $|\cumul\pa{(Y_i)_{i\in B}}|$, it is sufficient to upper-bound, for all subsets $B'\subseteq B$, the cumulants $|\cumul\pa{\prod_{i\in B'_1}Y_i,\ldots, \prod_{i\in B'_l}Y_i}|$, where  $B'_1,\ldots,B'_l$ are the vertex sets of the connected components of $L\<1\>[B']$. 
\end{rem}

To establish non-asymptotic bounds of cumulants, we provide the following upper bound on the sequence $|C_r|$ defined in Proposition \ref{prop:feray}.

\begin{lem}\label{lem:subexponential}
In Proposition \ref{prop:feray}, if we suppose that the sequence $(D_r)_{r\geq 1}$ is upper-bounded by some constant $D_{\infty}>0$, then, for all $r\geq 1$, $C_r\leq D_{\infty}r^{6r}$.
\end{lem}

\begin{proof}[Proof of Lemma \ref{lem:subexponential}]

Without loss of generality, we can suppose that $D_{\infty}=1$. We shall prove by induction over $r\geq 0$ that $C_r\leq r^{6r}$. By definition, $C_0=1=0^0$.

We  now assume  that, for all $r'<r$, we have $C_{r'}\leq r'^{6r'}$, and we shall prove that $C_r\leq r^{6r}$. We begin by isolating one subset in the recursion formula 
\begin{align*}
C_r&=1+\sum_{\pi\neq \{[r]\}}\prod_{R\in\pi}C_{|R|}
=1+\sum_{\emptyset\subsetneq T\subsetneq [r]}C_{|T|}\pa{\sum_{\pi\in\mathcal{P}([r]\setminus T)}\prod_{R\in\pi}C_{|R|}}
=1+\sum_{\emptyset\subsetneq T\subsetneq [r]}C_{|T|}\pa{2C_{r-|T|}-1}\enspace .
\end{align*} 
Then, we apply the recursion hypothesis, 
\begin{align*}
C_r&\leq 1+\sum_{\emptyset\subsetneq T\subsetneq [r]}|T|^{6|T|}\pa{2(r-|T|)^{6(r-|T|)}-1}
\leq 2\sum_{\emptyset\subsetneq T\subsetneq [r]}|T|^{6|T|}(r-|T|)^{6(r-|T|)}\\
&\leq 2\sum_{l=1}^{r-1}\binom{r}{l}l^{6l}(r-l)^{6(r-l)}\enspace.
\end{align*}
For $l\in[1,r/2]$, we have $\binom{r}{l}l^{6l}(r-l)^{6(r-l)}\leq \binom{r}{l}l^{6l}r^{6(r-l)}\leq r^{6r}\pa{\frac{l}{r}}^{6l}\binom{r}{l}$. Then,  $\log\pa{\pa{\frac{l}{r}}^{6l}\binom{r}{l}}\leq 6l\log(l/r)+l\log(er/l)\leq -5l\log(r/l)+l\leq -5l\log(2)+l\leq -3l\log(2)$, where the last inequality comes from the fact that $\log(2)\geq \frac{1}{2}$. Hence, we arrive at
 \begin{align*}
C_r&\leq 4r^{6r}\sum_{l=1}^{r/2}2^{-3l}\leq \frac{1}{2}r^{6r}\frac{1-2^{-3r/2}}{1-2^{-3}}\leq r^{6r}\enspace , 
\end{align*}
which concludes the induction and the proof of Lemma \ref{lem:subexponential}.
\end{proof}

\section{Proof of Theorem \ref{thm:lowdegreeconstantsizeclustering} -- pedagogical instantiation of Section~\ref{sec:globalmethod}}\label{app:clustering}
To prove Theorem \ref{thm:lowdegreeconstantsizeclustering}, we proceed in three steps:
\begin{enumerate}
\item We apply \cite{SchrammWein22} to reduce the derivation of LD bound to proving bounds on cumulants $\kappa_{x,\alpha}=\cumul\pa{x,X_{\alpha}}$;
\item We instantiate the methodology of Section~\ref{sec:globalmethod} to upper bound  $\kappa_{x,\alpha}$;
\item We conclude following similar reasoning as in \cite{SchrammWein22,Even25a}.
\end{enumerate}

\paragraph{1. From LD bound to cumulants with \cite{SchrammWein22}.}

With no loss of generality, we assume through this proof that $\sigma^2=1$.
Let $D\in \N$ such that $n\geq \max\pa{2(D+2)^2K, \pa{D+2}^4}$ and $K\geq D+2$. We suppose $p\geq \frac{n}{K^2}$ and 
$$\zeta:= \frac{\overline{\Delta}^{4}}{p}D^{22}(D+2)^2\max\left(1, \frac{n}{K^2}\right)<1\enspace.$$
Since the minimization problem defining $MMSE_{\leq D}$ in Equation (\ref{def:MMSE}) is separable, and since the random variables $M^*_{ij}$ are exchangeable, the $MMSE_{\leq D}$ can be reduced to
\begin{align*}
    MMSE_{\leq D}&:=\frac{1}{n(n-1)}\sum_{i\neq j=1}^{n}\inf_{f_{ij}\in \R_{D}(Y)}\E\cro{\pa{f_{ij}(Y)-\Gamma^*_{ij}}^2}\\
&=\inf_{f\in \R_{D}(Y)}\E\cro{\pa{f(Y)-\Gamma^*_{12}}^2}\enspace.
\end{align*}
In the remaining of the proof, we write $x=\Gamma^*_{12}=\1_{\pi^*(1)=\pi^*(2)}$. Then, our goal is to upper-bound 
$$MMSE_{\leq D}=\inf_{f\in \R_{D}(Y)}\E\cro{\pa{f(Y)-x}^2}\enspace.$$
The model of Definition \ref{def:prior} is a particular instance of the Additive Gaussian Noise Model considered in \cite{SchrammWein22}. 
Therefore, we can use Theorem 2.2 from \cite{SchrammWein22}, stated in Equation \eqref{eq:SW22}. For a matrix $\alpha\in\N^{n\times p}$, we define $|\alpha|=\sum_{i=1}^{n}\sum_{j=1}^{p}\alpha_{ij}$ and $\alpha!=\prod_{i=1}^{n}\prod_{j=1}^{p}\alpha_{ij}!$. We write $X\in \R^{n\times p}$ for the signal matrix, whose $i$-th row is the vector $\mu_{\pi^*(i)}$. Using Equation \eqref{eq:SW22}, we have the lower-bound
\begin{equation}
MMSE_{\leq D}\geq \frac{1}{K}-\frac{1}{K^2}-\underset{\alpha\neq 0}{\sum_{\alpha\in \N^{n\times p}}}\frac{\kappa_{x,\alpha}^2}{\alpha!}\enspace,
\end{equation}
where, for $\alpha\in \N^{n\times p}$, $\kappa_{x,\alpha}=\cumul(x,X^{\alpha}):=\cumul(x,X_{a_{1}},\ldots ,X_{a_{m}})$, with $\{a_{1},\ldots ,a_{m}\}$  the multiset that contains $\alpha_{ij}$ copies of $(i,j)$, for $i,j\in[1,n]\times [1,p]$.

Hence, in order to prove the theorem, it is enough to prove the bound
\begin{equation}\label{eq:SW22clustering}
    \underset{\alpha\neq 0}{\sum_{\alpha\in \N^{n\times p}}}\frac{\kappa_{x,\alpha}^2}{\alpha!}\leq \frac{18D^{21}}{K^2}\frac{\zeta}{1-\sqrt{\zeta}}\enspace.
\end{equation}
In light of Equation \eqref{eq:SW22clustering}, we shall identify the zero cumulant $\kappa_{x,\alpha}=0$ and
control $\kappa_{x,\alpha}$ for any $\alpha\in \N^{n\times p}$. For this purpose, it is convenient to represent $\alpha\in \N^{n\times p}$ as a bipartite multi-graph. 
More precisely, we define $\mathcal{G}_{\alpha}$ as the bipartite multi-graph on two disjoint sets of nodes $U=\{u_{i},i\in[1,n]\}$ and $V=\{v_{j}, j\in[1,p]\}$, with $\alpha_{ij}$ edges between $u_i$ and $v_j$, for $i,j\in [1,n]\times [1,p]$. We denote $\mathcal{G}_{\alpha}^-$ the restriction of $\mathcal{G}_{\alpha}$ to its non-isolated nodes and to $\ac{u_1,u_2}$. 
Henceforth, $\mathcal{G}_{\alpha}^-\cup \ac{(u_1,u_2)}$ stands for the multigraph $\cG_{\alpha}^{-}$ to which we have added an edge between $u_1$ and $u_2$. 
We write $cc(\alpha)$ the number of connected components of $\mathcal{G}_{\alpha}^-\cup \ac{(u_1,u_2)}$. We denote $supp(\alpha):=\ac{i\in [n], \alpha_{i:}\neq 0}$, $col(\alpha)=\ac{j\in [p], \alpha_{:j}\neq 0}$.

\paragraph{2. Bounding $\kappa_{x,\alpha}$ with the methodology outlined Section~\ref{sec:globalmethod}.}
The crux for deriving the lower bound of Theorem \ref{thm:lowdegreeconstantsizeclustering} lies in obtaining an upper-bound on the cumulant $\kappa_{x,\alpha}$.
Following our methodology outlined Section~\ref{sec:globalmethod}, we obtain the following upper-bound. 

\begin{lem}\label{lem:controlcumulantsconstantsizeclustering}
Let $\alpha\neq 0\in \N^{n\times p}$. If $n\geq 2\pa{|\alpha|+2}^2K$ and $K\geq |supp(\alpha)\cup\ac{1,2}|$, then, 
$$|\kappa_{x,\alpha}|\leq 4 \lambda^{|\alpha|}|\alpha|^{5|\alpha|+9}\pa{\frac{1}{K}}^{|supp(\alpha)\cup\ac{1,2}|-1}\pa{\frac{K(|supp(\alpha)\cup\ac{1,2}|)^2}{n}}^{cc(\alpha)-1}\enspace.$$
\end{lem}

For clarity and pedagogical purposes, the proof of this result is deferred to Appendix~\ref{prf:controlcumulantsconstantsizeclustering}.

\begin{rem} In the case where $cc(\alpha)=1$ and where $[1,2]\subseteq supp(\alpha)$, we recover, up to some $|\alpha|^{|\alpha|}$ terms, the bound obtained in~\cite{Even25a} for independent labels. 
\end{rem}

\paragraph{3. Concluding similarly as in \cite{SchrammWein22, Even25a}.}
It remains to characterize the set of $\alpha$'s such that $\kappa_{x,\alpha}\neq 0$ and to upper-bound the sum of the remaining cumulants in~\eqref{eq:SW22clustering}. 
\medskip 

\noindent{\bf 3.a. Identifying non zero $\kappa_{x,\alpha}$.} In the case where the labels are sampled independently, as in~\cite{Even25a}, it is possible to build upon the independence of the random variables involved in the cumulant $\kappa_{x,\alpha}$ and thereby to easily deduce that $\kappa_{x,\alpha}=0$ when $cc(\alpha)>1$ -- see~\cite{Even25a}. As no such independence holds in our case, the condition for having $\kappa_{x,\alpha}=0$ is more involved.
Suppose that there exists $j_0\in col(\alpha)$ such that $|\alpha_{:j_0}|=1$. Then, since $\pa{x,(X_{ij})_{ij\in [n]\times [p]}}$ has the same distribution as $\pa{x,((-1)^{j=j_0}X_{ij})_{ij\in [n]\times [p]}}$, we deduce from the multilinearity of cumulants that $\kappa_{x,\alpha}=-\cumul\pa{x, X_\alpha}=-\kappa_{x,\alpha}$, and so $\kappa_{x,\alpha}=0$. Hence, for having $\kappa_{x,\alpha}\neq 0$, we need $|\alpha|\geq 2|col(\alpha)|\geq 2$. Moreover, since $|\alpha|+1$ corresponds to the number of edges of $\cG_\alpha^-\cup\ac{(u_1,u_2)}$ and since $|supp(\alpha)\cup\ac{1,2}|+|col(\alpha)|$ corresponds to the number of nodes in $\cG_{\alpha}^{-}$, we deduce that $|\alpha|\geq |col(\alpha)|+|supp(\alpha)\cup\ac{1,2}|-cc(\alpha)-1$. Putting these two conditions together leads to the following condition on the $\alpha$'s such that $\kappa_{x,\alpha}\neq 0$
$$|\alpha|\geq \max\pa{2|col(\alpha)|, |col(\alpha)|+|supp(\alpha)\cup \ac{1,2}|-cc(\alpha)-1}\enspace.$$
Let us identify two other constraints satisfied by $supp(\alpha)$, $cc(\alpha)$ and $|\alpha|$. 
Each connected component of $\cG_\alpha^-\cup\ac{(u_1,u_2)}$ contains at least one point of $\pa{u_i}_{i\in supp(\alpha)\cup \ac{1,2}}$, with one of them containing both $u_1$ and $u_2$. Hence, $|supp(\alpha)\cup \ac{1,2}|\geq cc(\alpha)+1$. It is also clear that $|\alpha|+2\geq |supp(\alpha)\cup \ac{1,2}|$. 
\medskip

\noindent{\bf 3.b. Bounding the sum~\eqref{eq:SW22clustering}.}
Let $d,r,m,cc\geq 1$, with $d\geq \max\pa{2r,r+m-cc-1}$ and $d+2\geq m\geq cc+1$. There are at most $(rm)^{d}p^r n^{m-2}\leq \pa{d(d+2)}^{d}p^r n^{m-2}$ matrices $\alpha$ satisfying $|\alpha|=d$, $|supp(\alpha)\cup\ac{1,2}|=m$, $|col(\alpha)|=r$ and $cc(\alpha)=cc$. 
Since $|\alpha|\leq D$, $K\geq D+2$, and $n\geq 2(D+2)^2K$
by assumption of Theorem~\ref{thm:lowdegreeconstantsizeclustering},
we can apply Lemma~\ref{lem:controlcumulantsconstantsizeclustering}, and get 

\begin{align*}
\underset{\alpha\neq 0}{\sum_{\alpha\in \N^{n\times p}}}\frac{\kappa_{x,\alpha}^2}{\alpha!}\leq& \sum_{d=2}^{D}\sum_{\alpha: \ |\alpha|=d}\kappa_{x,\alpha}^2\\
\leq& \frac{16}{K^2}\sum_{d=2}^{D}\underset{d\geq \max\pa{2r,r+m-cc-1}}{\underset{d+2\geq m\geq cc+1}{\sum_{r,cc\geq1}}} \lambda^{2d}(d+2)^{d}d^{11d+18}p^r\pa{\frac{n}{K^2}}^{m-2}\pa{\frac{K^2m^4}{n^2}}^{cc-1}\\
\leq &\frac{16D^{18}}{K^2}\sum_{d=2}^{D}\underset{d\geq \max\pa{2r,r+m-cc-1}}{\underset{d+2\geq m\geq cc+1}{\sum_{r,cc\geq1}}} \lambda^{2d}\pa{D^{11}(D+2)}^{d}p^r\pa{\frac{n}{K^2}}^{m-cc-1}\pa{\frac{(D+2)^4}{n}}^{cc-1}\\
\leq &\frac{16D^{18}}{K^2}\sum_{d=2}^{D}\underset{d\geq \max\pa{2r,r+m-cc-1}}{\underset{d+2\geq m\geq cc+1}{\sum_{r,cc\geq1}}} \pa{D^{11}(D+2)\lambda^2}^{d}p^r\pa{\frac{n}{K^2}}^{m-cc-1}\enspace \ , 
\end{align*}
where we used in the last line that $n\geq (D+2)^4$. Let us fix $d,r,m,cc\geq 1$, with $d\geq \max\pa{2r,r+m-cc-1}$ and $d+2\geq m\geq cc+1$, and let us upper-bound the quantity $\Upsilon:= \pa{D^{11}(D+2)\lambda^2}^{d}p^r\pa{\frac{n}{K^2}}^{m-cc-1}$. 

First, let us consider the case $r\geq m-cc-1$. Leveraging the constraint $d\geq 2r$, we derive
\begin{align*}
\Upsilon\leq & \pa{D^{11}(D+2)\lambda^2}^{d-2r}\pa{\lambda^4 D^{22}(D+2)^2p}^{r-(m-cc-1)} \pa{\lambda^4 D^{22}(D+2)^2p\frac{n}{K^2}}^{m-cc-1}
\leq \zeta^{d/2}\enspace ,
\end{align*}
by definition of $\zeta$. 

Second, when $r<m-cc-1$, we use the constraint $d\geq r+m-cc-1$ to obtain
\begin{align*}
\Upsilon \leq &\pa{D^{11}(D+2)\lambda^2}^{d-r-m+cc+1}\pa{D^{11}(D+2)\lambda^2 p}^{r}\pa{D^{11}(D+2)\lambda^2\frac{n}{K^2}}^{m-cc-1}\\
\leq & \pa{D^{11}(D+2)\lambda^2}^{d-r-m+cc+1} \pa{D^{11}(D+2)\lambda^2\frac{n}{K^2}}^{m-cc-1-r}\pa{D^{22}(D+2)^2\lambda^4\frac{pn}{K^2}}^r\enspace.
\end{align*}

It is clear that $D^{11}(D+2)\lambda^2\leq \sqrt{\zeta}$ and $D^{22}(D+2)^2\lambda^4\frac{pn}{K^2}\leq \zeta$. For the remaining term, using $p\geq \frac{n}{K^2}$, we have $D^{11}(D+2)\lambda^2\frac{n}{K^2}=\frac{n}{K^2p}D^{11}(D+2)\bar{\Delta}^2\leq \sqrt{D^{22}(D+2)^2\frac{n}{K^2p}\bar{\Delta}^4}\leq \sqrt{\zeta}$. We end up with $\Upsilon \leq \zeta^{d/2}$.
Hence, we can conclude 
\begin{align*}
\underset{\alpha\neq 0}{\sum_{\alpha\in \N^{n\times p}}}\frac{\kappa_{x,\alpha}^2}{\alpha!}\leq&\frac{16D^{18}}{K^2}\sum_{d=2}^{D}\underset{d\geq \max\pa{2r,r+m-cc-1}}{\underset{d+2\geq m\geq cc+1}{\sum_{r,cc\geq1}}}\zeta^{d/2}\\
\leq &\frac{16D^{18}}{K^2} \sum_{d=2}^{D} \frac{d(d+1)(d+4)}{4}\zeta^{d/2}\\
\leq &\frac{18D^{21}}{K^2}\frac{\zeta}{1-\sqrt{\zeta}}\enspace,
\end{align*}
where we used, $d(d+1)(d+4)\leq 4.5d^3$ for $d\geq 2$, and the identity  $\sum_{d\geq 2}z^d=\frac{z^2}{1-z}$, for any $-1<z<1$. This concludes the proof of Theorem \ref{thm:lowdegreeconstantsizeclustering}.

\subsection{Proof of Lemma \ref{lem:controlcumulantsconstantsizeclustering} -- simple instantiation of Section~\ref{sec:globalmethod}}\label{prf:controlcumulantsconstantsizeclustering}

For proving Lemma \ref{lem:controlcumulantsconstantsizeclustering}, we instantiate the methodology outlined in Section~\ref{sec:globalmethod}.

\paragraph{Step 0: From $\kappa_{x,\alpha}$ to cumulants like \eqref{eq:generalformcumulant}.} 
A preliminary step is to apply  Proposition~\ref{thm:LTC} in order to reduce the problem of bounding $\kappa_{x,\alpha}$ to the problem of bounding cumulants like \eqref{eq:generalformcumulant}. 

Let $\alpha\neq 0\in \N^{n\times p}$. We recall the hypotheses $n\geq 2(|\alpha|+2)^2K$ and $K\geq |supp(\alpha)\cup \ac{1,2}|$. 
According to Proposition~\ref{thm:LTC}, all we need to upper-bound $\kappa_{x,\alpha}$ is to upper-bound,
 for any decomposition $(\beta_1,\ldots, \beta_l)$ of $\alpha$ with $|\beta_s|=2$ of the form $\beta_s=\ac{(i_s, j_s), (i'_s, j_s)}$, the cumulant 
$\cumul\pa{x, \pa{\1\ac{\pi^*(i_s)=\pi^*\pa{i'_s}}}}$.
Let us take the convention $i_0=1$, $i_0'=2$, so that we can write
\begin{equation}\label{eq:start:clustering}
\cumul\pa{x, \pa{\1\ac{\pi^*\pa{i_s}=\pi^*\pa{i'_s}}}}=\cumul\pa{\pa{\1\ac{\pi^*\pa{i_s}=\pi^*\pa{i'_s} }}_{s\in [0,l]}}\enspace, 
\end{equation}
which is of the form~\eqref{eq:generalformcumulant} of the cumulants that are dealt with in Section~\ref{sec:globalmethod}. We can now follow the general methodology described in Section~\ref{sec:globalmethod}.

\paragraph{Step 1: Reduction.} 
As explained in Section~\ref{sec:globalmethod}, the first step is to reduce the control of the cumulant \eqref{eq:start:clustering}, to the control of cumulants of products of strongly dependent $\1\ac{\pi^*\pa{i_s}=\pi^*\pa{i'_s} }$.
We introduce a weighted graph $\cW$ to which we will apply Proposition \ref{prop:feray}. Let $\cW$ be the weighted graph on $[0,l]$ with weights $w_{s,s'}$ defined by;
\begin{itemize}
\item If $\ac{i_s, i'_s}\cap \ac{i_{s'}, i'_{s'}}\neq \emptyset$, then $w_{s,s'}=1$;
\item If $\ac{i_s, i'_s}\cap \ac{i_{s'}, i'_{s'}}=\emptyset$, then $w_{s,s'}=K(|supp(\alpha)\cup \ac{1,2}|)^2/n$.
\end{itemize}
We observe that when $w_{s,s'}=1$, there exists some strong dependencies within the variables  $\pi^*\pa{i_s},\pi^*\pa{i'_s},\pi^*\pa{i_{s'}},\pi^*\pa{i'_{s'}}$. As for the value $w_{s,s'}=K(|supp(\alpha)\cup \ac{1,2}|)^2/n$ for the second case, it is chosen based on forthcoming Lemma~\ref{lem:controlreductedcumulantsclusteringnew}.
Let us fix $S^{(0)}\subseteq [0,l]$. We recall that $\cW\<1\>[S^{(0)}]$ stands for the weighted graph induced by $S^{(0)}$, where we only keep the edges of weight one. Let us write $\bS=S_1,\ldots, S_{|\bS|}$ for the partition of $S^{(0)}$ induced by the connected components of $\cW\<1\>[S^{(0)}]$.   In order to apply Proposition \ref{prop:feray}, we shall upper-bound, for such an $S^{(0)}$, the cumulant 
$$C_{\bS}:=\cumul\pa{\pa{\prod_{s\in S}\1\ac{\pi^*\pa{i_s}=\pi^*\pa{i'_s}}}_{S\in\bS}}\enspace,$$
with respect to the maximal weight $\bbM\pa{\cW[S^{(0)}]}=\pa{\frac{K(|supp(\alpha)\cup\ac{1,2}|)^2}{n}}^{|\bS|-1}$. The next lemma (steps 2-4 of Section~\ref{sec:globalmethod}), proved in Section \ref{prf:controlreductedcumulantsclusteringnew}, provides such an upper-bound.

\begin{lem}\label{lem:controlreductedcumulantsclusteringnew}
Suppose $n\geq 2(|supp(\alpha)\cup \ac{1,2}|)^2K$ and $K\geq |supp(\alpha)\cup\ac{1,2}|$. Let $S^{(0)} \subset [0,l]$ and let $\bS=S_1,\ldots, S_{|\bS|}$ denote the connected components of $\cW\<1\>[S^{(0)}]$. Writing $supp(S^{(0)})=\cup_{s\in S^{(0)}}\ac{i_s,i'_s}$, we have
$$\left|C_{\bS}\right|\leq 4|\bS|^{3|\bS|+1}\pa{\frac{1}{K}}^{|supp(S^{(0)})|-1}\pa{\frac{K|supp(S^{(0)})|^2}{n}}^{|\bS|-1}\enspace.$$
\end{lem}
From Lemma \ref{lem:controlreductedcumulantsclusteringnew}, we deduce that for any $S^{(0)}\subseteq [0,l]$, with $\bS$ the set of connected components of $\cW\<1\>[S^{(0)}]$,
$$\left|C_{\bS}\right|\leq 4|\bS|^{3|\bS|+1}\pa{\frac{1}{K}}^{|supp(S^{(0)})|-1}\bbM\pa{\cW[S^{(0)}]}\enspace.$$
The function $S^{(0)}\to \pa{\frac{1}{K}}^{|supp(S^{(0)})|-1} $ is super-multiplicative. Moreover, $4|\bS|^{3|\bS|+1}\leq 4\pa{l+1}^{3l+4}$. Hence, applying Proposition \ref{prop:feray} and Lemma~\ref{lem:subexponential}, we get that 
$$\left|\cumul\pa{\pa{\1\ac{\pi^*\pa{i_s}=\pi^*\pa{i'_s}}}_{s\in [0,l]}} \right| \leq4\pa{l+1}^{9l+10} \pa{\frac{1}{K}}^{|supp(\alpha)\cup\ac{1,2}|-1}\bbM\pa{\cW}\enspace.$$
In other words, writing $cc(\cW)$ the number of connected components of $\cW\<1\>$, we have 
\begin{multline*}
\lefteqn{\left|\cumul\pa{\pa{\1\ac{\pi^*\pa{i_s}=\pi^*\pa{i'_s}}}_{s\in [0,l]}} \right|\leq}\\
4 \pa{l+1}^{9l+10}\pa{\frac{1}{K}}^{|supp(\alpha)\cup\ac{1,2}|-1}\pa{\frac{K(|supp(\alpha)\cup\ac{1,2}|)^2}{n}}^{cc(\cW)-1}\enspace.
\end{multline*}
We recall that $cc(\alpha)$ stands for the number of connected components of $\mathcal{G}_\alpha^-\cup \ac{(u_1,u_2)}$. For all $s\in [0,l]$, $u_{i_s}$ and $u_{i'_s}$ are in the same connected components of $\mathcal{G}_\alpha^-\cup \ac{(u_1,u_2)}$, since they are connected via $v_{j_{s}}$. Moreover, for $s\neq s'$,  $\{u_{i_s}, u_{i'_s}\}$ and $\{u_{i_{s'}}, u_{i'_{s'}}\}$ are in the same connected components of $\mathcal{G}_\alpha^-\cup \ac{(u_1,u_2)}$ if and only if $s$ and $s'$ are in the same connected component of $\cW\<1\>$. We deduce from these two claims that $cc(\alpha)=cc(\cW)$, and so 
\begin{multline*}
\lefteqn{\left|\cumul\pa{\pa{\1\ac{\pi^*\pa{i_s}=\pi^*\pa{i'_s}}}_{s\in [0,l]}} \right|\leq}\\ 4 \pa{l+1}^{9l+10}\pa{\frac{1}{K}}^{|supp(\alpha)\cup\ac{1,2}|-1}\pa{\frac{K(|supp(\alpha)\cup\ac{1,2}|)^2}{n}}^{cc(\alpha)-1}\enspace.
\end{multline*}
This being true for all decomposition $\beta_1,\ldots,\beta_l$ of $\alpha$,  we get from Proposition~\ref{thm:LTC} that 
$$|\kappa_{x,\alpha}|\leq 4 \lambda^{|\alpha|}|\alpha|^{5|\alpha|+9}\pa{\frac{1}{K}}^{|supp(\alpha)\cup\ac{1,2}|-1}\pa{\frac{K(|supp(\alpha)\cup\ac{1,2}|)^2}{n}}^{cc(\alpha)-1}\enspace.$$
This concludes the proof of Lemma \ref{lem:controlcumulantsconstantsizeclustering}.

\subsection{Proof of Lemma \ref{lem:controlreductedcumulantsclusteringnew}}\label{prf:controlreductedcumulantsclusteringnew}

\paragraph{Step 2: Expansion.} 
Let $S^{(0)}\subseteq [0,l]$ and let $\bS=S_1,\ldots, S_{|\bS|}$ be the connected components of $\cW\<1\>[S^{(0)}]$, which is the subgraph induced by $S^{(0)}$ obtained only keeping the edges of $\cW$ of weight $1$. 
The second step amounts to expand the cumulant 
$$|C_{\bS}|=\left|\cumul\pa{\pa{\prod_{s\in S}\1\ac{\pi^*\pa{i_s}=\pi^*\pa{i'_s}}}_{S\in \bS}}\right|\enspace,$$
over all possible values of $\pi^*$.
For any $S\subseteq [0,l]$, We write $supp(S)=\cup_{s\in S}\ac{i_s,i'_s}$. By definition of the graph $\cW$ and of $\bS$, the subsets $supp(S)$, for $S\in \bS$, are disjoint, and form a partition of $supp(S^{(0)})$. Moreover, for all $S\in \bS$, and for all $i,i'\in supp(S)$, there exists a path $s_1,\ldots, s_{h}\subseteq S$, such that $i\in \ac{i_{s_1}, i'_{s_1}}$, $i'\in \ac{i_{s_h}, i'_{s_h}}$, and such that $(s_{h'-1},s_{h'})$, for $h'\in [2,h]$, is an edge of weight $1$ of $\cW$. Thus, for $S\in\bS$, 
$\pi^*(i_{s})=\pi^*(i'_{s})$ if and only if $\pi^*$ is constant on $supp(S)$, and we have
$$\prod_{s\in S}\1\ac{\pi^*\pa{i_s}=\pi^*\pa{i'_s}}=\sum_{a \in [K]}\prod_{i\in supp(S)}\1\ac{\pi^*(i)= a }\enspace.$$
Hence, we are interested in configurations $\pi^*$ that are constant on each $S_l$ for $l=1,\ldots, |\bS|$. For that purpose, we focus on 
$\pi= (\pi(S_l))_{l=1,\ldots,|\bS|}\in  [K]^{|\bS|}$. Then, using the multilinearity of cumulants, we get
\begin{equation}\label{eq:multilinearityclustering}
C_{\bS}:=\cumul\pa{\pa{\prod_{s\in S}\1\ac{\pi^*\pa{i_s}=\pi^*\pa{i'_s}}}_{S\in \bS}}=\sum_{(\pi(S_l))_{l=1,\ldots,|\bS|}\in  [K]^{|\bS|}}
C_{\bS}(\pi)\enspace , 
\end{equation}
where 
\[
C_{\bS}(\pi):= \cumul\pa{\pa{\prod_{i\in supp(S)}\1\ac{\pi^*(i)=\pi(S)}}_{S\in \bS}}\enspace . 
\]

\paragraph{Step 3: Pruning incompatible configurations.} No pruning required in this case.

\paragraph{Step 4: Leveraging Lemma \ref{lem:upperboundgeneralcumulant}.}
The last step is to recognize some variables having mixed moment  of the form \eqref{eq:formula:general_momodent}.
In the following, for $S\in \bS$, we write $Z_{\pi,S}=\prod_{i\in supp(S)}\1\ac{\pi^*(i)=\pi\pa{S}}$. We denote  $B_1,\ldots, B_K$ the partition\footnote{Actually, this is not a partition as some of the $B_k$'s are empty; we only use that the $B_k$'s do not intersect and cover $supp(S^{(0)})$.} of $supp(S^{(0)})$ induced by $\pi$, i.e $B_k=\cup_{S\in \bS, \pi\pa{S}=k}supp(S)$. We recall that we write $\bS=S_1,\ldots, S_{|\bS|}$. Let $\Delta\subseteq [|\bS|]$ and let us denote $I_{\Delta}=\cup_{t\in \Delta}supp(S_t)$. Since $K\geq |supp(\alpha)\cup\ac{1,2}|$ by assumption, we have
\begin{align}\nonumber
\E\cro{\prod_{t\in \Delta}Z_{\pi,S_t}}=&\P\cro{\forall t\in \Delta, \forall i\in supp(S_t), \pi^*(i)=\pi\pa{S_t}}\\\nonumber
=&\frac{1}{n(n-1)\times \ldots\times (n-|I_\Delta|+1)}\prod_{k\in [K]}\frac{n}{K}\times \ldots\times \pa{\frac{n}{K}-|I_\Delta\cap B_k|+1}\\
=&\pa{\frac{1}{K}}^{|I_\Delta|}\frac{1}{1-\frac{1}{n}}\times \ldots\times \frac{1}{1-\frac{|I_\Delta|-1}{n}}\prod_{k\in [K]}\pa{1-\frac{K}{n}}\times \ldots\times \pa{1-\pa{|I_\Delta\cap B_k|-1}\frac{K}{n}}\enspace.\label{eq:leveragingclustering}
\end{align}
We recognize mixed moments of the form \eqref{eq:formula:general_momodent}.
Applying Lemma \ref{lem:upperboundgeneralcumulant} with $\eta=\frac{1}{K}$, $x_0=\frac{1}{n}$ and $y_0=\frac{K}{n}$, we get that, whenever $|supp(S^{(0)})|^2K/n\leq 0.5$ and $K\geq 2$, 
$$\left|C_{\bS}(\pi)\right|\leq 4|\bS|^{2|\bS|}\pa{\frac{1}{K}}^{|supp(S^{(0)})|}\pa{|supp(S^{(0)})|^2\frac{K}{n}}^{|\bS|-1}\pa{\frac{1}{K}}^{\#\pi -1}\enspace,  $$
where we denote $\#\pi=|\ac{\pi(S), S\in \bS}|$. 

\paragraph{Conclusion.}
Back to \eqref{eq:multilinearityclustering}, to upper-bound $|C_{\bS}|$, it remains to sum
 over all possible $\pi= (\pi(S_l))_{l=1,\ldots,|\bS|}\in  [K]^{|\bS|}$ $\pi\in [K]^{\bS}$
\begin{align*}
\left|C_{\bS}\right|\leq &\sum_{\pi=\pa{\pi(S)}_{S\in\bS}}4|\bS|^{2|\bS|}\pa{\frac{1}{K}}^{|supp(S^{(0)})|}\pa{\frac{K|supp(S^{(0)})|^2}{n}}^{|\bS|-1}\pa{\frac{1}{K}}^{\#\pi-1}\\
\leq&4|\bS|^{2|\bS|}\pa{\frac{1}{K}}^{|supp(S^{(0)})|}\pa{\frac{K|supp(S^{(0)})|^2}{n}}^{|\bS|-1}\sum_{k=1}^{|\bS|}\underset{|\pi|=k}{\sum_{\pi=\pa{\pi(S)}_{S\in \bS}}}\pa{\frac{1}{K}}^{k-1}\\
\leq&4|\bS|^{2|\bS|}\pa{\frac{1}{K}}^{|supp(S^{(0)})|}\pa{\frac{K|supp(S^{(0)})|^2}{n}}^{|\bS|-1}\sum_{k=1}^{|\bS|}\pa{\frac{1}{K}}^{k-1}K^k k^{|\bS|}\\
\leq& 4|\bS|^{3|\bS|+1}\pa{\frac{1}{K}}^{|supp(S^{(0)})|-1}\pa{\frac{K|supp(S^{(0)})|^2}{n}}^{|\bS|-1}\enspace,
\end{align*}
which concludes the proof of Lemma \ref{lem:controlreductedcumulantsclusteringnew}.

\section{Proof of Theorem \ref{thm:lowdegreefeature}}\label{prf:lowdegreefeature}

\begin{rem}
The observation $Y$ can either be seen as a matrix $Y\in \R^{KM\times p}$ or as a tensor $Y\in \R^{K\times M\times p}$. For clarity, we consider throughout the proof the tensor notation. But, we can still apply the framework for proving LD lower-bounds with latent structure from \cite{Even25a} which is written for the case where we observe a matrix $Y\in \R^{n\times p}$. 
\end{rem}

Without loss of generality, we assume in this proof that $\sigma^2=1$. Let $D\in \N$ such that $K\geq 2(D+2)^2$. We suppose
\begin{equation}\label{eq:def:zeta:feature}
\zeta:=64D^{40}\lambda^4p\max\pa{1, \frac{M}{K}}<1\enspace.
\end{equation}
Since the minimization problem defining $MMSE_{\leq D}$ in Equation (\ref{def:MMSE}) is separable, and since the random variables $\Gamma^*_{(k,m), (k',m')}$ are jointly exchangeable both with respect to $(k,k')$ and $(m,m')$, the $MMSE_{\leq D}$ can be reduced to
\begin{align*}
    MMSE_{\leq D}&:=\frac{1}{KMK(M-1)}\sum_{u=(k,m), u'=(k',m') \text{ with }m\neq m'} \inf_{f_{uu'}\in \R_{D}(Y)}\E\cro{\pa{f_{uu'}(Y)-\Gamma^*_{uu'}}^2}\\
&=\inf_{f\in \R_{D}(Y)}\E\cro{\pa{f(Y)-\Gamma^*_{(1,1),(1,2)}}^2}\enspace.
\end{align*}
For all $(k,m)\in [K]\times [M]$, in order to keep in line with the notation of Section \ref{sec:globalmethod}, we write $\pi^*(k,m)=\pi^*_m(k)$. In the remaining of the proof, we write $x=\Gamma^*_{(1,1),(1,2)}=\1\{\pi^*(1,1)=\pi^*(1,2)\}$. Then, our goal is to upper-bound 
$$MMSE_{\leq D}=\inf_{f\in \R_{D}(Y)}\E\cro{\pa{f(Y)-x}^2}\enspace.$$
 The model of Definition \ref{def:priorfeaturematching} is a particular case of the Gaussian Additive Model considered by \cite{SchrammWein22}. Therefore, we can use Theorem 2.2 from \cite{SchrammWein22}, stated in Equation \eqref{eq:SW22}. For a tensor $\alpha\in \N^{K\times M\times p}$, we define $|\alpha|=\sum_{k=1}^K\sum_{m=1}^{M}\sum_{j=1}^p \alpha_{k,m,j}$ and $\alpha!=\prod_{k,m,j}\alpha_{k,m,j}$. We write $X\in \R^{K\times M\times p}$ the signal matrix defined by $X_{k,m,j}=\pa{\mu_{\pi^*(k,m)}}_j$. Using Equation \eqref{eq:SW22}, we have the following lower-bound on the low-degree MMSE
 \begin{equation}\label{eq:SW22featurematching}
 MMSE_{\leq D}\geq \frac{1}{K}-\frac{1}{K^2}-\underset{0<|\alpha|\leq D}{\sum_{\alpha\in \N^{K\times M\times p}}}\frac{\kappa_{x,\alpha}^2}{\alpha!}\enspace,
 \end{equation}
where $\kappa_{x,\alpha}$ for $\alpha\in \N^{K\times M\times p}$ is the cumulant $\cumul(x,\ac{X_{k,m,j}}_{(k,m,j)\in \alpha})$, where $\ac{X_{k,m,j}}_{(k,m,j)\in \alpha}$ is the multiset that contains $\alpha_{k,m,j}$ copies of $X_{k,m,j}$, for $(k,m,j)\in[K]\times [M]\times [p]$.
Hence, in order to prove Theorem~\ref{thm:lowdegreefeature}, it is sufficient  to prove that
\begin{equation}\label{eq:objective_lower_bound_feature_matching}
   \underset{0<|\alpha|\leq D}{\sum_{\alpha\in \N^{K\times M\times p}}}\frac{\kappa_{x,\alpha}^2}{\alpha!} \leq 
   \frac{64(D+1)D^{36}}{K^2}\frac{\zeta}{1-\sqrt{\zeta}}\enspace . 
\end{equation}
The first step is to find necessary conditions for having $\kappa_{x,\alpha}\neq 0$. In the following, we define 
\begin{align*}
supp(\alpha)&:=\ac{(k,m)\in [K]\times [M],\enspace \sum_{j\in [p]}\alpha_{k,m,j}\neq 0}\enspace ;\\
sample(\alpha)&=\ac{m\in [M],\enspace \sum_{k\in [K]}\sum_{j\in [p]}\alpha_{k,m,j}\neq 0}\enspace ;\\
col(\alpha)&=\ac{j\in [p],\enspace \sum_{k\in [K]}\sum_{l\in [M]}\alpha_{k,m,j}\neq 0}\enspace\ ,
\end{align*}
where $col(\alpha)$ is the set of features involved in $\alpha$, $sample(\alpha)$ is the collection of data-sets involved in $\alpha$. 
The following lemma establishes some sufficient conditions for the nullity of the cumulant $\kappa_{x,\alpha}$. It is based on some independence properties together with Lemma~\ref{lem:independentcumulant}.

\begin{lem}\label{lem:nullitycumulantfeature}
Let $\alpha\in \N^{K\times M\times p}\neq0$ such that $\kappa_{x,\alpha}\neq 0$. Then, $|\alpha|\geq \max\pa{2|col(\alpha)|, 2|sample(\alpha)\cup\ac{1,2}|-2}$. In particular, we have $|\alpha|\geq 2$.
\end{lem}
The second step, which contains in fact the core of the arguments amounts to controlling the remaining $\kappa_{x,\alpha}$'s. 

\begin{lem}\label{lem:controlcumulantsfeature}
Let $\alpha\neq 0\in \N^{K\times M\times p}$ with $|\alpha|\geq 2$. Suppose that $K\geq 2\left|supp(\alpha)\cup \ac{(1,1), (1,2)}\right|^2$, then
$$|\kappa_{x,\alpha}|\leq \lambda^{|\alpha|}\pa{8|\alpha|^{18}}^{\frac{|\alpha|}{2}+1}\pa{\frac{1}{K}}^{\left|supp(\alpha)\cup\ac{(1,1),(1,2)}\right|-1}\enspace.$$
\end{lem}

For $d\in [2,D]$, $v\in [KM]$ and $l\in [M]$ satisfying $d\geq 2l$ and $v\geq l$, there exist at most $\pa{vd/2}^{d}K^{v-2}m^{l-2}p^{d/2}\leq \pa{d}^{2d}K^{v-2}M^{l-2}p^{d/2}$ matrices $\alpha$ with $d=|\alpha|$, $d\geq 2|col(\alpha)|$, $l=|sample(\alpha)\cup\ac{1,2}|$ and $v=|supp(\alpha)\cup \ac{(1,1), (1,2)}|$. Combining Lemma \ref{lem:nullitycumulantfeature} and Lemma \ref{lem:controlcumulantsfeature} leads us to 
\begin{align*}
\underset{0<|\alpha|\leq D}{\sum_{\alpha\in \N^{K\times m\times p}}}\frac{\kappa_{x,\alpha}^2}{\alpha!}\leq& 4\sum_{d=2}^{D}\sum_{v=2}^{d+2}\ \sum_{l=2}^{v\wedge (d/2+1)}D^{2d}K^{v-2}M^{l-2}p^{d/2}\pa{8D^{18}}^{d+2}\lambda^{2d}\pa{\frac{1}{K^2}}^{v-1}\\
\leq&\frac{64D^{36}}{K^2}\sum_{d=2}^{D}\sum_{v=2}^{d+2}\sum_{l=2}^{v\wedge (d/2+1)}\pa{\lambda^2\sqrt{p}8D^{20}}^{d}\pa{\frac{1}{K}}^{v-2}M^{l-2}\\
\leq& \frac{64D^{36}}{K^2}\sum_{d=2}^{D}\sum_{v=2}^{d+2}\sum_{l=2}^{v\wedge (d/2+1)}\pa{\lambda^2\sqrt{p}8D^{20}}^{d}\pa{\frac{M}{K}}^{l-2}\\
\leq &\frac{64(D+1)D^{36}}{K^2}\sum_{d=2}^{D}\sum_{l=0}^{d/2-1}\pa{\lambda^2\sqrt{p}8D^{20}}^{d-2l}\pa{64D^{40}\lambda^4p\frac{M}{K}}^{l}\\
\leq & \frac{64(D+1)D^{36}}{K^2}\sum_{d=2}^{+\infty}\sqrt{\zeta}^d\\
\leq & \frac{64(D+1)D^{36}}{K^2}\frac{\zeta}{1-\sqrt{\zeta}}\enspace,
\end{align*}
where we used the definition~\eqref{eq:def:zeta:feature} of $\zeta$ in the penultimate line. We have shown~\eqref{eq:objective_lower_bound_feature_matching}, which concludes the proof of Theorem \ref{thm:lowdegreefeature}.

\subsection{Proof of Lemma~\ref{lem:nullitycumulantfeature}}

Let $\alpha\in \N^{K\times M\times p}\neq 0$ such that $\kappa_{x,\alpha}\neq 0$. Let us first prove that $|\alpha|\geq 2|col(\alpha)|$. For that, it is sufficient to prove that, for all $j\in col(\alpha)$, $\sum_{k=1}^K\sum_{m=1}^M\alpha_{k,m,j}\geq 2$.

 Let us consider $\alpha$ such that there exists $j_0\in col(\alpha)$ with $\sum_{k=1}^K\sum_{m=1}^M\alpha_{k,m,j_0}=1$. The joint distribution of $(x, \pa{X_{k,m,j}}_{k,m,j\in [K]\times [M]\times [p]})$ and  $(x, \pa{\pa{-1}^{\1\ac{j=j_0}}X_{k,m,j}}_{k,m,j\in [K]\times [M]\times [p]})$ are the same, so we have 
\begin{align*}\kappa_{x,\alpha}=&\cumul\pa{x, (X_{k,m,j})_{k,m,j\in \alpha}}=\cumul\pa{x, (\pa{-1}^{\1\ac{j=j_0}}X_{k,m,j})_{k,m,j\in \alpha}}\\
=&-\cumul\pa{x, (X_{k,m,j})_{k,m,j\in \alpha}}=-\kappa_{x,\alpha}\enspace,
\end{align*}
which, in turn, implies $\kappa_{x,\alpha}=0$ for such an $\alpha$. As a consequence, if $\kappa_{x,\alpha}\neq 0$ then $|\alpha|\geq 2|col(\alpha)|$.

Let us now prove that $|\alpha|\geq 2|sample(\alpha)\cup \ac{1,2}|-2$. Using Proposition~\ref{thm:LTC}, we know that $|\alpha|$ is even and that there exists at least a decomposition $(\beta_1,\ldots,\beta_l)$ of $\alpha$, with $|\beta_s|=2$ of the form $\beta_s=\ac{(k_s, m_s, j_s), (k'_s, m'_s, j_s)}$, such that 
$$\cumul\pa{x, \pa{\1\ac{\pi^*(k_s,m_s)=\pi^*(k'_s,m'_s)}}_{s\in [1,l]}}\neq 0\enspace. $$
Let us take the convention $k_0=1$, $k'_0=1$, $m_0=1$ and $m'_0=2$. Hence, we get 
\begin{equation}\label{eq:cumulant:simplie:non_zero}
\cumul\pa{\pa{\1\ac{\pi^*(k_s,m_s)=\pi^*(k'_s,m'_s)}}_{s\in [0,l]}}\neq 0\enspace. 
\end{equation}
Let us prove the following claim: for all $m\in \cup_{s\in [0,l]}\ac{m_s, m'_s}$, we have $\sum_{s\in [0,l]}\1\ac{m_s=m}+\1\ac{m'_s=m}\geq 2$. To do so, we suppose that it is not the case and we shall find a contradiction. Let $m\in \cup_{s\in [0,l]}\ac{m_s, m'_s}$ with $\sum_{s\in [0,l]}\1\ac{m_s=m}+\1\ac{m'_s=m}= 1$. By symmetry, suppose that  $m=m_{s_0}$. Then, $\pi^*\pa{k_{s_0},m_{s_0}}$ is independent of $\pa{\pa{\pi^*(k_s,m_s)}_{s\in [0,l] \setminus s_0}, \pi^*\pa{k'_{s_0},m'_{s_0}}}$ and so $\1\ac{\pi^*(k_{s_0},m_{s_0})=\pi^*(k'_{s_0},m'_{s_0})}$ is a Bernoulli with parameter $1/K$ independent from the random variables $\pa{\1\ac{\pi^*(k_{s},m_{s})=\pi^*(k'_{s},m'_{s})}}_{s\in [0,l]\setminus \ac{s_0}}$. Using Lemma \ref{lem:independentcumulant}, we deduce that $\cumul\pa{\pa{\1\ac{\pi^*(k_s,m_s)=\pi^*(k'_s,m'_s)}}_{s\in [0,l]}}=0$, which contradicts~\eqref{eq:cumulant:simplie:non_zero}. We conclude that, for all $m\in \cup_{s\in [0,l]}\ac{m_s, m'_s}$, $\sum_{s\in [0,l]}\1\ac{m_s=m}+\1\ac{m'_s=m}\geq 2$. 

From this claim, we deduce that $2l+2\geq 2|sample(\alpha)\cup \ac{1,2}|$ which implies $|\alpha|\geq 2|sample(\alpha)\cup \ac{1,2}|-2$. This concludes the proof of Lemma~\ref{lem:nullitycumulantfeature}.

\subsection{Proof of Lemma \ref{lem:controlcumulantsfeature}}\label{prf:controlcumulantsfeature}

Let $\alpha\neq 0\in \N^{K\times M\times p}$. We recall the hypothesis $K\geq 2\left|supp(\alpha)\cup \ac{(1,1), (1,2)}\right|^2$. Using Proposition~\ref{thm:LTC}, we are reduced to upper-bounding, for any decomposition $(\beta_1,\ldots,\beta_l)$ of $\alpha$, with $|\beta_s|=2$ of the form $\beta_s=\ac{(k_s, m_s, j_s), (k'_s, m'_s, j_s)}$, the cumulant 
$$\cumul\pa{x, \pa{\1\ac{\pi^*(k_s,m_s)=\pi^*(k'_s,m'_s)}}_{s\in [1,l]}}\enspace. $$
Let us take the convention $k_0=1$, $k'_0=1$, $m_0=1$ and $m'_0=2$. With this notation, we can write
\begin{equation}\label{eq:cumulant:feature:partition}
\cumul\pa{x, \pa{\1\ac{\pi^*(k_s,m_s)=\pi^*(k'_s,m'_s)}}_{s\in [1,l]}}=\cumul\pa{\pa{\1\ac{\pi^*(k_s,m_s)=\pi^*(k'_s,m'_s)}}_{s\in [0,l]}}\enspace. 
\end{equation}

\paragraph{Step 1: Reduction.}  As explained in Section~\ref{sec:globalmethod}, we first reduce the control of the cumulant \eqref{eq:cumulant:feature:partition} to the control of cumulants of products of strongly dependent $\1\ac{\pi^*\pa{k_s,m_s}=\pi^*\pa{k'_s,m'_s} }$. As previously, we introduce a weighted graph to which we will apply Proposition \ref{prop:feray}. Let $\cW$ be the weighted graph on $[0,l]$ with weights $w_{s,s'}$ defined by;
\begin{itemize}
\item If $\ac{(k_s, m_s), (k'_s, m'_s)}\cap \ac{(k_{s'}, m_{s'}), (k'_{s'}, m'_{s'})}\neq \emptyset$, then $w_{s,s'}=1$;
\item If $\ac{(k_s, m_s), (k'_s, m'_s)}\cap \ac{(k_{s'}, m_{s'}), (k'_{s'}, m'_{s'})}=\emptyset$, then $w_{s,s'}=w\in (0,1)$, 
\end{itemize}
where $w$ is to be fixed later.

Let $S^{(0)}\subseteq [0,l]$. Denote by $\cW\<1\>[S^{(0)}]$ the subgraph induced by $S^{(0)}$ while only keeping the edges of weight 1, and by
$\bS=S_1,\ldots, S_{|\bS|}$ the partition of $S^{(0)}$  into the connected components of  $\cW\<1\>[S^{(0)}]$. In order to apply Proposition \ref{prop:feray}, we shall upper-bound, for such a $S^{(0)}$, the cumulant 
$$C_{\bS}:=\cumul\pa{\pa{\prod_{s\in S}\1\ac{\pi^*(k_s,m_s)=\pi^*(k'_s,m'_s)}}_{S\in \bS}}\enspace,$$
with respect to the maximum weight $\bbM\pa{\cW[S^{(0)}]}=\pa{w}^{|\bS|-1}$. The next lemma  (steps 2-4 of Section~\ref{sec:globalmethod}) provides such an upper-bound.

\begin{lem}\label{lem:controlreductedcumulantsfeature}
Let $S^{(0)}\subseteq [0,l]$ and let $\bS$ be its partition into the connected components of $\cW\<1\>[S^{(0)}]$. Writing $supp(S^{(0)})=\cup_{s\in S^{(0)}}\ac{(k_s, m_s), (k'_s, m'_s)}$, we have, whenever $K\geq 2|supp(S^{(0)})|^2$,
$$\left|C_{\bS}\right|\leq \pa{2|\bS|^{9}}^{|\bS|}|supp(S^{(0)})|^{2|\bS|}\pa{\frac{1}{K}}^{|supp(S^{(0)})|-1}\enspace.$$
\end{lem}

From Lemma \ref{lem:controlreductedcumulantsfeature}, we know that for any $S^{(0)}\subseteq [0,l]$ with $\bS$ the connected components of $\cW\<1\>[S^{(0)}]$,
$$\left|C_{\bS}\right|\leq \pa{2(l+1)^{9}}^{l+1}\pa{
2l+2}^{2(l+1)}\pa{\frac{1}{K}}^{|supp(S^{(0)})|-1}w^{-l-1}\bbM\pa{\cW[S^{(0)}]}\enspace.$$
The function $S\to \pa{\frac{1}{K}}^{|supp(S)|-1} $ is super-multiplicative. Hence, applying Proposition \ref{prop:feray} and Lemma~\ref{lem:subexponential} together with the fact that $\bbM\pa{\cW}\leq 1$, we get that 
$$\left|\cumul\pa{\pa{\1\ac{\pi^*(k_s,m_s)=\pi^*(k'_s,m'_s)}}_{s\in [0,l]}}\right| \leq w^{-l-1}\pa{8\pa{l+1}^{17}}^{l+1}\pa{\frac{1}{K}}^{|supp(\alpha)\cup\ac{(1,1), (1,2)}|-1}.$$
We also have $l=\frac{|\alpha|}{2}$ and so $l+1\leq |\alpha|$. Since the above bound is valid for any $w<1$, by taking $w$ arbitrarily close to one, we conclude that 
$$\left|\cumul\pa{\pa{\1\ac{\pi^*(k_s,m_s)=\pi^*(k'_s,m'_s)}}_{s\in [0,l]}}\right|\leq \pa{8|\alpha|^{17}}^{\frac{|\alpha|}{2}+1}\pa{\frac{1}{K}}^{\left|supp(\alpha)\cup\ac{(1,1),(1,2)}\right|-1}.$$
This bound of the cumulant~\eqref{eq:cumulant:feature:partition} is valid for for any decomposition $\beta_1,\ldots, \beta_l$ of $\alpha$ with $\beta_s=\ac{(k_s, m_s, j_s), (k'_s, m'_s, j_s)}$. Since the number of such decomposition is at most $|\alpha|^{\frac{|\alpha|}{2}-1}$, we deduce from Proposition~\ref{thm:LTC} that 
$$|\kappa_{x,\alpha}|\leq \lambda^{|\alpha|}\pa{8|\alpha|^{18}}^{\frac{|\alpha|}{2}+1}\pa{\frac{1}{K}}^{\left|supp(\alpha)\cup\ac{(1,1),(1,2)}\right|-1}\enspace.$$
This concludes the proof of Lemma \ref{lem:controlcumulantsfeature}.

\subsection{Proof of Lemma \ref{lem:controlreductedcumulantsfeature}}\label{prf:controlreductedcumulantsfeature}

\paragraph{Step 2: Expansion.} Let $S^{(0)}\subseteq [0,l]$ and let $\bS=S_1,\ldots, S_{|\bS|}$ the connected components of $\cW\<1\>[S^{(0)}]$. We seek to control 
$$C_{\bS}:=\cumul\pa{\pa{\prod_{s\in S}\1\ac{\pi^*(k_s,m_s)=\pi^*(k'_s,m'_s)}}_{S\in \bS}}\enspace.$$
For any $S\subseteq [0,l]$, we write $supp(S)=\cup_{s\in S}\ac{(k_{s}, m_s), (k'_s, m'_s)}$. For $R\subseteq [|\bS|]$, we write $I_{R}=\cup_{t\in R}supp(S_t)$. By definition of the graph $\cW$, the sets $supp(S_t)$ are disjoint and form a partition of $supp(S^{(0)})$. Moreover, for all $(k,m'),(k', m'')\in supp(S_t)$, there exists a path $s_1,\ldots, s_{h}\subseteq S_{t}$ such that $(k,m')\in \ac{(k_{s_1}, m_{s_1}), (k'_{s_1}, m'_{s_1})}$ and $(k', m'')\in \ac{(k_{s_h}, m_{s_h}), (k'_{s_h}, m'_{s_h})}$ and such that $\cW$ has an edge of weight $1$ between all $s_{h'}, s_{h'+1}$. Thus, for $t\in [|\bS|]$, $\prod_{s\in S_{t}}\1\ac{\pi^*(k_s,m_s)=\pi^*(k'_s,m'_s)}$ can be written 
$$\prod_{s\in S_{t}}\1\ac{\pi^*(k_s,m_s)=\pi^*(k'_s,m'_s)}=\sum_{a \in [K]}\prod_{(k,m)\in supp(S_t)}\1\ac{\pi^*(k,m)=a}\enspace.$$
Hence, we are interested in configurations $\pi^*$ that are constant on each $S_t$ for $t=1,\ldots, |\bS|$. For that purpose, we consider  $\pi= (\pi(S_t))_{t=1,\ldots,|\bS|}\in  [K]^{|\bS|}$. Then, using the multilinearity of cumulants, we get
\begin{equation}\label{eq:multilinearityfeature}
C_{\bS}=\sum_{\pi=\pa{\pi(S)}_{S\in \bS}\in [K]^{\bS}}\cumul\pa{\pa{\prod_{(k,m)\in supp(S)}\1\ac{\pi^*(k,m)=\pi(S)}}_{S\in \bS}}\enspace.
\end{equation}
Let us fix such a $\pi=\pa{\pi_S}_{S\in \bS}\in [K]^{\bS}$, and  for $S\in \bS$, let us define 
$$Z_{\pi,S}=\prod_{(k,m)\in supp(S)}\1\ac{\pi^*(k,m)=\pi(S)}.$$ 
  In the following, we shall control 
  $$C_{\bS}(\pi):=\cumul\pa{\pa{Z_{\pi,S}}_{S\in \bS}}=\cumul\pa{\pa{\prod_{(k,m)\in supp(S)}\1\ac{\pi^*(k,m)=\pi(S)}}_{S\in \bS}}.$$

\paragraph{Step 3: Pruning incompatible configurations.}
Note that $k\to \pi(k,m)$ defined by $\pi(k,m)=\pi(S)$ for $(k,m)\in supp(S)$ may not be injective, in which case it is impossible to have all the $Z_{\pi,S}$ non zero for $S\in\bS$. To cope with such cases,
we shall apply Proposition \ref{prop:feray} to a weighted dependency graph $\cV$  on $|\bS|$ defined by
\begin{itemize}
\item an edge of weight $1$ between $t$ and $t'$ if and only if $\pi(S_t)=\pi(S_{t'})$ and $\cup_{s\in S_t}\ac{m_{s}, m'_s}\cap \cup_{s\in S_{t'}}\ac{m_{s}, m'_s}\neq \emptyset$,
\item an edge of weight $\frac{|supp(S^{(0)})|^2}{K}$ otherwise.
\end{itemize}
We highlight the fact that, although this is not explicit, $\cV$ actually depends on the choice of $\pi=\pa{\pi(S)}_{S\in \bS}$. 
Weights 1 are attributed to connected components having the same value $\pi(S)$ and sharing some common datasets. 

Consider the case where  $R\subseteq [|\bS|]$ is such that $\cV$ has an edge of weight $1$ between two nodes of $R$. Then, writing $t$ and $t'$ two such nodes that are connected in $\cV$ be an edge of weight $1$, we observe that this implies that $\pi(S_t)= \pi(S_{t'})$ and there exists $m\in [M]$ such that $m$ belongs to both $\cup_{s\in S_t}\ac{m_{s}, m'_s}$ and $\cup_{s\in S_{t'}}\ac{m_{s}, m'_s}$. Hence, since $S_{t}$ and $S_{t'}$ are disjoint, having $Z_{\pi,S_t}=1$ and $Z_{\pi,S_{t'}}=1$ implies that there exists $k\neq k'$ such that $\pi^*(k,m)=\pi^*(k',m)$, the latter being impossible since $\pi^*(.,m)$ is a permutation of $[K]$. As a consequence, $Z_{\pi,S_{t}}Z_{\pi,S_{t'}}=0$ a.s. Denoting $\bR=R_1,\ldots, R_{|\bR|}$ the connected components of $\cV\<1\>[R]$,  we deduce from Lemma~\ref{lem:independentcumulant} that
\begin{equation}\label{eq:nullityfeaturereduced}
\cumul\pa{\pa{\prod_{t\in R_{h}}Z_{\pi,S_t}}_{h\in [|\bR|]}}=0\enspace.
\end{equation}
On the other hand, if $R\subseteq [|\bS|]$ is such that $\cV$ has no edges of weight $1$ between nodes in $R$, the partition $\bR$ of $R$ induced by the connected components of $\cV\<1\>[R]$ is made of singletons, and can be identified to $R$ itself. Thus, in order to apply Proposition \ref{prop:feray}, we only have to control the cumulants
$$C_{\bS[R]}(\pi):=\cumul\pa{\pa{Z_{\pi,S_t}}_{t\in R}}\enspace ,$$
for all subsets $R$ of $[|\bS|]$ such that $\cV[R]$ does not have any edge of weight $1$. In other words, we only have to control the cumulants for configurations $(\pi(S))_{S\in \bS}$ that are possibly compatible with $\pi^*$.

\paragraph{Step 4: Leveraging Lemma \ref{lem:upperboundgeneralcumulant}.} 

\begin{lem}\label{lem:reducedcumulantfeature2}
Let $R\subseteq [|\bS|]$ such that $\cV[R]$ does not contain any edge of weight $1$. If $K\geq 2|I_R|^2$, with $I_R=\cup_{t\in R}supp(S_t)$,
$$\left|C_{\bS[R]}(\pi)\right|\leq 2|R|^{2|R|}\pa{\frac{1}{K}}^{|I_R|}\pa{\frac{|I_R|^2}{K}}^{|R|-1}\enspace.$$
\end{lem}

\begin{proof}[Proof of Lemma~\ref{lem:reducedcumulantfeature2}]
Let $R\subseteq [|\bS|]$ such that $\cV[R]$ does not have any edge of weight $1$. For $m\in [M]$, we denote $A_{m}$ the set $\ac{(k,m), k\in [K]}$. Let $\Delta\subseteq R$, $I_\Delta=\cup_{t\in \Delta}supp(S_t)$, and let us compute 
\begin{align}\nonumber
\E\cro{\prod_{t\in \Delta}Z_{\pi,S_t}}=&\P\cro{\forall t\in \Delta, \forall (k,m)\in supp(S_t), \pi^*(k,m)=\pi(S_t)}\\\nonumber
=& \prod_{m=1}^M \frac{1}{K(K-1)\times\ldots\times \pa{K-\pa{|I_{\Delta}\cap A_{m}|-1}}}\\
=& \pa{\frac{1}{K}}^{|I_{\Delta}|}\prod_{m=1}^M \frac{1}{\pa{1-\frac{1}{K}}\times\ldots\times \pa{1-\pa{|I_{\Delta}\cap A_{m}|-1}\frac{1}{K}}}\enspace.\label{eq:leveragingfeature}
\end{align}
We are therefore in position to apply Lemma~\ref{lem:upperboundgeneralcumulant} with $\eta=\frac{1}{K}$ and $x_0=\frac{1}{K}$. Hence, we deduce from \eqref{eq:upperboundgeneralcumulant2} that, whenever $K\geq 2|I_R|^2$, we have
$$\left|C_{\bS[R]}(\pi)\right|\leq2|R|^{2|R|}\pa{\frac{1}{K}}^{|I_R|}\pa{\frac{|I_R|^2}{K}}^{|R|-1}\enspace,$$
which concludes the proof of the lemma.
\end{proof}

The function $R\mapsto \pa{\frac{1}{K}}^{|I_R|}$ is super-multiplicative.  Combining Lemma~\ref{lem:reducedcumulantfeature2} with Proposition~\ref{prop:feray} and Lemma~\ref{lem:subexponential}, we end up with 
$$\left|C_{\bS}(\pi)\right|\leq 2|\bS|^{8|\bS|}\pa{\frac{1}{K}}^{|supp(S^{(0)})|}\bbM\pa{\cV}\enspace.$$

\paragraph{Conclusion.}
For $\pi=\pa{\pi(S)}_{S\in\bS}\in [K]^{\bS}$, we write $cc(\pi)$ the number of connected components of the induced graph $\cV\<1\>$ (we recall that the graph $\cV$ depends on the choice of $\pi$). In light of the definition of $\bbM\pa{\cV}$, we arrive at
$$\left|C_{\bS}(\pi)\right|\leq 2|\bS|^{8|\bS|}\pa{\frac{1}{K}}^{|supp(S^{(0)})|}\pa{\frac{|supp(S^{(0)})|^2}{K}}^{cc(\pi)-1}\enspace , $$
for any $\pi$. 
Plugging this inequality in \eqref{eq:multilinearityfeature} leads us to 
$$|C_{\bS}|\leq 2|\bS|^{8|\bS|}\pa{\frac{1}{K}}^{|supp(S^{(0)})|}\sum_{cc=1}^{|\bS|}\underset{cc(\pi)=cc}{\sum_{\pi\in [K]^{\bS}}}\pa{\frac{|supp(S^{(0)})|^2}{K}}^{cc-1}\enspace.$$
The following lemma bounds the number of configurations $\pi$ with a given number of $cc(\pi)$. 
\begin{lem}\label{lem:numberk}
For any $cc\in [|\bS|]$, we have 
\[
\left|\ac{\pi=\pa{\pi_S}_{S\in \bS}\in [K]^{\bS},\enspace cc(\pi)=cc}\right|\leq (2|\bS|)^{|\bS|}K^{cc}\enspace.
\]
\end{lem}

Putting everything together, we conclude that 
\begin{align*}
|C_{\bS}|\leq& 2|\bS|^{8|\bS|}\pa{\frac{1}{K}}^{|supp(S^{(0)})|}\sum_{cc=1}^{|\bS|}\underset{cc(\pi)=cc}{\sum_{\pi\in [K]^{\bS}}}\pa{\frac{|supp(S^{(0)})|^2}{K}}^{cc-1}\\
\leq & 2\pa{2|\bS|^9}^{\bS}\pa{\frac{1}{K}}^{|supp(S^{(0)})|-1}\sum_{cc=1}^{|\bS|}|supp(S^{(0)})|^{2cc-2}\\
\leq& \pa{2|\bS|^{9}}^{|\bS|}|supp(S^{(0)})|^{2|\bS|}\pa{\frac{1}{K}}^{|supp(S^{(0)})|-1}\enspace.
\end{align*}
The proof of Lemma \ref{lem:controlreductedcumulantsfeature} is complete

\subsection{Proof of Lemma \ref{lem:numberk}}\label{prf:numberk}

Let us fix $cc\in [|\bS|]$, and let us count the number of $\pi=\pa{\pi_S}_{S\in \bS}$ with $cc(\pi)=cc$. To do so, we fix a permutation $\psi\in \mathcal{S}([|\bS|])$ of $[|\bS|]$ and we fix $\cG$ a subgraph of the set of edges $\ac{(\psi(1), \psi(2)),(\psi(2), \psi(3)),\ldots, (\psi(|\bS|-1), \psi(|\bS|))}$ having exactly $cc$ connected components. We first seek to upper-bound the number of $\pi$'s such that the induced graph $\cV\<1\>$ has the same connected components as $\cG$. In order to highlight the dependency of $\cV$ with respect to $\pi$, we now write $\cG_{\pi}$ for $\cV\<1\>$. If $\cG$ has an edge between $t$, $t'\in [|\bS|]$, then we must choose $\pi$ satisfying $\pi_{S_t}=\pi_{S_t'}$. So, for $R$ a connected component of $\cG$, we must have $\pi_{S_{t}}=\pi_{S_{t'}}$ for all $t,t'\in R$. Hence, there are at most $K^{cc}$ possibilities for choosing $\pi$ such that $\cG_\pi$ has the same connected components as  $\cG$. 

In total, the number of $\pi\in [K]^{\bS}$ such that $cc(\pi)=cc$ is at most $ (2|\bS|)^{|\bS|}K^{cc}$. This concludes the proof of Lemma \ref{lem:numberk}.

\section{Proof of Theorem \ref{thm:lowdegreeseriation}}\label{prf:lowdegreeseriation}

Without loss of generality, we assume through the proof that $\sigma^2=1$. Let $D\in \N$ such that $n\geq 2(2D+2)^2$. We suppose
$$\zeta:=2^{12}\lambda^2\pa{D+1}^{42}\max\pa{1,\frac{4\rho^2}{n}}<1\enspace.$$
Denote $x=\1\ac{|\pi^*(1)-\pi^*(2)|\leq \rho}$. Since the minimization problem defining $MMSE_{\leq D}$ in Equation (\ref{def:MMSE:seriation}) is separable, and since the random variables $X_{ij}$ are jointly exchangeable, the $MMSE_{\leq D}$ can be reduced to
$$MMSE_{\leq D}=\lambda^2 \frac{n-1}{n}\inf_{f\in \R_D[Y]}\E\cro{\pa{f(Y)-x}^2}\enspace.$$
We can compute $\E\cro{x}= \frac{2\rho}{n-1}\pa{1- \frac{\rho+1}{2n}}$.
The model considered is a particular case of the Gaussian Additive model analyzed by \cite{SchrammWein22} and hence, we can use Equation \eqref{eq:SW22}
\begin{equation*}\label{eq:SW22seriation}
\frac{n}{n-1} MMSE_{\leq D}\geq  \lambda^2 \frac{2\rho}{n-1}\pa{1- \frac{\rho+1}{2n}}-\frac{4\lambda^2 \rho^2}{(n-1)^2}\pa{1- \frac{\rho+1}{2n}}^2-\lambda^2 \underset{D\geq |\alpha|> 0}{\sum_{\alpha\in \N^{n\times n}}}\frac{\kappa_{x,\alpha}^2}{\alpha!}\enspace,
\end{equation*}
where $\kappa_{x,\alpha}$ for $\alpha\in \N^{n\times n}$ is the cumulant $\kappa_{x,\alpha}=\cumul\pa{x,\pa{X_{ij}}_{ij\in \alpha}}$, where $\pa{X_{ij}}_{ij\in \alpha}$
 is the multiset that contains $\alpha_{ij}$ copies of $X_{ij}$, for $i,j\in[1,n]$. In light of Equation \eqref{eq:SW22seriation}, it is sufficient, for proving Theorem \ref{thm:lowdegreeseriation}, to establish that 
$$\underset{D\geq |\alpha|> 0}{\sum_{\alpha\in \N^{n\times n}}}\frac{\kappa_{x,\alpha}^2}{\alpha!}\leq 2^{18}\frac{\pa{D+1}^{44}\rho^2}{n^2}\frac{\zeta}{1-\zeta}\enspace.$$
In the proof of this theorem, for $\alpha\in \N^{n\times n}$, we denote $supp(\alpha)=\ac{i\in [n],\enspace \alpha_{i:}\neq 0}\bigcup \ac{i\in [n],\enspace \alpha_{:i}\neq 0}$. Let us define a multigraph $\mathcal{G}_\alpha$ on $[n]$ with $\alpha_{ij}$ edges between the nodes $i$ and $j$, for $i,j\in [n]$. We define $\mathcal{G}_\alpha^-$ the graph obtained after removing the isolated notes (except $1$ and $2$) and we denote $cc(\alpha)$ the number of connected components of $\mathcal{G}_\alpha^-\cup \ac{(1,2)}$, which is the graph $\cG_{\alpha}^{-}$ where we have added an edge between $1$ and $2$. 

The following lemma states the desired control of the cumulant. 
\begin{lem}\label{lem:controlcumulantseriation}
For any $\alpha\neq 0$ with $n\geq 2\pa{|supp(\alpha)\cup\ac{1,2}|}^2$, we have 
$$|\kappa_{x,\alpha}|\leq \lambda^{|\alpha|}2^{5|\alpha|+7 }\pa{|\alpha|+1}^{20|\alpha|+21}
\pa{\frac{2\rho}{n}}^{|supp(\alpha)\cup \ac{1,2}|-1}\pa{\frac{1}{2\rho}}^{cc(\alpha)-1}\enspace.$$
\end{lem}
Before showing this lemma, we first conclude the proof of the theorem.

For $\alpha\neq 0$,  the number of edges of $\mathcal{G}_\alpha^-\cup \ac{(1,2)}$ is $|\alpha|+1$, and  the number of nodes is $|supp(\alpha)\cup \ac{1,2}|$. Since in $\cG_{\alpha}^{-}$,  no node except possibly $1$ and $2$ are isolated, one deduces from these definitions that  $$|\alpha|\geq \max\pa{|supp(\alpha)\cup\ac{1,2}|-cc(\alpha)-1,\frac{|supp(\alpha)\cup\ac{1,2}|-2}{2}} \ . $$ Moreover, the nodes $1$ and $2$ belong to the same connected component of $\mathcal{G}_\alpha^-\cup \ac{(1,2)}$, so we also have $$|supp(\alpha)\cup\ac{1,2}|\geq cc(\alpha)+1\enspace.$$ Moreover, given $d,cc\geq 1$ and $m\geq cc+1$, with $d\geq \max\pa{m-cc-1,\frac{m-2}{2}}$, there are at most $n^{m-2}m^{2d}\leq n^{m-2}(2d+2)^{2d}$  matrices $\alpha\in \N^{n\times n}$ satisfying $|\alpha|=d$, $|supp(\alpha)\cup \ac{1,2}|=m$ and $cc(\alpha)=cc$. Recall that we assume that $n\geq 2\pa{2D+2}^2$. Combining the last bound with Lemma~\ref{lem:controlcumulantseriation} leads us  to
\begin{align*}
\underset{D\geq |\alpha|> 0}{\sum_{\alpha\in \N^{n\times n}}}\frac{\kappa_{x,\alpha}^2}{\alpha!}\leq&2^{14}\pa{D+1}^{42}\hspace{-1cm}\underset{D\geq d\geq \max\pa{m-cc-1,\frac{m-2}{2}}}{\underset{m\geq cc+1}{\sum_{cc\geq 1}}}\hspace{-1cm}\pa{2D+2}^{2d}\lambda^{2d}\pa{2^{10}\pa{D+1}^{40}}^{d}\pa{\frac{4\rho^2}{n^2}}^{m-1}\pa{\frac{1}{4\rho^2}}^{cc-1}n^{m-2}\\
\leq&2^{16}\frac{\pa{D+1}^{42}\rho^2}{n^2}\underset{D\geq d\geq \max\pa{m-cc-1,\frac{m-2}{2}}}{\underset{m\geq cc+1}{\sum_{cc\geq 1}}}\pa{2^{12}\lambda^2\pa{D+1}^{42}}^d \pa{\frac{4\rho^2}{n}}^{m-2}\pa{\frac{1}{4\rho^2}}^{cc-1}.
\end{align*}
Let us fix $d,cc\geq 1$ and $m\geq cc+1$, with $d\geq \max\pa{m-cc-1,\frac{m-2}{2}}$, and let us upper-bound the term $\Upsilon:= \pa{2^{12}\lambda^2\pa{D+1}^{42}}^d \pa{\frac{4\rho^2}{n}}^{m-2}\pa{\frac{1}{4\rho^2}}^{cc-1}$. Using the hypothesis $d\geq m-cc-1$, we get 
\begin{align*}
\Upsilon\leq &\pa{2^{12}\lambda^2\pa{D+1}^{42}}^{d-m+cc+1}
\pa{\frac{4\rho^2}{n}2^{12}\lambda^2\pa{D+1}^{42}}^{m-cc-1}\pa{\frac{1}{n}}^{cc-1}
\leq \max(\zeta,1/n)^d\enspace 
\end{align*}
by definition of $\zeta$.
We end up with 
\begin{align*}
\underset{D\geq |\alpha|> 0}{\sum_{\alpha\in \N^{n\times n}}}\frac{\kappa_{x,\alpha}^2}{\alpha!}\leq& 2^{16}\frac{\pa{D+1}^{42}\rho^2}{n^2}\underset{D\geq d\geq \max\pa{m-cc-1,\frac{m-2}{2}}}{\underset{m\geq cc+1}{\sum_{cc\geq 1}}}\max(\zeta,1/n)^d\\
\leq & 2^{16}\frac{\pa{D+1}^{42}\rho^2}{n^2}\pa{2D+1}^{2}\sum_{d=1}^D \max(\zeta,1/n)^d\\
\leq& 2^{18}\frac{\pa{D+1}^{44}\rho^2}{n^2}\frac{\max(\zeta,1/n)}{1-\max(\zeta,1/n)}\enspace,
\end{align*}
which concludes the proof of Theorem \ref{thm:lowdegreeseriation}.

\subsection{Proof of Lemma \ref{lem:controlcumulantseriation}}\label{prf:controlcumulantseriation}

Let $\alpha\neq 0\in \N^{n\times n}$ and let us upper-bound the absolute value of $\kappa_{x,\alpha}=\cumul\pa{x,\pa{X_{ij}}_{ij\in \alpha}}$.
We recall that $X_{ij}=\lambda \1\ac{|\pi^*(i)-\pi^*(j)|\leq \rho}$. 
 From the multilinearity of cumulants, we know that $\kappa_{x,\alpha}$ is proportional to $\lambda^{|\alpha|}$. Without loss of generality, we suppose in the following that $\lambda=1$. Let $l=|\alpha|$ and let $(i_s,j_s)_{s\in [l]}$ such that $\alpha=\ac{(i_s, j_s), s\in [l]}$. We use  the convention $i_0=1$ and $j_0=2$. With those notation, $supp(\alpha)\cup \ac{1,2}=\cup_{s\in [0,l]}\ac{i_s, j_s}$, and we have 
\begin{equation}\label{eq:definition:cumulant:seriation}
\kappa_{x,\alpha}=\cumul\pa{\pa{\1\ac{|\pi^*(i_s)-\pi^*(j_s)|\leq \rho}}_{s\in [0,l]}}\enspace.
\end{equation}
This cumulant can be controlled using the general approach described in Section~\ref{sec:globalmethod}.

\paragraph{Step 1: Reduction.}  As explained in Section~\ref{sec:globalmethod}, we first reduce the control of the cumulant \eqref{eq:definition:cumulant:seriation} to the control of cumulants of products of strongly dependent $\1\ac{\pi^*\pa{i_s}=\pi^*\pa{j_s} }$. We introduce a weighted graph to which we shall apply Proposition \ref{prop:feray}. Let $\cW$ be the weighted graph on $[0,l]$ with weights $w_{s,s'}$ defined by:
\begin{itemize}
\item if $\ac{i_s,j_s}\cap \ac{i_{s'},j_{s'}}\neq \emptyset$, then $w_{s,s'}=1$,
\item if $\ac{i_s,j_s}\cap \ac{i_{s'},j_{s'}}= \emptyset$, then $w_{s,s'}=\frac{1}{2\rho}$.
\end{itemize}
For any subset $S^{(0)}\subseteq [0,l]$, we denote by $\cW\<1\>[S^{(0)}]$   the subgraph of $\cW[S^{(0)}]$, where we only keep the edges with weight one, and we 
 denote by $\bS$ the partition of $S^{(0)}$  into connected components of $\cW\<1\>[S^{(0)}]$. In order to apply Proposition \ref{prop:feray}, we have to bound the cumulant
$$C_{\bS}:=\cumul \pa{\pa{\prod_{s\in S}\1\ac{|\pi^*(i_s)-\pi^*(j_s)|\leq \rho}}_{S\in \bS}}\enspace  .$$
Such a control is provided in the next lemma (steps 2-4 of Section~\ref{sec:globalmethod}), which is proved in Section \ref{prf:reducedcumulantseriation}. For any $S\subseteq [0,l]$, we write $supp(S)=\cup_{s\in S}\ac{i_s,j_s}$. In particular, we have $supp(\alpha)=supp([1,l])$ and $supp(\alpha)\cup \ac{1,2}=supp([0,l])$. 

\begin{lem}\label{lem:reducedcumulantseriation}
Let $S^{(0)}\subseteq [0,l]$ and suppose $n\geq 2\pa{|supp(S^{(0)})|}^2$. Let $\bS$ be the set of connected components of $\cW\<1\>[S^{(0)}]$. Then,
$$|C_{\bS}|\leq  4|\bS|\pa{2|\bS|^{10}|supp(S^{(0)})|^4}^{|\bS|} \pa{\frac{2\rho}{n}}^{|supp(S^{(0)})|-1}\pa{\frac{1}{2\rho}}^{|\bS|-1}\enspace.$$
\end{lem}

Since $\bbM\pa{\cW[S^{(0)}]}=\pa{\frac{1}{2\rho}}^{|\bS|-1}$,
Lemma \ref{lem:reducedcumulantseriation} implies that, for all $S^{(0)}\subseteq [0,l]$, 
$$\left|C_{\bS}\right|\leq 4\pa{l+1}\pa{2(l+1)^{10}|supp(\alpha)\cup\ac{1,2}|^4}^{l+1}\pa{\frac{2\rho}{n}}^{|supp(S^{(0)})|-1}\bbM\pa{\cW[S^{(0)}]}.$$
The function $S^{(0)}\mapsto \pa{\frac{2\rho}{n}}^{|supp(S^{(0)})|-1}$ being super-multiplicative,
applying Proposition \ref{prop:feray} and Lemma~\ref{lem:subexponential}, with the weighted dependency graph $\cW$, we get that, since $l=|\alpha|$ and $|supp(\alpha)\cup\ac{1,2}|\leq 2(|\alpha|+1)$, 
$$|\kappa_{x,\alpha}|\leq 4\pa{|\alpha|+1}\pa{2^5(|\alpha|+1)^{20}}^{|\alpha|+1}\pa{\frac{2\rho}{n}}^{|supp(\alpha)\cup \ac{1,2}|-1}\bbM\pa{\cW}\enspace.$$
The proof of Lemma \ref{lem:controlcumulantseriation} is complete since $\bbM\pa{\cW}=\pa{\frac{1}{2\rho}}^{cc(\alpha)-1}$.

\subsection{Proof of Lemma \ref{lem:reducedcumulantseriation}}\label{prf:reducedcumulantseriation}

\paragraph{Step 2: Expansion.} Let $S^{(0)}\subseteq [0,l]$ and let $\bS=S_1,\ldots, S_{|\bS|}$ the connected components $\cW\<1\>[S^{(0)}]$. Recall that we aim at bounding in absolute value the following cumulant
$$C_{\bS}=\cumul \pa{\pa{\prod_{s\in S}\1\ac{|\pi^*(i_s)-\pi^*(j_s)|\leq \rho}}_{S\in \bS}} \enspace.$$

For $S\in \bS$, we denote $\mathcal{A}_{S}$ the set of all injections $\pi_S:supp(S)\to [n]$ such that, for all $s\in S$, $|\pi_S(i_s)-\pi_S(j_s)|\leq \rho$. We have, for $S\in \bS$, 
$$\prod_{s\in S}\1\ac{|\pi^*(i_s)-\pi^*(j_s)|\leq \rho}=\sum_{\pi_S\in \mathcal{A}_{S}}\prod_{i\in supp(S)}\1\ac{\pi^*(i)=\pi_S(i)}\enspace.$$
By the multilinearity of cumulants, we deduce 
\begin{equation}\label{eq:multilinearityseriation}
C_{\bS}=\sum_{\pi=\pa{\pi_{S}}_{S\in \bS}\in \prod_{S\in \bS}\mathcal{A}_{S}}\cumul\pa{\pa{\prod_{i\in supp(S)}\1\ac{\pi^*(i)=\pi_{S}(i)}}_{S\in \bS}}\enspace.
\end{equation}

Let us fix $\pi=\pa{\pi_S}_{S\in \bS}\in \prod_{S\in \bS}\mathcal{A}_{S}$. For $S\in \bS$ and $i\in supp(S)$, we write $\pi(i)=\pi_S(i)$. For $S\in \bS$, let us write $Z_{\pi, S}=\prod_{i\in supp(S)}\1\ac{\pi^*(i)=\pi(i)}$. In light of~\eqref{eq:multilinearityseriation}, we need to control the cumulants
$$C_{\bS}(\pi):=\cumul\pa{\pa{Z_{\pi, S}}_{S\in \bS}}\enspace.$$
Note that the corresponding function  $\pi$ is possibly non-injective, in which case it is impossible that all $Z_{\pi, S}$ with $S\in \bS$ are simultaneously non-zero. Indeed, if there exists $i\in S$ and $i'\in S'$ such that $\pi(i)=\pi(i')$, it is impossible to have $\pi^*(i)=\pi(i)$ and $\pi^*(i')=\pi(i')$ since $\pi^*$ is permutation of $[n]$.

\paragraph{Step 3: Pruning incompatible configurations.}  We recall that we write $\bS=S_1,\ldots, S_{|\bS|}$. We shall apply Proposition \ref{prop:feray} to some weighted graph $\cV$ with vertex set $[|\bS|]$ with:
\begin{itemize}
\item an edge of weight $1$ between $t,t'\in [|\bS|]$ if and only if $\ac{\pi(i), i\in supp(S_t)}\cap \ac{\pi(i), i\in supp(S_{t'})}\neq \emptyset$,
\item an edge of weight $\frac{|supp(S^{(0)})|^2}{n}$ between $t$ and $t'$ otherwise.
\end{itemize}
As argued above, if the graph $\cV$ contains at least an edge with weight $1$, then $\pi^*$ cannot be compatible with $\pi$. The graph $\cV$ depends on $\pi$ but we do not write explicitly this dependency for simplicity of notation. If $R\subseteq [|\bS|]$ is such that $\cV$ has an edge of weight $1$ between two nodes of $R$,  denote $\bR=R_1,\ldots,R_{|\bR|}$ the partition of $R$ into connected components of $\cV\<1\>[R]$. If such a connected component $R_h$ contains at least two nodes, this implies that $\prod_{t\in R_{h}}Z_{\pi, S_t}=0$ a.s. since $\prod_{t\in R_{h}}Z_{\pi, S_t}\neq 0$ would imply that, for some $i\neq i'$, $\pi^*(i)= \pi^*(i')$. It then follows from Lemma~\ref{lem:independentcumulant} that 
\begin{equation}\label{eq:nullityseriationreduced}
\cumul\pa{\pa{\prod_{t\in R_{h}}Z_{\pi, S_t}}_{h\in [|\bR|]}}=0\enspace.
\end{equation}

On the other hand, if $R\subseteq [|\bS|]$ is such that $\cV$ has no edges of weight $1$ between nodes in $R$, the partition $\bR$ of $R$ induced by the connected components of $\cV\<1\>[R]$ is made of singletons, and can be identified to $R$ itself. Thus, in order to apply Proposition \ref{prop:feray}, we only have  to control the cumulant $\left| \cumul\pa{\pa{Z_{\pi, S_t}}_{t\in R}}\right|$ for all subsets $R$ of $[|\bS|]$ such that $\cV[R]$ does not have any edge of weight $1$.

\paragraph{Step 4: Leveraging Lemma \ref{lem:upperboundgeneralcumulant}.} For $R\subseteq [|\bS|]$, let us write $I_R=\cup_{t\in R}supp(S_t)$.

\begin{lem}\label{lem:reducedcumulantseriation2}
Let $R\subseteq [|\bS|]$ such that $\cV[R]$ does not have any edge of weight $1$. If $n\geq 2|I_R|^2$, 
$$\left|\cumul\pa{\pa{Z_{\pi, S_t}}_{t\in R}}\right|\leq 2|R|^{2|R|}\pa{\frac{1}{n}}^{|I_R|}\pa{\frac{|I_R|^2}{n}}^{|R|-1}\enspace.$$
\end{lem}
\begin{proof}[Proof of Lemma~\ref{lem:reducedcumulantseriation2}]

Let $R\subseteq [|\bS|]$ such that $\cV[R]$ does not have any edge of weight $1$. We recall that  this corresponds to the case where the event such that $\pi^*(i)= \pi(i)$ for all $i\in S_t$ with $t\in R$ holds with positive probability. 
Equivalently, this implies that $\pi$ is an injection on $I_{R}$. We seek to upper-bound $\left| \cumul\pa{\pa{Z_{\pi, S_t}}_{t\in R}}\right|$, where, for $t\in R$, $Z_{\pi, S_t}=\prod_{i\in supp(S_t)}\1\ac{\pi^*(i)=\pi(i)}$. Let $\Delta\subseteq R$, and let us compute, writing $I_\Delta=\cup_{t\in \Delta}supp(S_t)$,
\begin{align}\nonumber
\E\cro{\prod_{t\in \Delta} Z_{\pi, S_t}}=&\P\cro{\forall t\in \Delta, \enspace \forall i\in supp(S_t),\enspace \pi^*(i)=\pi(i)}\\\nonumber
=&\frac{1}{n(n-1)\times \ldots\times (n-|I_\Delta|+1)}\\
=&\pa{\frac{1}{n}}^{|I_\Delta|}\frac{1}{1-\frac{1}{n}}\times \ldots\times \frac{1}{1-\frac{|I_\Delta|-1}{n}}\enspace.\label{eq:leveragingseriation}
\end{align}
We are in position to apply the general bound for cumulants  with weakly dependant random variables from Lemma~\ref{lem:upperboundgeneralcumulant}. Indeed, applying Equation \eqref{eq:upperboundgeneralcumulant2} with $x_0=\eta=\frac{1}{n}$, leads us to
$$\left| \cumul\pa{\pa{Z_{\pi, S_t}}_{t\in R}}\right|\leq 2|R|^{2|R|}\pa{\frac{1}{n}}^{|I_R|}\pa{\frac{|I_R|^2}{n}}^{|R|-1}\enspace.$$
whenever $n\geq 2|I_R|^2$.

\end{proof}

For all $R\subseteq [|\bS|]$, denoting $\bR=R_1,\ldots, R_{|\bR|}$ the connected components of $\cV\<1\>[R]$, we have 
$\bbM\pa{\cV[R]}=\pa{\frac{|I_R|^2}{n}}^{|R|-1}$.
Hence, when $n\geq 2\pa{supp(S^{(0)})}^2$, the identity \eqref{eq:nullityseriationreduced} and Lemma \ref{lem:reducedcumulantseriation2} imply that, 
$$\left|\cumul\pa{\pa{\prod_{t\in R_{h}}Z_{\pi, S_t}}_{h\in [|\bR|]}}\right|\leq 2|R|^{2|R|}\pa{\frac{1}{n}}^{|I_R|}\bbM\pa{\cV[R]}\enspace.$$
Since the function $R\to \pa{\frac{1}{n}}^{|I_R|}$ is super-multiplicative, we have all the ingredients to apply Proposition \ref{prop:feray} and Lemma~\ref{lem:subexponential},  leading to
$$\left|C_{\bS}(\pi)\right|= \left|\cumul\pa{\pa{Z_{\pi, S_t}}_{t\in [|\bS|]}}\right|\leq 2|\bS|^{8|\bS|}\pa{\frac{1}{n}}^{|supp(S^{(0)})|}\bbM\pa{\cV}\enspace.$$

\paragraph{Conclusion.} 
For $\pi=\pa{\pi_S}_{S\in \bS}\in \prod_{S\in \bS}\mathcal{A}_{S}$, we denote $cc(\pi)$ the number of connected components of $\cV\<1\>$ so that $\bbM\pa{\cV} = (|supp(S^{(0)})|^2/n)^{cc(\pi)-1}$. Hence, we have 
$$\left|C_{\bS}(\pi)\right|\leq 2|\bS|^{8|\bS|}\pa{\frac{1}{n}}^{|supp(S^{(0)})|}\pa{\frac{|supp(S^{(0)})|^2}{n}}^{cc(\pi)-1}\enspace.$$
Plugging this inequality in \eqref{eq:multilinearityseriation} leads us to 
\begin{equation}\label{eq:upper:C_bS}
|C_{\bS}|\leq 2|\bS|^{8|\bS|}\pa{\frac{1}{n}}^{|supp(S^{(0)})|}\sum_{cc=1}^{|\bS|}\left|\ac{\pi,\enspace cc(\pi)=cc}\right| \pa{\frac{|supp(S^{(0)})|^2}{n}}^{cc-1}\enspace.
\end{equation}

Let us fix $cc\in [|\bS|]$ and let us count the number of $\pi$'s with $cc(\pi)=cc$.  
\begin{lem}\label{lem:numbersigma}
For any $cc\in [|\bS|]$,
$$\left|\ac{\pi\in \prod_{S\in \bS}\cA_S,\enspace cc(\pi)=cc}\right|\leq (2|\bS|^2|supp(S^{(0)})|^2)^{|\bS|}n^{cc}\pa{2\rho}^{|supp(S^{(0)})|-|\bS|}\enspace.$$
\end{lem}

Combining Lemma \ref{lem:numbersigma} and Eq.~\eqref{eq:upper:C_bS} leads us to 
\begin{align*}
|C_{\bS}|\leq& 2|\bS|^{8|\bS|}\pa{\frac{1}{n}}^{|supp(S^{(0)})|}\sum_{cc=1}^{|\bS|}(2|\bS|^2|supp(S^{(0)})|^2)^{|\bS|}n^{cc}\pa{2\rho}^{|supp(S^{(0)})|-|\bS|} \pa{\frac{|supp(S^{(0)})|^2}{n}}^{cc-1}\\
\leq&2n\pa{2|\bS|^{10}|supp(S^{(0)})|^2}^{|\bS|} \pa{\frac{2\rho}{n}}^{|supp(S^{(0)})|}\pa{\frac{1}{2\rho}}^{|\bS|}\sum_{cc=1}^{|\bS|}|supp(S^{(0)})|^{2(cc-1)}\\
\leq& 4|\bS|\pa{2|\bS|^{10}|supp(S^{(0)})|^4}^{|\bS|} \pa{\frac{2\rho}{n}}^{|supp(S^{(0)})|-1}\pa{\frac{1}{2\rho}}^{|\bS|-1}\enspace.
\end{align*}
This concludes the proof of Lemma \ref{lem:reducedcumulantseriation}.

\subsection{Proof of Lemma \ref{lem:numbersigma}}\label{prf:numbersigma}

Let us fix $cc\in [|\bS|]$ and let us count the number of $\pi=\pa{\pi_S}_{S\in \bS}\in \prod_{S\in \bS}\mathcal{A}_S$'s with $cc(\pi)=cc$. To do so, we fix a permutation $\psi\in \mathcal{S}([|\bS|])$ of $|\bS|$ and we fix $\cG$ a subgraph of the set of edges $\ac{(\psi(1), \psi(2)),(\psi(2), \psi(3)),\ldots, (\psi(|\bS|-1), \psi(|\bS|))}$ having exactly $cc$ connected components. There are less than $|\bS|! 2^{|\bS|}\leq (2|\bS|)^{|\bS|}$ possible choices for $\cG$ and $\psi$. We first seek to upper-bound the number of $\pi$'s such that the induced weighted graph $\cV\<1\>$ with vertex set $[|\bS|]$ has the same connected components as $\cG$. We recall that $\cV\<1\>$ has an edge of weight $1$ between $t,t'\in [|\bS|]$ if and only if there exists $i\in S_t$ and $i'\in S_{t'}$ with $\pi(i)=\pi(i')$. In the following, we write $\cG_{\pi}$ for the graph $\cV\<1\>$ in order to emphasize the dependency with respect to $\pi$. Let us upper-bound the number of $\pi\in \prod_{S\in \bS}\mathcal{A}_S$ such that $\cG_{\pi}$ and $\cG$ have the same connected components. Without loss of generality, we suppose in the following that $\psi$ is the identity; in particular, the connected components of $\cG$ are intervals.

Let us fix $R\subseteq [|\bS|]$ a connected component of $\cG$ and we upper-bound the number of choices for $\pi^{(R)}=\pa{\pi_{S_t}}_{t\in R}\in \prod_{t\in R}\mathcal{A}_{S_t}$ connecting $R$, i.e such that the graph $\cG_{\pi^{(R)}}$ with vertex set $R$ is connected. By symmetry, we can suppose that $R=[|R|]$. Let us choose a permutation $\phi$ of $[|R|]$ and let us count the number of $\pi^{(R)}$ satisfying; for all $t\in [2,|R|]$, $\cG_{\pi^{(R)}}$ has an edge between $\phi(t)$ and $\ac{\phi(t'), t'\leq t-1}$. Without loss of generality, we now suppose that the chosen and fixed permutation $\phi$ is the identity. Let us choose for $t\in [2,|R|]$, an element $i^{t}_1\in supp(S_{t})$ and let us count the number of $\pi^{(R)}=\pa{\pi_{S_t}}_{t\in R}$ such that $\pi_{S_{t}}(i^{t})$ is in $\pa{\cup_{t'\leq t-1}\cup_{i\in S_{t'}}\pi_{S_{t'}}(i)}$. There are at most $\prod_{t\in [|R|]}|supp(S_{t})|$ possibilities for choosing the $i^{t}_1$'s. For $t\leq [|R|]$, let us order $S_{t}=\ac{s_1^{t},\ldots, s^{t}_{|S_{t}|}}$, with $i^{t}_1\in \ac{i_{s_1^{t}}, j_{s_1^{t}}}$, in such a way that, for all $j\in [2,|S_{t}|]$, $\ac{i_{s_j^{t}}, j_{s_j^{t}}}$ intersects $\bigcup_{j'\leq j-1}\ac{i_{s_{j'}^{t}}, j_{s_{j'}^{t}}}$. Then:

\begin{itemize}
\item For $S_{1}$, we choose a value in $[n]$ for $\pi_{S_1}(i_{s_1^{1}})$. That makes $n$ possibilities. Then, choosing the other labels for the elements $supp({S_{1}})$ in the order of appearance in the couples indexed by $\ac{s_1^{1},\ldots, s^1_{|S_{1}|}}$, we have at each time only $2\rho$ possibilities (each label must be at distance at most $\rho$ from the already chosen label of the point which is in the same $\ac{i_{s_j^{1}}, j_{s_j^{1}}}$ for some $j\leq |S_1|$). That makes at most $n(2\rho)^{|supp(S_1)|-1}$ possibilities.
\item For the other $S_{t}$'s, the same thing occurs except that the label of the first element $i^{t}_1$ must be in the labels already chosen previously. That makes at most $|I_R|\pa{2\rho}^{|supp\pa{S_{t}}|-1}$ possibilities, where we recall that $I_R=\cup_{t\leq |R|}supp(S_{t})$.
\end{itemize}
In total, counting the number of possibilities for choosing the $i^{t}_1$'s and for choosing the permutation $\phi$, we have $|R|^{|R|}n\pa{2\rho}^{|I_R|-|R|}|I_R|^{|R|-1}\prod_{t\in R}|supp(S_{t})|$ possibilities. Doing so for all the connected components of $\cG$, the number of $\pi$'s such that $\cG_{\pi}$ has the same connected components as $\cG$ is at most 
$$n^{cc}\pa{2\rho}^{|supp(S^{(0)})|-|\bS|}\pa{|\bS||supp(S^{(0)})|^2}^{|\bS|}\enspace.$$

Multiplying for the number of choices for the graph $G$ and for the permutation $\psi$, we deduce that the number of $\pi\in \prod_{S\in \bS}\mathcal{A}_{S}$ with $cc(\pi)=cc$ is at most 
$$(2|\bS|^2|supp(S^{(0)})|^2)^{\bS}n^{cc}\pa{2\rho}^{|supp(S^{(0)})|-|\bS|}\enspace,$$
which concludes the proof of the lemma.

\section{Proof of Lemma~\ref{lem:upperboundgeneralcumulant}}\label{sec:quazifactorizationseries}

This section is devoted to the proof of Lemma~\ref{lem:upperboundgeneralcumulant}. We first treat the first point of the lemma (case $r\geq 1$), postponing to Section \ref{ref:caser0} the proof for the particular case $r=0$. Assume that $Z_{1},\ldots,Z_{\ell}$ have mixed moments given by \eqref{eq:formula:general_momodent}: for all $\Delta\subseteq R$
\[
\E\cro{\prod_{t\in \Delta}Z_{t}}=\eta^{|I_\Delta|}\prod_{j=1}^q \frac{1}{1-x_0}\times\ldots\times \frac{1}{1-(|I_\Delta\cap A_{j}|-1)x_0}\times\prod_{i=1}^r (1-y_0)\times\ldots\times \pa{1-\pa{|I_\Delta \cap B_{i}|-1}y_0}. \nonumber\ 
\]
The random variables $\pa{Z_{t}}_{t\in [\ell]}$ have their moments that are close to be factorized. If instead of~\eqref{eq:formula:general_momodent}, we had considered random variables $\pa{\tilde{Z}_{t}}_{t\in [\ell]}$, such that, for any $\Delta\subseteq R$, we had 
\[
\E\cro{\prod_{t\in \Delta}\tilde{Z}_{t}}=\eta^{|I_\Delta|}\prod_{j=1}^q \left[\frac{1}{1-x_0}\right]^{|I_{\Delta}\cap A_j|}\prod_{i=1}^r (1-y_0)^{|I_\Delta \cap B_{i}|}\enspace , \nonumber\ 
\]
then, we would have $\E\cro{\prod_{t\in \Delta}\tilde{Z}_{t}}= \prod_{t\in \Delta}\E\cro{\tilde{Z}_{t}}$. Since the mixed moments of $\tilde{Z}_{t}$ are the same as if they were independent, it then follows from Lemmas~\ref{lem:independentcumulant} and~\ref{lem:mobiusformula} that $ \cumul\pa{\pa{\tilde{Z}_{t}}_{t\in [\ell]}}=0$.  The purpose of Lemma~\ref{lem:upperboundgeneralcumulant} is to quantify to what extent the fact that the moments of the $Z_{t}$'s  in~\eqref{eq:formula:general_momodent} are close to be factorized implies that their joint cumulant is small.

\subsection{Small-cumulant and quasi-factorization}

A key result of F\'eray~\cite{feray2018} is Proposition~5.8 (in~\cite{feray2018}), that we recall here informally. Let $Z_{1},\ldots,Z_{\ell}$ be random variables, with a distribution depending on some parameter $n$, and let $\cL$ be a weighted graph with vertex set $[\ell]$. If for all $\Delta \subseteq [\ell]$ of size at least 2 we have the quasi-factorisation
\begin{equation} \label{eq:quasi-factor}
\prod_{\delta\subseteq \Delta}\E\cro{\prod_{t\in\delta}Z_{t}}^{(-1)^{|\Delta|-|\delta|}}=1+O_{n}(\bbM(\cL[\Delta])),
\end{equation}
then, we have the small cumulant property  for all $\Delta \subseteq [\ell]$ of size at least 2
\begin{equation} \label{eq:small-cumul}
\cumul\pa{\pa{Z_{t}}_{t\in \Delta}}=\prod_{t\in\Delta} \E\cro{Z_{t}} \times O_{n}(\bbM(\cL[\Delta])),
\end{equation}
where $O_{n}(\cdot)$ refers to the parameter $n$, and $\bbM(\cL[\Delta])$ is the weight of the graph $\cL$ restricted to the node set $\Delta$, see Section~\ref{sec:weighted-graph} for the definition of $\bbM$.

To prove Lemma~\ref{lem:upperboundgeneralcumulant}, we provide a non-asymptotic and quantitative version of this result.
Instead of considering moments $\E\cro{\prod_{t\in\delta}Z_{t}}$ as real numbers, our approach is to consider them as bivariates series in $(x_{0},y_{0})$. In this direction, 
for any subset $\Delta\subseteq [\ell]$, we introduce the multivariate series
\begin{align}\label{eq:definition_u_Delta}
u_\Delta(x,y):=&\eta^{|I_\Delta|}\prod_{j=1}^q \frac{1}{1-x}\times\ldots\times \frac{1}{1-(|I_\Delta\cap A_{j}|-1)x}\cdot \prod_{i=1}^{r}(1-y)\times\ldots\times \pa{1-\pa{|I_\Delta \cap B_{i}|-1}y}, 
\end{align}
with the convention that $u_{\emptyset}=1$. We have $\E\cro{\prod_{t\in \Delta}Z_{t}}= u_\Delta(x_0,y_0)$, and 
 $u_\Delta(x,y)$  is well-defined when $x<\frac{1}{L}$, where $L=|I_{[\ell]}|$.

Having in mind the Möbius formula (Lemma~\ref{lem:mobiusformula}) for cumulants,  we define, for $\Delta\subseteq [\ell]$, the bivariate series
 \begin{equation}\label{eq:definition:kappa_Delta}
 \kappa_{\Delta}(x,y):=\sum_{G\in \mathcal{P}[\Delta]}m(G)\prod_{\delta\in G}u_{\delta}(x,y)\enspace ,
 \end{equation}
so that 
 \begin{equation}\label{eq:cumul:kappa}
    \cumul\pa{\pa{Z_{t}}_{t\in [\ell]}}=\kappa_{[\ell]}(x_0,y_0)\enspace.
    \end{equation}
 To derive a non asymptotic upper bound on $\kappa_{[\ell]}(x_0,y_0)$, we will first control the so-called \emph{order} (defined in Section~\ref{sec:order})  of the bivariate series $\kappa_{[\ell]}(x,y)$, by building on ideas from~\cite{feray2018}.  Then, this will allow us to work around the non-asymptotic bound by carefully bounding the terms. The counterpart of the 
 quasi-factorisation property \eqref{eq:quasi-factor} in our framework, will be some lower bound on the order of  the following bivariate rational function $P_{\Delta}[u]$.
 
 \begin{definition}\label{defi:P_delta}
Consider any family $(u_{\Delta}(x,y))$ with $\Delta\subseteq [\ell]$ of bivariate power series with an open domain of convergence.  
For $\Delta\subseteq  [\ell]$ non empty, we define the multivariate rational function 
\begin{equation}\label{eq:definition_P_Delta}
P_{\Delta}[u](x,y):=\prod_{\delta\subseteq \Delta}\pa{u_{\delta}(x,y)}^{(-1)^{|\Delta|-|\delta|}}\enspace. 
\end{equation}
\end{definition}

A key property of the rational function $P_{\Delta}[v](x,y)$ is that when the power series $(v_\Delta(x,y))_{\Delta\subset[\ell]}$ can be factorized, as if they were mixed moments of variables with some independence, then $P_{\Delta}[v](x,y)=1$.

\begin{lem}\label{lem:independentquasifacto}
Let $\ell \geq 1$ and $v=(v_\Delta(x,y))_{\Delta\subseteq[\ell]}$ be a family of power series indexed by subsets of $[\ell]$. For some $(x,y)$ in the domain, we suppose that  for all $\Delta\subseteq [\ell]$, $v_\Delta(x,y)\neq 0$. Let $\Delta\subseteq[\ell]$, and let us suppose that there exist $I_1$ and $I_2$ forming a partition of $\Delta$ such that, for all $\delta\subseteq \Delta$, one has $v_\delta(x,y)=v_{\delta\cap I_1}(x,y)v_{\delta\cap I_2}(x,y)$. Then, $P_\Delta[v](x,y)=1$.
\end{lem}

In the two following subsection, we provide a counterpart of the small cumulant property~\eqref{eq:small-cumul}, by proving that
the order of $\kappa_{[\ell]}(x,y)$ is at least that of $P_{[\ell]}[u](x,y)-1$.  Then, we will lower bound the latter  for $u$ defined in~\eqref{eq:definition_u_Delta}.

\subsection{Order of a power series}\label{sec:order}

For our purpose, we introduce a  specific notion of \emph{order} of a power series, which slightly differs from that in the literature. 
Let us first introduce a partial order relation on $\mathbb{N}^2$. 
\begin{definition}\label{def:orderrelation}
Let $s_0,s'_0\geq 0$ and $s_1,s'_1\geq 0$. We write that $(s_1,s'_1)\succcurlyeq (s_0,s'_0)$ if and only if $s_1\geq s_0$ and $s_1+s'_1\geq  s_0+s'_0$. The binary relation $\succcurlyeq$ is an order relation.
\end{definition}

\begin{lem}\label{lem:minimumexposant}
Any subset  $S\subseteq \N^2$ admits a unique infimum for the order relation $\succcurlyeq$. This infimum is not necessarily in $S$.
\end{lem}
We are now in position to define the \emph{order} of a multivariate series.
\begin{definition}\label{def:minimumexposant}
Let $f(x,y)=\sum_{s,s'\geq 0}f_{s,s'}x^sy^{s'}$ a multivariate series with positive radius. Let $\mathcal{D}(f)=\ac{(s,s')\in \N^2,\enspace f_{s,s'}\neq 0}$. We define the \textbf{order} of $f$ as the infimum of $\mathcal{D}(f)$ with respect to the partial order $\succcurlyeq$. It is referred as $\mathrm{ord}(f):=\inf \mathcal{D}(f)$.
\end{definition}

Our objective in this subsection is to control $\mathrm{ord}\pa{\kappa_{\Delta}}$ for the power series $\kappa_{\Delta}(x,y)$ defined by \eqref{eq:definition:kappa_Delta}.
For that purpose, we introduce a notion of {\it polynomial graph}, which are weighed graphs whose edges have weights $x$, $y$ or $1$. Consider such a polynomial graph $\mathcal{L}$ with edges $w_e\in \ac{x,y,1}$. Consider, if it exists (i.e if $\cL$ is connected), a spanning tree $\cT$ of $\cL$. We define the weight of $\cT$ as the monomial in $x$ and $y$ defined by $\prod_{e\in \cT}w_e$, and we write $deg(\cT)\in \N^2$ the \textit{degree} of this monomial, i.e. if $\prod_{e\in \cT}w_e=x^sy^{s'}$, then $deg(\cT):=(s,s')$.

\begin{definition}\label{def:minimumexposantgraph}
For $\mathcal{L}$ a polynomial graph, we define the \textbf{order} of $\cL$, and we write $\mathrm{ord}(\cL)$ the infimum (w.r.t. $\succcurlyeq$) of $deg(\mathcal{T})$ over all spanning trees $\mathcal{T}$, with the convention that $\inf \emptyset = (+\infty, +\infty)$.
\end{definition}

It turns out that the order of $\cL$ is achieved by a spanning tree $\mathcal{T}$, its there exists any. 
\begin{lem}\label{lem:minimumexposantgraph}
Let $\cL$ be a connected   polynomial graph. There exists a spanning tree $\mathcal{T}$  such that $\mathrm{ord}(\cL)=deg(\mathcal{T})$.
\end{lem}

In the following lemma, we adapt a result from \cite{feray2018} to the setup of multivariate series, which enables us to lower-bound the \textit{order} of a polynomial graph $\cL$ by the \textit{order} of well-chosen sub-graphs. Given two partition $\bV$ and $\bW$ of a set $V$, we write $\mathbf{W}\vee \mathbf{V}$ as the finest partition which is coarser than $\bW$ and coarser than $\bW$. In other words, this partition is obtained by merging the groups of $\bW$ and $\bV$ whenever they intersect.

Given $\bold{\Delta}=\ac{\Delta_1,\ldots, \Delta_m}$ a family of subsets of $[\ell]$, we write, for $i\in [m]$, $\Pi\pa{\Delta_i}$ the partition of $[\ell]$ made of $\Delta_i$ and singletons. The partition $\Pi(\Delta_1)\vee\ldots\vee\Pi(\Delta_m)$ is defined as previously as the finest partition with is coarser than all $\Pi(\Delta_i)$.

\begin{lem}\label{lem:decompmimimumexposantgraph}
Let $\cL=(V,E)$ by a polynomial graph. Let $\mathbf{V}=(V_1,\ldots, V_{|CC|})$ be the partition of the vertices induced by the connected components of $\cL$ while only keeping the edges with weights one.  Let $\bold{\Delta}=\ac{\Delta_1,\ldots, \Delta_m}$ be another partition of $V$ that satisfies  $\mathbf{V}\vee \Pi\pa{\Delta_1}\vee \ldots\vee \Pi\pa{\Delta_m}=\{V\}$. Then, we have  
$$\sum_{i=1}^m \mathrm{ord}\pa{\cL[\Delta_i]}\succcurlyeq \mathrm{ord}\pa{\cL}\enspace.$$
\end{lem}

We gather in next lemma some properties of the \textit{order} of multivariate power series that we will use throughout the proofs. 

\begin{lem}\label{lem:basicpropertiesorder}
Let $f,g$ be multivariate series with positive radius. We have;
\begin{enumerate}
\item $\mathrm{ord}(fg)\succcurlyeq \mathrm{ord}(f)+\mathrm{ord}(g)$;
\item $\mathrm{ord}(f+g)\succcurlyeq \inf\pa{\mathrm{ord}(f), \mathrm{ord}(g)}$, where the infimum is taken for the order $\succcurlyeq$;
\item If $f(0,0)=0$, then $\frac{1}{1+f}-1$ is a multivariate series with positive radius and $\mathrm{ord}(\frac{1}{1+f}-1)=\mathrm{ord}(f)$.
\end{enumerate}
\end{lem}

\subsection{Control of the order of $\kappa_{_{[\ell]}}$}\label{sec:orderkappa}

We are now in position to introduce the polynomial graph $\cL^*$ that will be useful to lower-bound the  \textit{order} of $\kappa_{[\ell]}(x,y)$. 
We recall that $I_{1},\ldots,I_{\ell},A_{1},\ldots,A_{q},B_{1},\ldots,B_{r}$ are subsets of $[n]$ involved in the definition of the 
mixed moments~\eqref{eq:definition_u_Delta} of $Z_{1},\ldots,Z_{\ell}$. 

\begin{definition}\label{def:Lstar}
Let $\cL^*$ be the polynomial graph with vertex set $[\ell]$ whose edges are defined as follows. For $(t,t')\in [\ell]$: 
   \begin{itemize}
    \item there is an edge with weight $y$ between $t$ and $t'$ if there exists  $i \in [r]$ such that both $I_{t}$ and $I_{t'}$ intersect $B_{i}$.
    \item there is an edge with weight $x$ between $t$ and $t'$, if there does not exist $i\in [r]$ such that both $I_{t}$ and $I_{t'}$ intersect $B_{i}$ and if there exists $j\in [q]$ such that that both  $I_{t}$ and $I_{t'}$ intersect $A_{j}$.
    \item Otherwise, there is no edge between $t$ and $t'$. 
    \end{itemize}
\end{definition}

    We shall lower-bound the \textit{order} of $\kappa_{\ell}(x,y)$ in term of  the \textit{order} of $\cL^*$.
    The following lemma, which is the counterpart  in our setting of Proposition 5.8  of \cite{feray2018}, ensures that, if one is able to lower-bound the \textit{order} of $P_{\Delta}[u]$, for all $\Delta\subseteq [\ell]$ of size at least $2$, with respect to the \textit{order} of some well chosen polynomial graph, then one is able to lower-bound the \textit{order} of $\kappa_{\Delta}$ for all $\Delta\subseteq R$. The proof is postponed to the end of the section.

    \begin{lem}\label{lem:ferayseries}
    Let $\pa{u_\Delta(x,y)}_{\Delta\subseteq [\ell]}$ be a sequence of multivariate series indexed by subsets of $[\ell]$. Let $\cL$ be a polynomial graph on $[\ell]$ whose edges are equal either to $x$, $y$, $0$ or $1$. Let us suppose that, for all $\Delta\subseteq [\ell]$ of size at least $2$, one has $\mathrm{ord}\pa{P_{\Delta}[u]-1}\succcurlyeq \mathrm{ord}\pa{\cL[\Delta]}$. Then, for all $\Delta\subseteq [\ell]$, one also has $$\mathrm{ord}(\kappa_{\Delta})\succcurlyeq \mathrm{ord}\pa{\cL[\Delta]}+\sum_{i\in \Delta}\mathrm{ord}(u_{\{i\}})\enspace.$$
    \end{lem}

    Hence, to lower bound the order of $\kappa_{\Delta}$, we mostly have to lower bound the order of $P_{\Delta}[u]-1$ in terms of that of $\cL^*[\Delta]$, and then in turn to control the order of $\cL^*[\Delta]$. 
    The following lemma, proved in Section \ref{prf:quazifactogeneral}, lower-bounds the \textit{order} of $P_{\Delta}[u]-1$, for all $\Delta$ of size at least $2$.

    \begin{lem}\label{lem:quazifactogeneral}
    For all $\Delta\subseteq [\ell]$ of size at least $2$, one has 
    $$\mathrm{ord}\pa{P_{\Delta}[u]-1}\succcurlyeq \mathrm{ord}\pa{\cL^*[\Delta]}\enspace.$$
    \end{lem}

    Combining Lemma \ref{lem:ferayseries} and Lemma \ref{lem:quazifactogeneral}, we get that, for all $\Delta\subseteq [\ell]$ of size at least $2$, 
    $$\mathrm{ord}\pa{\kappa_{\Delta}}\succcurlyeq \mathrm{ord}\pa{\cL^*[\Delta]}+\sum_{i\in \Delta}\mathrm{ord}\pa{u_{\{i\}}}\enspace.$$
    When $\Delta$  is a singleton that is $\Delta=\ac{i}$, we directly have $\kappa_{\Delta}=u_{\{i\}}$ and so the above inequality still holds. 
    Since $\mathrm{ord}\pa{u_{\{i\}}}\in\N^2$, we have  for all non-empty $\Delta\subseteq R$, 
    \begin{equation}\label{eq:kappa2L}
    \mathrm{ord}\pa{\kappa_{\Delta}}\succcurlyeq \mathrm{ord}\pa{\cL^*[\Delta]}\enspace\enspace.
    \end{equation}
    It remains to lower-bound $\mathrm{ord}\pa{\cL^*}=\mathrm{ord}\pa{\cL^*[[\ell]]}$. We recall that the graph $\mathcal{B}$ defined in the statement of Lemma~\ref{lem:upperboundgeneralcumulant}, is a graph on the vertex set $[\ell]$ with an edge between $t$, $t'$ in $[\ell]$ if and only if there exists $i \in [r]$ such that both $I_{t}$ and $I_{t'}$ intersects $B_{i}$ and that $|cc(\cB)|$ stands for the number of connected components of this graph.

    \begin{lem}\label{lem:quasifactographgeneral}
    Writing $\mathrm{ord}\pa{\cL^*}=(d_{[\ell]}, d'_{[\ell]})$, we have 
    $$d_{[\ell]}\geq |cc\pa{\mathcal{B}}|-1\quad \quad \text{ and } d_{[\ell]}+d'_{[\ell]}\geq \ell-1\enspace.$$
    In particular, $\mathrm{ord}\pa{\cL^*}\succcurlyeq \pa{|cc\pa{\mathcal{B}}|-1, \ell-|cc\pa{\mathcal{B}}|}$.
    \end{lem}

Combining  \eqref{eq:kappa2L} with Lemma \ref{lem:quasifactographgeneral}  leads to o lower-bound on the order of $\kappa_{[\ell]}$
    \begin{equation}\label{eq:lower:bound:kappa_R}
    \mathrm{ord}\pa{\kappa_{[\ell]}}\succcurlyeq \pa{|cc\pa{\cB}|-1, \ell-|cc\pa{\cB}|}\enspace.
    \end{equation}
    
\subsection{Non-asymptotic control of the coefficients}\label{sec:controlcoeff}

 In order to deduce from \eqref{eq:lower:bound:kappa_R} a non-asymptotic upper-bound of $\kappa_{[\ell]}(x_0, y_0)$, it is sufficient to upper-bound the absolute value of the coefficients of the decomposition of $\kappa_{[\ell]}(x,y)$ as a multivariate series. We recall that, for $[\ell]\subseteq R$, we have 
    \begin{align*}
    u_\Delta(x,y)=&\eta^{|I_\Delta|}\prod_{j=1}^{q}\frac{1}{1-x}\times\ldots\times \frac{1}{1-(|I_\Delta\cap A_{j}|-1)x}
   \times\prod_{i=1}^{r}(1-y)\times\ldots\times \pa{1-\pa{|I_\Delta \cap B_{i}|-1}y}\enspace.
    \end{align*}
 In order to upper-bound the coefficients of the series $\kappa_{[\ell]}(x,y)$, we shall first do so for the $u_{\Delta}(x,y)$'s. In the following, given $f$ a multivariate series and $d,d'\geq 0$, we write $[d,d']_f$ the coefficient of $x^dy^{d'}$ in the decomposition of $f$. For any $d$ ,$d'\geq 0$, and $j=1,\ldots, q$, we define
\begin{align*}
g_{1,j}(x,y)&:= \frac{1}{1-x}\times\ldots\times \frac{1}{1-(|I_\Delta\cap A_{j}|-1)x}\ ;\\
g_{1}(x,y)&:= \prod_{j=1}^{q}g_{1,j}(x,y)  \ ; \\
g_{2}(x,y)&:= \prod_{i=1}^{r}(1-y)\times\ldots\times \pa{1-\pa{|I_\Delta \cap B_{i}|-1}y} \ ,
\end{align*}
so that  $ [d,d']_{u_{\Delta}} =\eta^{|I_\Delta|}[d,0]_{g_1} [0,d']_{g_2}$. One the one hand, for any $j=1,\ldots, q$ and any $d\geq 0$, we have 
    \begin{align*}
    [d,0]_{g_1,j}
    =&\underset{d_1+\ldots+d_{|I_\Delta\cap A_{j}|-1}=d}{\sum_{d_1,\ldots, d_{|I_\Delta\cap A_{j}|-1}\geq 0}}\prod_{a=1}^{|I_\Delta\cap A_{j}|-1}a^{d_a}
    \leq\pa{1+\ldots+(|I_\Delta\cap A_{j}|-1)}^d\\ 
    \leq& \pa{\frac{|I_\Delta\cap A_{j}|^2}{2}}^d\enspace .
    \end{align*}
We deduce that 
    \begin{align*}
    [d,0]_{g_1}=&\underset{d_1+\ldots+d_q=d}{\sum_{d_1,\ldots, d_q\geq 0}}\prod_{j=1}^{q}[d_{j}, 0]_{g_1,j}\\
    \leq&\underset{d_1+\ldots+d_q=d}{\sum_{d_1,\ldots, d_q\geq 0}}\prod_{j=1}^{q}\pa{\frac{|I_\Delta\cap A_{j}|^2}{2}}^{d_j}\\
    \leq& \pa{\sum_{j=1}^q\frac{|I_\Delta\cap A_{j}|^2}{2}}^d \leq  \pa{\frac{|I_\Delta|^2}{2}}^d \ ,
    \end{align*}
    where the last equality stems from the disjunction of the $A_{i}$.
Then, turning our attention to $g_2$, we have 
    \begin{align*}
    \left|[0,d']_{g_2}\right|\leq& \binom{\sum_{i=1,\ldots, r}|I_\Delta\cap B_{i}|-1}{d'}|I_{\Delta}|^{d'}
    \leq  |I_{\Delta}|^{2d'}\enspace . 
    \end{align*}
    We conclude that 
    \begin{equation}\label{eq:boundcoeffseriemomentclustering1}
    |[d,d']_{u_\Delta}|\leq \eta^{|I_\Delta|}|I_\Delta|^{2d+2d'}\enspace. 
    \end{equation}
    Then, we consider any partition $\boldsymbol{\Delta}=(\Delta_1,\ldots,\Delta_{|\boldsymbol{\Delta}|} )$ of $[\ell]$. Define $h_{\boldsymbol{\Delta}}= \prod_{a=1}^{{|\boldsymbol{\Delta}|}} u_{\Delta_a}$. 
    Recall that the $I_t$'s defined for \eqref{eq:formula:general_momodent} have disjoint support so that the $I_{\Delta_a}$'s are disjoint and $\sum_{a=1}^{|\boldsymbol{\Delta}|}|I_{\Delta_a}|= L$.  Hence, we deduce from \eqref{eq:boundcoeffseriemomentclustering1} that
    \begin{align*}
    \left|[d,d']_{h_{\boldsymbol{\Delta}}}\right|&=\left|\underset{d_1+\ldots+d_{|\boldsymbol{\Delta}|}=d}{\sum_{d_1,\ldots,d_{|\boldsymbol{\Delta}|}\geq 0}}\quad \underset{d'_1+\ldots+d'_{|\boldsymbol{\Delta}|}=d'}{\sum_{d'_1,\ldots,d'_{|\boldsymbol{\Delta}|}\geq 0}}\prod_{a=1}^{|\boldsymbol{\Delta}|}[d_a, d'_a]_{u_{\Delta_a}} \right|\\
    &\leq  \eta^{L}\underset{d_1+\ldots+d_{|\boldsymbol{\Delta}|}=d}{\sum_{d_1,\ldots,d_{|\boldsymbol{\Delta}|}\geq 0}}\quad \underset{d'_1+\ldots+d'_{|\boldsymbol{\Delta}|}=d'}{\sum_{d'_1,\ldots,d'_{|\boldsymbol{\Delta}|}\geq 0}}\prod_{a=1}^{|\boldsymbol{\Delta}|}|I_{\Delta_a}|^{2d_{a}+2d'_{a}}\\
    &\leq   \eta^{L}\pa{\sum_{a=1}^{|\boldsymbol{\Delta}|}|I_{\Delta_a}|}^{2d+2d'}=  \eta^{L}L^{2d+2d'}\enspace.
    \end{align*}
    Then, combining this with a triangular inequality and the Möbius formula (Lemma \ref{lem:mobiusformula}), we obtain the following bound for the coefficients of $\kappa_{[\ell]}$. 
    \begin{equation}\label{eq:boundcoeffseriemomentclustering2}
    \left|[d,d']_{\kappa_{[\ell]}}\right|\leq \eta^{L}L^{2d+2d'}\ell^{2\ell}\enspace.
    \end{equation}

 \subsection{Conclusion}
Combining Equation \eqref{eq:boundcoeffseriemomentclustering2} together with \eqref{eq:lower:bound:kappa_R} leads us to 
    \begin{align*}
    \left|\kappa_{[\ell]}\pa{x_0,y_0}\right|\leq&\ell^{2\ell}\eta^{L}\underset{d+d'\geq \ell-1}{\sum_{d\geq cc\pa{\cB}-1}}L^{2d+2d'}x_0^dy_0^{d'}\\
    &\leq  \ell^{2\ell}\eta^{L}\sum_{d\geq cc\pa{\cB}-1}\sum_{d'\geq \ell-1-d}\pa{L^2 x_0}^d\pa{L^2y_0}^{d'}\\
    &\leq 2\ell^{2\ell}\eta^{L}\sum_{d\geq cc\pa{\cB}-1}\pa{L^2 x_0}^d \pa{L^2y_0}^{\ell-1-d}\\
    &\leq 2\ell^{2\ell}\eta^{L}\pa{L^2y_0}^{\ell-1}\sum_{d\geq cc\pa{\cB}-1}\pa{x_0/y_0}^d\\
    &\leq  4\ell^{2\ell}\eta^{L}\pa{L^2y_0}^{\ell-1}\pa{\frac{x_0}{y_0}}^{cc\pa{\cB}-1}\enspace,
    \end{align*}
    where the third and fifth inequalities are satisfied whenever $L^2y_0\leq\frac{1}{2}$ and $x_0\leq\frac{1}{2}y_0$. We have shown
    $$\left|\cumul\pa{\pa{Z_t}_{t\in [\ell]}}\right|\leq 4\ell^{2\ell}\eta^{L}\pa{L^2y_0}^{\ell-1}\pa{\frac{x_0}{y_0}}^{cc\pa{\cB}-1}\enspace.$$
This concludes the proof of Lemma~\ref{lem:upperboundgeneralcumulant}.

\subsection{Proof of Lemma \ref{lem:ferayseries}}\label{prf:ferayseries}

The proof of Lemma \ref{lem:ferayseries} build on the ideas of the proof of Proposition 5.8 of \cite{feray2018}, which are adapted to the framework of multivariate series.

Let us first reduce the problem to the case where $u_{\{i\}}= 1$. Let us define the family of power series
$$w_\Delta(x,y)=\frac{u_\Delta (x,y)}{\prod_{i\in \Delta}u_{\{i\}}(x,y)}.$$ 
We have $P_\Delta[w]=P_{\Delta}[u]$, for all $\Delta\subseteq [\ell]$ of size at least $2$. Defining  $\kappa_{\Delta}[w]$ as the counterpart of $\kappa_{\Delta}$ when the function $u$ is replaced by $w$, we get  
$$\kappa_{\Delta}(x,y)=\pa{\prod_{i\in \Delta}u_{\{i\}}(x,y)}\kappa_{\Delta}[w](x,y)\enspace .$$
So, if we prove that $\mathrm{ord}(\kappa_{\Delta}[w])\succcurlyeq \mathrm{ord}(\cL[\Delta])$ (see Definition \ref{def:minimumexposantgraph} for the order of a power serie), then by Lemma~\ref{lem:basicpropertiesorder}
$$\mathrm{ord}(\kappa_{\Delta})\succcurlyeq \mathrm{ord}(\cL[\Delta])+\sum_{i\in\Delta}\mathrm{ord}(u_{\ac{i}}). $$
Hence, in the following, we assume that $u_{\{i\}}=1$ for all $i\in [\ell]$. By assumption,  for all $\Delta$ of size at least $2$, one has $\mathrm{ord}\pa{P_\Delta (u)-1}\succcurlyeq \mathrm{ord}\pa{\cL[\Delta]}$ and we shall lower-bound $\mathrm{ord}\pa{\kappa_{\Delta}}$ for all $\Delta\subseteq [\ell]$. Without loss of generality, it is sufficient to do so for $\Delta=[\ell]$.

For $\Delta\subseteq [\ell]$, define $R_\Delta(x,y)=P_\Delta (u)(x,y)-1$. By assumption, we have $\mathrm{ord}\pa{R_\Delta}\succcurlyeq \mathrm{ord}\pa{\cL[\Delta]}$. Let us fix a partition $\bW=(W_1,\ldots, W_s)$ of $[\ell]$, and let us consider $W\in \bW$. Since $u_\emptyset=1$, 
we have
\begin{align*}
\prod_{\Delta\subseteq W}P_{\Delta}[u](x,y)=&\prod_{\Delta\subseteq W}\prod_{\delta\subseteq \Delta}\pa{u_{\delta}(x,y)}^{\pa{-1}^{|\Delta|-|\delta|}}\\
=&\prod_{\delta\subseteq W}\pa{u_{\delta}(x,y)}^{\sum_{\delta\subseteq \Delta \subseteq W}(-1)^{|\Delta|-|\delta|}}\\
=&u_{W}(x,y)\enspace,
\end{align*}
where the last inequality comes from the fact that, whenever $\delta\subsetneq W$, $\sum_{\delta\subseteq \Delta \subseteq W}(-1)^{|\Delta|-|\delta|}=\pa{1-1}^{|W|-|\delta|}=0$. In turn,  expanding $P_{\Delta}(x,y)=1+R_{\Delta}(x,y)$ and using the fact that for all $\ac{i}$, $P_{\ac{i}}[u](x,y)=1$,
we get,
$$u_W(x,y)=\underset{|\Delta|\geq 2}{\prod_{\Delta\subseteq W}}\pa{1+R_\Delta(x,y)}=\sum_{\ac{\Delta_1,\ldots, \Delta_m}}R_{\Delta_1}(x,y)\ldots R_{\Delta_m}(x,y)\enspace,$$
where the sum runs over all finite sets of distinct subsets of $W$ of size at least $2$. Therefore, 
\begin{equation}\label{eq:prod-uWi}
\prod_{i=1}^su_{W_i}(x,y)=\sum_{\ac{\Delta_1,\ldots,\Delta_m}}R_{\Delta_1}(x,y)\ldots R_{\Delta_m}(x,y)\enspace,
\end{equation}
where the sum is taken over all finite sets of distinct subsets (but not necessarily disjoint) of $[\ell]$ of size at least $2$, such that each $\Delta_i$ is contained in a group of the partition $\bW$. Let us fix such a set $\bold{\Delta}=\ac{\Delta_1,\ldots,\Delta_m}$ and, for $i\in [m]$, let us write $\Pi\pa{\Delta_i}$ the partition of $[\ell]$ containing $\Delta_{i}$ and singletons. We recall that $\Pi\pa{\Delta_1}\vee\ldots \vee  \Pi\pa{\Delta_m}$ is the partition of $[\ell]$ defined by; for $t,t'\in [\ell]$, $t$ and $t'$ are in the same group of $\Pi\pa{\Delta_1}\vee\ldots \vee  \Pi\pa{\Delta_m}$ if and only if there exist $t_1,\ldots, t_h$ with $t_1=t$, $t_h=t'$ and, for all $h'\leq h-1$, there exists $i\in [m]$ such that $t_h$ and $t_{h+1}$ are in the same group of the partition $\Pi\pa{\Delta_i}$.  Since each $\Delta_i$ is contained in a group of $\bW$, we have that each group of $\Pi\pa{\Delta_1}\vee\ldots \vee  \Pi\pa{\Delta_m}$ must be contained in a group of $\bW$. Consequently, $\bW\vee\Pi\pa{\Delta_1}\vee\ldots \vee  \Pi\pa{\Delta_m}=\bW$. 

In the following, given $\bold{\Delta}=\ac{\Delta_1,\ldots,\Delta_m}$ a set of distinct subsets of $[\ell]$ of size at least $2$, we denote $\cP_{\bold{\Delta}}([\ell])$ the set of all partitions $\bW$ of $[\ell]$ satisfying that, for all $i\in [m]$, $\Delta_i$ is contained in a group of $\bW$. For all $\bW\in \cP_{\bold{\Delta}}([\ell])$, we have $\bW\vee\Pi\pa{\Delta_1}\vee\ldots \vee  \Pi\pa{\Delta_m}=\bW$.
Using the Möbius formula stated in Lemma \ref{lem:mobiusformula}, together with \eqref{eq:prod-uWi}, and after reversing the summation order, we get
\begin{equation}\label{eq:cumulantsumRDelta}
\kappa_{[\ell]}(x,y)=\underset{|\Delta_i|\geq 2}{\underset{\Delta_i\subseteq [\ell]}{\sum_{\bold{\Delta}=\ac{\Delta_1,\ldots, \Delta_m}}}}R_{\Delta_1}(x,y)\ldots R_{\Delta_m}(x,y)\pa{\sum_{\bW\in \mathcal{P}_{\bold{\Delta}}([\ell])}m(\bW)}\enspace,
\end{equation}
where $m(\bW)=(-1)^{|\bW|-1}(|\bW|-1)!$ is the Möbius function.
Next lemma prunes some of the terms of~\eqref{eq:cumulantsumRDelta}.

\begin{lem}\label{lem:cumulantsumRDelta}
Let $\bold{\Delta}=\ac{\Delta_1,\ldots, \Delta_m}$ a family of distinct subsets of $[\ell]$. If $\Pi(\Delta_1)\vee\ldots\vee \Pi(\Delta_m)\neq \ac{[\ell]}$, we have $$\sum_{\bW\in \mathcal{P}_{\bold{\Delta}}([\ell])}m(\bW)=0\enspace.$$
\end{lem}

\begin{proof}[Proof of Lemma \ref{lem:cumulantsumRDelta}]
Let us consider $\bW^{(0)}=\Pi(\Delta_1)\vee\ldots\vee \Pi(\Delta_m)$, and let us write  $\bW^{(0)}=W^{(0)}_1,\ldots, W^{(0)}_{|\bW^{(0)}|}$, with $|\bW^{(0)}|\geq 2$. By definition of $\bW^{(0)}$, we have $\mathcal{P}_{\bold{\Delta}}([\ell])=\mathcal{P}_{\bW^{(0)}}([\ell])$ and so it is sufficient to prove that $\sum_{\bW\in \mathcal{P}_{\cW^{(0)}}([\ell])}m(\bW)=0$, when $\bW^{(0)}$ is a partition of $[\ell]$ different from the trivial partition $\ac{[\ell]}$. We can construct a bijection $\phi$ from all partitions $\bG$ of $[|\bW^{(0)}|]$ to the set $\mathcal{P}_{\cW^{(0)}}([\ell])$ defined by $\phi(\bG)=\ac{\cup_{i\in G}W^{(0)}_i,\enspace G\in \bG}$. This bijection preserves the Möbius function in the sense that $m(\phi(\bG))=m(\bG)$ and so $$\sum_{\bW\in \mathcal{P}_{\cW^{(0)}}([\ell])}m(\bW)=\sum_{\bG\in \cP([|\bW^{(0)}|])}m(\bG)=\cumul\pa{1,\ldots,1}=0\enspace,$$
where the last cumulant is the cumulant of $|\bW^{(0)}|$ copies of the deterministic variable $1$, which is equal to $0$ as soon as $|\bW^{(0)}|\geq 2$ (see Lemma \ref{lem:independentcumulant}). This concludes the proof of Lemma \ref{lem:cumulantsumRDelta}.
\end{proof}

Combining Lemma \ref{lem:cumulantsumRDelta} and Equality \eqref{eq:cumulantsumRDelta}, we get 
\begin{equation}\label{eq:cumulantsumRDelta2}
\kappa_{[\ell]}(x,y)=\underset{|\Delta_i|\geq 2}{\underset{\Pi(\Delta_1)\vee\ldots\vee \Pi(\Delta_m)=\ac{[\ell]}}{\sum_{\bold{\Delta}=\ac{\Delta_1,\ldots, \Delta_m}}}}R_{\Delta_1}(x,y)\ldots R_{\Delta_m}(x,y)
\enspace.
\end{equation}
Let us now consider $\bold{\Delta}=\ac{\Delta_1,\ldots, \Delta_m}$ a set of subsets of $[\ell]$ such that $\Pi\pa{\Delta_1}\vee\ldots \vee  \Pi\pa{\Delta_m}=\ac{[\ell]}$. We consider $\bV$ the partition of $[\ell]$ induced by the connected components of $\cL\<1\>$, which is the graph keeping only the edges of weight $1$ of $\cL$. Since $\Pi\pa{\Delta_1}\vee\ldots \vee  \Pi\pa{\Delta_m}=\ac{[\ell]}$, one necessarily has $\bV\vee\Pi\pa{\Delta_1}\vee\ldots \vee  \Pi\pa{\Delta_m}=\ac{[\ell]}$. Hence, using Lemma \ref{lem:decompmimimumexposantgraph}, we get
$$\sum_{i=1}^m \mathrm{ord}(\cL[\Delta_i])\succcurlyeq \mathrm{ord}(\cL)\enspace.$$
We recall our assumption that  $\mathrm{ord}(R_{\Delta_i})\succcurlyeq \mathrm{ord}(\cL[\Delta_i])$. Lemma \ref{lem:basicpropertiesorder} ensures that, 
$$\mathrm{ord}\pa{R_{\Delta_{1}}\ldots R_{\Delta_m}}\succcurlyeq \sum_{i=1}^m \mathrm{ord}(\cL[\Delta_i])\succcurlyeq \mathrm{ord}(\cL)\enspace.$$
Thus, all the non-zero terms in the sum \eqref{eq:cumulantsumRDelta} have their \textit{order} higher than $\mathrm{ord}(\cL)$ with respect to the ordering $\succcurlyeq$. Invoquing a last time Lemma \ref{lem:basicpropertiesorder}, we get the claimed inequality
$$\mathrm{ord}(\kappa_{[\ell]})\succcurlyeq \mathrm{ord}(\cL)\enspace.$$

\subsection{Proof of Lemma \ref{lem:quazifactogeneral}}\label{prf:quazifactogeneral}

The power series $u_{\Delta}$ involves multiple terms that are alike factorial terms. For proving Lemma~\ref{lem:quasifactographgeneral}, we shall apply several times the following result which which is variation of Proposition 5.10 in~\cite{feray2018}. Its proof is postponed to the end of the subsection. 

\begin{lem}\label{lem:factorielseries}
Let $a_1, \ldots, a_\ell$ be non-negative integers, and, for $\Delta\subseteq [\ell]$, let us define 
\begin{align*}
v_\Delta(x,y)&:=\pa{1-x}\ldots \pa{1-\pa{\sum_{i\in \Delta}a_i-1}x}\enspace ;\\ w_\Delta(x,y)&:=\frac{1}{\pa{1-x}\ldots \pa{1-\pa{\sum_{i\in \Delta}a_i-1}x}}\enspace \ , 
\end{align*}
with the convention, if $\sum_{i\in \Delta} a_i =0$ or $\sum_{i\in \Delta} a_i =1$, $v_\Delta(x,y)=w_\Delta(x,y)=1$. Recall the operator $P_{\Delta}$ defined in~\eqref{eq:definition_P_Delta}.
 Then, for all $\Delta$ of size at least $2$, we have
$$\mathrm{ord}\pa{P_{\Delta}[v]-1}\succcurlyeq (|\Delta|-1, 0)\enspace\ ;\quad \quad  \mathrm{ord}\pa{P_\Delta[w]-1}\succcurlyeq (|\Delta|-1, 0)\enspace.$$
\end{lem}

We shall also use next lemma, whose proof is also postponed to the end of this subsection, and which enables us to decompose $u$ as a product of power series.

\begin{lem}\label{lem:decomppowerseries}
Let $\pa{v_\Delta}_{\Delta\subseteq [\ell]}$ and $\pa{w_\Delta}_{\Delta\subseteq [\ell]}$ to families of power series in $x$ and $y$ indexed by subsets of $[\ell]$. Suppose that there exists $(s,s')\in \N^2$ such that, for all $\Delta\subseteq [\ell]$, we have $\mathrm{ord}\pa{P_{\Delta}[v]-1}\succcurlyeq (s,s')$ and $\mathrm{ord}\pa{P_\Delta[w]-1}\succcurlyeq (s,s')$. Then, for all $\Delta\subseteq [\ell]$, $\mathrm{ord}\pa{P_{\Delta}[vw]-1}\succcurlyeq (s,s')$.
\end{lem}

Now, we have all the ingredients to prove Lemma \ref{lem:quazifactogeneral}. Let $\Delta\subseteq [\ell]$ of size at least $2$. We seek to lower bound $\mathrm{ord}\pa{P_{\Delta}[u]-1}$, where $$P_{\Delta}[u](x,y)=\prod_{\delta\subseteq \Delta}\pa{u_\delta(x,y)}^{(-1)^{|\Delta|-|\delta|}}\enspace,$$
where we recall that, for all $\Delta\subseteq [\ell]$,
\begin{align*}
u_\Delta(x,y)=&\eta^{|I_\Delta|}\prod_{j=1}^q\frac{1}{1-x}\times\ldots\times \frac{1}{1-(|I_\Delta\cap A_{j}|-1)x}\\
&\times\prod_{i=1}^r(1-y)\times\ldots\times \pa{1-\pa{|I_\Delta \cap B_{i}|-1}y}\enspace.
\end{align*}
Let us define, for $j=1,\ldots, q$, $$w^{(j)}_\Delta(x,y):=\frac{1}{1-x}\times\ldots\times \frac{1}{1-(|I_\Delta\cap A_{j}|-1)x}\enspace,$$
and, for $i\in [r]$,
$$v^{(i)}_\Delta(x,y):=(1-y)\times\ldots\times \pa{1-\pa{|I_\Delta \cap B_{i}|-1}y}\enspace,$$
and $z_{\Delta}(x,y):=\eta^{|I_\Delta|}$.

For all $\Delta\subseteq [\ell]$, we have $u_\Delta(x,y)=z_{\Delta}(x,y)\prod_{j=1}^qw^{(j)}_\Delta(x,y)\prod_{i=1}^r v^{(i)}_\Delta(x,y)$. In the light of Lemma \ref{lem:decomppowerseries}, in order to prove that $\mathrm{ord}\pa{P_{\Delta}[u]-1}\succcurlyeq \mathrm{ord}(\cL^*[\Delta])$, it is sufficient to prove that we have $\mathrm{ord}\pa{P_\Delta[w^{(j)}]-1}\succcurlyeq \mathrm{ord}(\cL^*[\Delta])$, for all $j=1,\ldots, q$, and $\mathrm{ord}\pa{P_\Delta[v^{(i)}]-1}\succcurlyeq \mathrm{ord}(\cL^*[\Delta])$, for all $i\in [r]$, and  $\mathrm{ord}\pa{P_\Delta[z]-1}\succcurlyeq \mathrm{ord}(\cL^*[\Delta])$.

\paragraph{The term z}. Let us first deal with the constant term $P_\Delta[z]$. It is clear that for all $\delta\subseteq \Delta$, we have $z_{\delta}=\prod_{t'\in \delta}z_{\ac{t'}}$. Therefore, using Lemma \ref{lem:independentquasifacto}, we deduce that as soon as $\Delta$ is of size at least $2$, $P_\Delta[z]=1$ and thus $\mathrm{ord}\pa{P_\Delta[z]-1}=+\infty$. Hence, it is clear that $\mathrm{ord}\pa{P_\Delta[z]-1}\succcurlyeq \mathrm{ord}\pa{\cL^*[\Delta]}$.

\paragraph{The term $w^{(j)}$.} Let us now deal with the term $P_\Delta[w^{(j)}]$. If there exists $t\in \Delta$ such that $I_{t}\cap A_{j}=\emptyset$, we can easily check using Lemma \ref{lem:independentquasifacto} that $P_\Delta[w^{(j)}]=1$ and so $\mathrm{ord}\pa{P_\Delta[w^{(j)}]-1}\succcurlyeq \mathrm{ord}(\cL^*[\Delta])$. Let us now suppose that for all $t\in \Delta$, we have $I_{t}\cap A_{j}\neq \emptyset$. Note that this implies that $\cL^*[\Delta]$ is a complete graph, and is in particular connected (we recall that $\cL^*$ is defined in Definition~\ref{def:Lstar}, page~\pageref{def:Lstar}). Since any spanning tree of $\cL^*[\Delta]$ has exactly $|\Delta|-1$ edges, we derive that  $\mathrm{ord}(\cL^*[\Delta])\preccurlyeq \pa{|\Delta|-1, 0}$.

For all $\delta\subseteq \Delta$, we remark that
$$w^{(j)}_\delta(x,y)=\frac{1}{(1-x)\ldots (1-(\sum_{t\in \delta}|I_{\ac{t}}\cap A_{j}|-1)x)}\enspace.$$
 Lemma \ref{lem:factorielseries} then ensures that, for all $\Delta$ of size at least $2$,  $\mathrm{ord}\pa{P_\Delta[w^{(j)}]-1}\succcurlyeq \pa{|\Delta|-1, 0}$. We have proved that $\mathrm{ord}\pa{P_\Delta[w^{(j)}]-1}\succcurlyeq \mathrm{ord}(\cL^*[\Delta])$.

\paragraph{The term $v^{(i)}$.} Finally, we shall prove that, for all $i \in [r]$, $\mathrm{ord}\pa{P_\Delta[v^{(i)}]-1}\succcurlyeq \mathrm{ord}\pa{\cL^*[\Delta]}$. If there exists $t\in \Delta$ such that $I_{t}\cap B_{i}=\emptyset$, then, using Lemma \ref{lem:independentquasifacto}, we deduce $P_\Delta[v^{(i)}]=1$ and so $\mathrm{ord}\pa{P_\Delta[v^{(i)}]-1}\succcurlyeq \mathrm{ord}(\cL^*[\Delta])$. Let us now suppose that all  $t \in \Delta$, we have $B_{i}\cap I_t\neq \emptyset$. In particular, this implies that $\cL^*[\Delta]$ is a complete graph with edges with weight $y$. Hence, $\mathrm{ord}(\cL^*[\Delta])=\pa{0, |\Delta|-1}$. For all $\delta\subseteq\Delta$, $v^{(i)}_\delta=\pa{1-y}\ldots \pa{1-\pa{\sum_{t'\in \delta}|I_{t'}\cap B_{i}|-1}y}$. So , using Lemma \ref{lem:factorielseries}, we deduce that $\mathrm{ord}\pa{P_{\Delta}[w]-1}\succcurlyeq (0,|\Delta|-1)=\mathrm{ord}\pa{\cL^*[\Delta]}$.

Combining those three points together concludes the proof of Lemma~\ref{lem:quazifactogeneral}.

\begin{proof}[Proof of Lemma \ref{lem:factorielseries}]

Let $a_1,\ldots, a_\ell$ be a set of non-negative integers. We remark that $v_{\Delta}=1/w_{\Delta}$ for any subset $\Delta$, which in turn implies $P_{\Delta}[v]-1=\frac{1}{P_{\Delta}[w]}-1=\frac{1}{1+(P_{\Delta}[w]-1)}-1$. Since $\pa{P_{\Delta}[v]-1(0,0)}=\pa{P_{\Delta}[w]-1} (0,0)=0$, we get, using Lemma \ref{lem:basicpropertiesorder}, that $\mathrm{ord}\pa{P_{\Delta}[v]-1}\succcurlyeq \mathrm{ord}\pa{P_{\Delta}[w]-1}$ and, similarly, $\mathrm{ord}\pa{P_{\Delta}[w]-1}\succcurlyeq \mathrm{ord}\pa{P_{\Delta}[v]-1}$. We deduce $\mathrm{ord}\pa{P_{\Delta}[w]-1}=\mathrm{ord}\pa{P_{\Delta}[v]-1}$. Hence, for $\Delta\subseteq [\ell]$, having $\mathrm{ord}\pa{P_{\Delta}[v]-1}\succcurlyeq \pa{|\Delta|-1,0}$ is equivalent to having $\mathrm{ord}\pa{P_{\Delta}[w]-1}\succcurlyeq \pa{|\Delta|-1,0}$. Hence, we only need to prove the result for $P_{\Delta}[v]$.

We argue by induction on $\ell$ and then $a_\ell$ that this holds for all $\Delta\subseteq [\ell]$. For $\ell=1$, the result is trivial. Consider $\ell>1$, and assume that for all $\Delta\subsetneq [\ell]$, we have $\mathrm{ord}\pa{P_{\Delta}[v]-1}\succcurlyeq \pa{|\Delta|-1,0}$. Let us prove that $\mathrm{ord}\pa{P_{[\ell]}[v]-1}\succcurlyeq (\ell-1,0)$.

In the following, for $k\geq 0$, we define $v_{\Delta}[k]$  as 
$$v_\Delta[k](x,y):=\pa{1-x}\ldots \pa{1-\pa{\sum_{i\in \Delta\setminus \ac{\ell}}a_i+k\1\ac{\ell\in \Delta}-1}x}\enspace.$$
In particular, we have  $v_\Delta=v_\Delta[a_\ell]$. Hence, we only need to prove that $\mathrm{ord}\pa{P_{[\ell]}(v[k])-1}\succcurlyeq (\ell-1,0)$ for all non-negative integers $k$ and any $(a_1,\ldots, a_{\ell-1})$.

Let us first consider the case $k=0$. Then, for any $\Delta\subseteq [\ell-1]$, 
$$v_\Delta[k](x,y)=v_{\Delta\cup\ac{\ell}}[k](x,y)\enspace.$$
By definition~\eqref{eq:definition_P_Delta} of $P_{\Delta}$, this implies that for all $x,y$, $P_{[\ell]}(v[a_\ell])(x,y)=1$. Hence, we have $\mathrm{ord}\pa{P_{[\ell]}(v[a_\ell])-1}\succcurlyeq (\ell-1,0)$.

Now assume for some $k>0$, we have $\mathrm{ord}\pa{P_{[\ell]}(v[k])-1}\succcurlyeq (\ell-1,0)$ and let us prove that $\mathrm{ord}\pa{P_{[\ell]}(v[k+1])-1}\succcurlyeq (\ell-1,0)$. We observe that, for $\ell\in \Delta$, 
\begin{align*}
v_\Delta[k+1]=&(1-x)\ldots\pa{1-\pa{\sum_{i\in \Delta\setminus \ac{l}}a_i+k}x}\\
=&(1-x)\ldots\pa{1-\pa{\sum_{i\in \Delta\setminus \ac{\ell}}a_i+k-1}x}\times \pa{1-\pa{\sum_{i\in \Delta\setminus \ac{\ell}}a_i+k}x}\\
=&v_{\Delta}[k]\pa{1-\pa{\sum_{i\in \Delta\setminus \ac{\ell}}a_i+k}x} \enspace . 
\end{align*}
On the other hand, if $\ell\notin \Delta$, we have $v_\Delta[k+1]=v_{\Delta}[k]$.
Thus, we deduce that 
$$P_{[\ell]}[v[k+1]]=\pa{\underset{\ell\in \Delta}{\prod_{\Delta\subseteq [\ell]}}\pa{1-\pa{\sum_{i\in \Delta\setminus \ac{\ell}}a_i+k}x}^{(-1)^{\ell-|\Delta|}}}P_{[\ell]}[v[k]]\enspace . $$
Let us define the multivariate series
\begin{align*}
Q(x):=&\underset{\ell\in \Delta}{\prod_{\Delta\subseteq [\ell]}}\pa{1-\pa{\sum_{i\in \Delta\setminus \ac{\ell}}a_i+k}x}^{(-1)^{\ell-|\Delta|}}
=\underset{\ell\in \Delta}{\prod_{\Delta\subseteq [\ell]}}x^{(-1)^{\ell-|\Delta|}}\pa{\frac{1}{x}-\pa{\sum_{i\in \Delta\setminus \ac{\ell}}a_i+k}}^{(-1)^{\ell-|\Delta|}}\\
=&\underset{\ell\in \Delta}{\prod_{\Delta\subseteq [\ell]}}\pa{\frac{1}{x}-\pa{\sum_{i\in \Delta\setminus \ac{\ell}}a_i+k}}^{(-1)^{\ell-|\Delta|}}\ , 
\end{align*}
where we used that,  for $\ell\geq 2$, $\underset{\ell\in \Delta}{\prod_{\Delta\subseteq [\ell]}}x^{(-1)^{\ell-|\Delta|}}=1$.

For any $t\in \R$, define $F(t):=\prod_{\Delta\subseteq [\ell-1]}\pa{t-\sum_{i\in \Delta}a_i}^{(-1)^{|\Delta|+1}}$. Then, we have 
\[
Q(x)= \pa{F\pa{\frac{1}{x}-k}}
^{(-1)^\ell}\enspace . 
\]
From Lemma A.1 of \cite{feray2018}, we know that $F(t)-1$ is a rational function in $t$ of degree at most $-\ell+1$. Hence, $Q(x)-1$ is a multivariate series in $x$ of \textit{order} at least $\ell-1$. Combining this with the induction hypothesis, we deduce that $\mathrm{ord}\pa{P_{[\ell]}[v[k+1]]}\succcurlyeq (\ell-1, 0)$, which concludes the induction and the proof of Lemma~\ref{lem:factorielseries}. 

\end{proof}

\begin{proof}[Proof of Lemma \ref{lem:decomppowerseries}]
Let $\Delta\subseteq [\ell]$. We have 
$$P_{\Delta}[vw]-1=P_{\Delta}[v]-1+P_{\Delta}[w]-1+\pa{P_{\Delta}[v]-1}\pa{P_{\Delta}[w]-1}\enspace.$$
By hypothesis of the lemma, we both have $\mathrm{ord}\pa{P_{\Delta}[v]-1}\succcurlyeq (s,s')$ and $\mathrm{ord}\pa{P_\Delta[w]-1}\succcurlyeq (s,s')$. This also implies by Lemma \ref{lem:basicpropertiesorder} that $\mathrm{ord}\pa{\pa{P_{\Delta}[v]-1}\pa{P_\Delta[w]-1}}\succcurlyeq (s,s')$ and so, still using Lemma \ref{lem:basicpropertiesorder}, $\pa{P_{\Delta}[vw]-1}\succcurlyeq (s,s')$. This concludes the proof of Lemma~\ref{lem:decomppowerseries}.

\end{proof}

\subsection{Proof of the remaining  technical lemmas}

\begin{proof}[Proof of Lemma \ref{lem:independentquasifacto}]

For $\delta\subset \Delta$, we write $\delta_1=\delta\cap I_1$ and $ \delta_2= \delta\cap I_2$. By definition, one has
\begin{align*}
P_{\Delta}[v]=&\prod_{\delta\subseteq \Delta}v_\delta^{(-1)^{|\Delta|-|\delta|}}
=\prod_{\delta\subseteq \Delta}v_{\delta_1}^{(-1)^{|\Delta|-|\delta|}}v_{\delta_2}^{(-1)^{|\Delta|-|\delta|}}\\
=&\prod_{\delta_1\subseteq I_1}\prod_{\delta_2\subseteq I_2}v_{\delta_1}^{(-1)^{|I_1|-|\delta_1|}(-1)^{|I_2|-|\delta_2|}}v_{\delta_2}^{(-1)^{|I_2|-|\delta_2|}(-1)^{|I_1|-|\delta_1|}}\\
=&\prod_{\delta_1\subseteq I_1}\pa{v_{\delta_1}^{(-1)^{|I_1|-|\delta_1|}}}^{\sum_{\delta_2\subseteq I_2}(-1)^{|I_2|-|\delta_2|}}
\prod_{\delta_2\subseteq I_2}\pa{v_{\delta_2}^{(-1)^{|I_2|-|\delta_2|}}}^{\sum_{\delta_1\subseteq I_1}(-1)^{|I_1|-|\delta_1|}}\enspace.
\end{align*}
We have $\sum_{\delta_1\subseteq I_1}(-1)^{|I_1|-|\delta_1|}=\pa{-1}^{-|I_1|}\sum_{t=1}^{|I_1|}\binom{|I_1|}{t}(-1)^t=0^{|I_1|}=0$ since $I_1\neq \emptyset$. As a consequence, $$\prod_{\delta_2\subseteq I_2}\pa{v_{\delta_2}^{(-1)^{|I_2|-|\delta_2|}}}^{\sum_{\delta_1\subseteq I_1}(-1)^{|I_1|-|\delta_1|}}=1\enspace,$$
and, similarly,
$$\prod_{\delta_1\subseteq I_1}\pa{v_{\delta_1}^{(-1)^{|I_1|-|\delta_1|}}}^{\sum_{\delta_2\subseteq I_2}(-1)^{|I_2|-|\delta_2|}}=1\enspace.$$
We conclude that $P_{\Delta}[v]=1$.

\end{proof}

\begin{proof}[Proof of lemma \ref{lem:minimumexposant}]
Consider any  $S\subset \N^{2}$. Let
$$s_0=\min\ac{s\in\N, \enspace \exists s'\in \N\enspace s.t\enspace (s,s')\in S}\enspace,$$
and let $s_0'$ the maximal element in $\N$ such that for all $(s,s')\in S$, $s_0+s_0'\leq s+s'$, i.e. 
$$s_0+s'_0=s_{0}+\min_{(s,s')\in S} (s+s').$$ 
It is clear that  $(s,s')\succcurlyeq (s_0,s'_0)$ for all $(s,s')\in S$. It remains to prove that for all $(s_1,s'_1)$ satisfying $(s,s')\succcurlyeq (s_1,s'_1)$ for all $(s,s')\in S$, we have $(s_0,s'_0)\succcurlyeq (s_1,s'_1)$. First, for all $(s,s')\in S$, $s_1\leq s$, which implies that $s_0\geq s_1$. Then, for all $(s,s')\in S$, we have  $s+s'\geq s_1+s'_1$. Again, this readily implies that  $s_0+s'_0\geq s_1+s'_1$. We conclude that $(s_0,s'_0)\succcurlyeq (s_1,s'_1)$. Hence, $(s_0,s'_0)$ is the infimum of $S$, which is unique by definition
\end{proof}

\begin{proof}[Proof of Lemma \ref{lem:minimumexposantgraph}]
Let $\cL=(V,E)$ be a polynomial graph. We write $(V_1,\ldots, V_{|CC|})$ for the partition of the vertices corresponding to the connected components of a subgraph of $\cL$, where we only keep the edges with weight $1$. Let $\mathcal{T}$ denote a spanning tree of $\cL$ and let $(s,s')$ be its degree. By definition, $\mathcal{T}$ must span all the $V_{t}$'s.  Hence, it has at least $ |CC|-1$ edges of weight $x$ or $y$. We deduce that $s+s'\geq |CC|-1$. 

Let us construct a spanning tree $\mathcal{T}$, with degree  $(s,s')$, such that $s+s'=|CC|-1$. To do so, we consider $E_1$ the set of edges of $\mathcal{L}$ whose extremities are between two different connected components $V_{t}$ and $V_{t'}$. We denote $\cG_{CC}$ the induced polynomial multi-graph on $[|CC|]$. We extract from $\cG_{CC}$ a spanning tree $\mathcal{T}^{CC}$, and for all edges of $\mathcal{T}_{CC}$ between some $t$ and $t'$, we keep an edge of $E_1$ connecting $V_{t}$ and $V_{t'}$. We denote by $E_1'$, which is of size $|CC|-1$, the corresponding  subset of $E_1$. Let $\cG$ denote the polynomial graph obtained by the union of $E_1'$ with all the edges of weight $1$. It is clear that $\cG$ spans the vertex set $V$. We extract $\mathcal{T}$ a spanning tree of $\cG$. By definition, we have $s+s'= |CC|-1$. 

Let us now take $\mathcal{T}_0$ a spanning tree of $\cG_{CC}$ with $|CC|-1$ edges of weight $x$ or $y$ and that minimizes the number of edges of weight $x$. By constructing a spanning graph $\mathcal{T}^*$ of $\cG$ from $\mathcal{T}_0$ as above, we deduce that  its degree is the infimum of the degrees of all spanning trees. 
\end{proof}

\begin{proof}[Proof of Lemma \ref{lem:decompmimimumexposantgraph}]

We consider $\bold{\Delta}=\ac{\Delta_1,\ldots,\Delta_m}$ a set of subsets of $[\ell]$  such that $\bV\vee \Pi(\Delta_1)\vee\ldots\vee\Pi(\Delta_m)=\{V\}$. If there exists $i\in [m]$ such that $\cL[\Delta_i]$ is not connected, then $\mathrm{ord}\pa{\cL[\Delta_i]}=\pa{+\infty,+\infty}$ and the result is trivial. Let us now suppose that, for all $i\in [m]$,  $\cL[\Delta_i]$ is connected. For $i\in [m]$, consider $\cT_i$ a spanning tree of $\cL[\Delta_i]$ with $deg(\cT_i)=\mathrm{ord}\pa{\cL[\Delta_i]}$ --such a tree exists by Lemma \ref{lem:minimumexposantgraph}. Each $\cT_i$ can then be viewed as a subset of edges of the original graph $\cL$. Let $E_0$ be the union of all the edges  of $\cT_i$'s and of all edges of weight $1$ of $\cL$. Since $\bV\vee \Pi(\Delta_1)\vee\ldots\vee\Pi(\Delta_m)=\{V\}$, we know that $E_0$ spans $V$. We extract a spanning tree $\cT$ from $E_0$. Considering $E_0$ as a monomial in $x$ and $y$, we deduce that  
$$deg(\cT)\preccurlyeq deg(E_0)=\sum_{i=1}^m deg(\cT_i)=\sum_{i=1}^m \mathrm{ord}\pa{\cL[\Delta_i]} \enspace .$$
Since $\mathrm{ord}\pa{\cL}\preccurlyeq deg(\cT)$, we conclude that 
$$\mathrm{ord}\pa{\cL} \preccurlyeq\sum_{i=1}^m \mathrm{ord}\pa{\cL[\Delta_i]} \enspace. $$
\end{proof}

\begin{proof}[Proof of Lemma \ref{lem:basicpropertiesorder}]
Let us take $f(x,y)=\sum_{s,s'\geq 0}f_{s,s'}x^s y^{s'}$ and $g=\sum_{s,s'\geq 0}g_{s,s'}x^sy^{s'}$.
\begin{enumerate}
\item We have $(fg)(x,y)=\sum_{s_0,s_0'\geq 0}\pa{\sum_{s\leq s_0}\sum_{s'\leq s'_0}f_{s,s'}g_{s_0-s, s'_0-s'}}x^{s_0}y^{s'_0}$. We write $(fg)_{s_0,s'_0}=\sum_{s\leq s_0}\sum_{s'\leq s'_0}f_{s,s'}g_{s_0-s, s'_0-s'}$. For having $(fg)_{s_0,s'_0}\neq 0$, there must exists $s\leq s_0$ and $s'\leq s'_0$ such that $f_{s,s'}\neq 0$ and $g_{s_0-s, s'_0-s'}\neq 0$. In that case, by definition, we have $(s,s')\succcurlyeq \mathrm{ord}(f)$ and $(s_0-s, s'_0-s')\succcurlyeq \mathrm{ord}(g)$ so that $(s_0,s'_0)\succcurlyeq  \mathrm{ord}(f)+\mathrm{ord}(g)$. This being true for all $(s_0,s'_0)$ with $(fg)_{s_0,s'_0}\neq 0$, we deduce taking the infimum that $\mathrm{ord}(fg)\succcurlyeq \mathrm{ord}(f)+\mathrm{ord}(g)$.

\item We have $(f+g)(x,y)=\sum_{s,s'\geq 0}\pa{f_{s,s'}+g_{s,s'}}x^sy^{s'}$. Let us write $(f+g)_{s,s'}=f_{s,s'}+g_{s,s'}$. For having $(f+g)_{s,s'}\neq 0$, one needs that either $f_{s,s'}\neq 0$ or $g_{s,s'}\neq 0$ and so $(s,s')\succcurlyeq \inf \pa{\mathrm{ord}(f),\mathrm{ord}(g)}$. We conclude $\mathrm{ord}(f+g)\succcurlyeq \inf \pa{\mathrm{ord}(f),\mathrm{ord}(g)}$:

\item Let us suppose that $f(0,0)=0$. Then, on an open ball around $(0,0)$, $|f(x,y)|<1$, and so 
$$\frac{1}{1+f(x,y)}-1=\sum_{d\geq 1}(-f(x,y))^d:=\sum_{s,s'\geq 0}\pa{\frac{1}{1+f}-1}_{s,s'}x^sy^{s'}\enspace.$$
Using point 1., we know that for all $d\geq 1$, $\mathrm{ord}((-f)^d)\succcurlyeq d\cdot\mathrm{ord}(f)\succcurlyeq \mathrm{ord}(f)$. For $s,s'$ such that $\pa{\frac{1}{1+f}-1}_{s,s'}\neq 0$, there exists $d$ such that, wherenever $d'\geq d+1$, $d'\cdot\mathrm{ord}(f)\succcurlyeq (s,s')$. And so, using the second point of this proof generalized to finite sums of multivariate power series, we get that $(s,s')\succcurlyeq \mathrm{ord}(\sum_{d'\leq d}f^d)\succcurlyeq \mathrm{ord}(f)$. We conclude $\mathrm{ord}(\frac{1}{1+f}-1)\succcurlyeq \mathrm{ord}(f)$.
\end{enumerate}
\end{proof}

\begin{proof}[Proof of Lemma \ref{lem:quasifactographgeneral}]
Let us write $\mathrm{ord}\pa{\cL^*}=\pa{d_{[\ell]},d'_{[\ell]}}$. Suppose $\pa{d_{[\ell]},d'_{[\ell]}}$ is finite. From Lemma \ref{lem:minimumexposantgraph}, we know that there exists a spanning tree $\mathcal{T}$ of $\cL^*$ of degree $\pa{d_{[\ell]},d'_{[\ell]}}$.

A spanning tree, if there exists one, has necessarily at least $|\ell|-1$ edges. All the edges in $\mathcal{L}^*$ are of weight $x$, $y$. Hence, $d_{[\ell]}+d'_{[\ell]}\geq \ell-1$.

We recall that $\mathcal{B}$ is the graph on $[\ell]$ with an edge between $t$ ,$t'$ if and only if there exists $i\in [r]$ such that both $I_{t}$ and $I_{t'}$ intersects $B_{i}$, and that $cc(\cB)$ stands for the number of connected components of this graph. We remark that $(t,t')$ is an edge of $\cB$ if and only if it is an edge of weight $y$ of $\cL^*$. Thus, $\mathcal{T}$ has at least $cc(\cB)-1$ edges of weight $x$ and so $d_{[\ell]}\geq cc(\cB)-1$. 

\end{proof}

\subsection{Case $r=0$.}\label{ref:caser0}

This section is devoted to proving the second point of Lemma \ref{lem:upperboundgeneralcumulant}. Assume that $Z_1, \ldots, Z_\ell$ have mixed moments given by 
$$\E\cro{\prod_{t\in \Delta}Z_{t}}=\eta^{|I_\Delta|}\prod_{j=1}^q \frac{1}{1-x_0}\times\ldots\times \frac{1}{1-(|I_\Delta\cap A_{j}|-1)x_0}\ ,\quad \text{for $\Delta\subseteq [R]$.}$$

This is a particular case of the cumulant dealt in the first point. We could use the result of the first point, but loose a factor $2$ compared to the result that we want to prove. However, let us explain in this section what simplifications we can make in the specific case $r=0$. We Define, for $\Delta\subseteq [\ell]$, $u_\Delta(x)=\eta^{|I_\Delta|}\prod_{j=1}^q \frac{1}{1-x}\times\ldots\times \frac{1}{1-(|I_\Delta\cap A_{j}|-1)x}$ which is a power series in $x$. We define the associated power series, for $\Delta\subseteq [\ell]$, $\kappa_{\Delta}(x)=\sum_{G\in \cP\pa{\Delta}}m(G)\prod_{\delta\in G}u_\delta(x)$. We remark that $\cumul\pa{Z_1,\ldots,Z_\ell}=\kappa_{[\ell]}(x_{0})$. The \textit{order} of this power series is of the form $(d_\Delta,0)$ and, using the results of Section \ref{sec:orderkappa}, we get that $d_{[\ell]}\geq \ell-1$. 
When $x$ is small enough, we can expand the power series $\kappa_{[\ell]}(x)=\sum_{d\geq  \ell-1}\kappa_{d}x^d$. We plug the results of Section \ref{sec:controlcoeff} and we get that, for all $d\geq  \ell-1$, $|\kappa_{d}|\leq \eta^L L^{2d}\ell^{2\ell}$, where we recall $L=|I_{[\ell]}|$. When $2Lx^2\leq 1$, we end up with, 
\begin{align*}
|\kappa_{[\ell]}(x)|&\leq \sum_{d\geq \ell-1}\ell^{2\ell} \eta^L L^{2d}x^d\\
&\leq 2\ell^{2\ell}\eta^L \pa{L^2 x}^{\ell-1}\enspace,
\end{align*}
and we conclude with $|\kappa_{[\ell]}(x_0)|=|\cumul\pa{Z_1,\ldots, Z_\ell}|\leq  2\ell^{2\ell}\eta^L \pa{L^2 x_{0}}^{\ell-1}$ when $2Lx_{0}^2\leq 1$.

\section{Other losses for seriation}\label{sec:seriation:complementary}

In this section, we introduce the general T\oe plitz-Robinson seriation model. Given a vector $\theta = (\theta_0,\dots, \theta_{n-1})\in \bbR^n$, we write $T(\theta)$ for the T\oe plitz matrix in $\mathbb{R}^{n\times n}$ with entries $T(\theta)_{i,j} = \theta_{i-j}$, where we use the convention $\theta_k := \theta_{-k}$ for $k \leq 0$. 
The matrix $T(\theta)$ is  said to be T\oe plitz Robinson if the entries of $\theta$ are non-increasing. 

Given a permutation $\pi^*$ in $\cS_n$, we write $P(\pi^*)$ for the corresponding permutation matrix. In this seriation model, we observe a $Y\in \mathbb{R}^{n\times n}$ with $Y=X + E$ where  the matrix $E$ is made of independent standard normal entries and 
\begin{equation}\label{eq:model:seriation:Toeplitz}
X=P(\pi^*)T(\theta)P(\pi^*)
\end{equation}
with $T(\theta)$ is a possibly unknown T\oe plitz-Robinson matrix, and $P(\pi^*)$ is an unknown permutation matrix to be estimated. The permutation $P(\pi^*)$ is not identifiable: $P(\pi^*)$ and the corresponding reverse permutation $P(\pi^*)_{-}$ are such that $P(\pi^*)T(\theta)P(\pi^*)^T= P(\pi)_-^*T(\theta)P(\pi^*)_-^{T}$. Hence, the main goal  in seriation is to recover a permutation $\pi$ such that $X= P(\pi)T(\theta)P(\pi)^T$.

Given a positive integer $\rho$ and $\lambda>0$, we defined the vector $\theta_{\lambda,\rho}\in \mathbb{R}^n$ by $(\theta_{\lambda,\rho})_{i}= \lambda \1_{i\leq \rho}$. Note that the model~\eqref{eq:definition:model:seriation} introduced in Section~\ref{sec:seriation} corresponds to $X= P(\pi^*)T(\theta_{\lambda,k})P(\pi^*)$

One way of quantifying the error of a given permutation matrix $\Pi$ is the $\ell_2$ loss which quantifies to what extent 
$X$ is close in Frobenius distance to a Robinson matrix ordered by $\pi$.
$$
\ell_2(\pi) := \inf_{R \in \cR_n} \|  X  - P(\pi) RP(\pi)^T\|_F\enspace ,
$$
where $\cR_n$ is the collection of Robinson matrices, that is the collection of symmetric matrices $R$ such, for any $i< j$, we have  $R_{i,j} \leq R_{i+1,j}$ and $R_{i,j}\geq R_{i,j+1}$. Equivalently, $R$ is a Robinson matrix if its entries are non-increasing when one moves away from the diagonal. 

It is established in \cite{berenfeld2024seriation} that the minimax rate in $\ell^2_2$ for the seriation problem is of the order of $n$. In contrast, the best known polynomial-time procedures~\cite{berenfeld2024seriation} only achieves a rate of the order of $n^{3/2}$. To provide evidence of the optimality of the $n^{3/2}$ for polynomial-time procedures, \cite{berenfeld2024seriation} also provide  a nearly matching LD polynomial lower bound but their model was quite different as $\pi^*$ was not necessarily a permutation. Here, we strengthen this evidence by deriving a computational $n^{3/2}$ lower bound in the model~\eqref{eq:model:seriation:Toeplitz}.

As for clustering, the problem of estimating a permutation $\hat{\pi}$ is of combinatorial nature and in order to state a LD lower bounds, we reduce it to the problem of estimating the matrix $X=P(\pi^*) T(\lambda,\rho)P(\pi^*)^T$. 

The following lemma is a slight adaptation of Lemma 8 in~\cite{berenfeld2024seriation} --see Section~\ref{sec:proof:reminaining:seriation} for a proof. 
\begin{lem}[Reduction to matrix estimation]\label{lem:reduction}
Consider any $\lambda>0$, positive integer $\rho$ and any estimator $\widehat{\pi}$. 
Define the matrix $\hat{X}$ by $\hat{X}_{i,j}= \lambda/2$ if $|\hat{\pi}^{-1}(i)-\hat{\pi}^{-1}(j)|\leq \rho$ and $\hat{X}_{i,j}= 0$ otherwise. Then, we have 
\begin{equation}\label{eq:lower_bound_risk}
\ell^2_2(\hat{\pi})\geq \frac{1}{4}\left[\left[\|\hat{X}-X\|_F^2\right]- \lambda^2\rho n \right]  
\end{equation}
\end{lem}
As a consequence of this lemma, if all polynomial-time estimators $\hat{X}$ of $X$ satisfy $\mathbb{E}[\|\hat{X}-X\|_F^2]\geq \frac{3}{2} \lambda^2 \rho n$, then no polynomial-time estimators of $\hat{\pi}$ satisfy 
$\mathbb{E}[\ell^2_2(\hat{\pi})]\geq c  \lambda^2 \rho n$ for some $c>0$. Hence, to establish the hardness of achieving $\mathbb{E}[\ell^2_2(\hat{\pi})] \preccurlyeq_{\log(n)} n^{3/2}$, it suffices to prove that, for $\rho=\lceil \sqrt{n} \rceil$ and $\lambda$ such that $1/\lambda$ is a polylogarithmic  term, one has $\mathbb{E}[\|\hat{X}-X\|_F^2]\geq \frac{3}{2} \lambda^2 \rho n$. The latter precisely corresponding to LD lower bound in Theorem~\ref{thm:lowdegreeseriation}.

In particular, taking $D\asymp \log(n)$, $\rho\asymp \sqrt{n}$, $\lambda$ as the inverse of a polylogarithmic term, we deduce from  that theorem that 
\[
MMSE_{\leq \log(n)}\geq \lambda^2 \frac{3\rho}{2}n  \succcurlyeq_{\log(n)} n^{3/2}\ . 
\] 
Together with Lemma~\ref{lem:reduction}, this suggests the hardness of achieving seriation with a $\ell_2^2$ rate smaller than $n^{3/2}$.

\section{Remaining proofs}

\subsection{Seriation}\label{sec:proof:reminaining:seriation}

\begin{proof}[Proof of Proposition~\ref{prp:upper:polynomial:seriation}]
    
The first result straightforwardly derives from the proof of Theorem 4 in~\cite{berenfeld2024seriation}. Translating the third line of that proof with our notation, we deduce that there exists an event $\cA$ of probability higher than $1-4/n$ such that 
\[
\|X_{\hat{\pi}} - X\|_F^2 \leq \lambda^2 n + \lambda n^{3/2} \sqrt{\log(n)}
\]
Since $\|X_{\hat{\pi}} - X\|_F^2\leq \lambda^2 n^2$, we deduce that 
\[
\mathbb{E}\left[\|X_{\hat{\pi}} - X\|_F^2\right]\leq 5\lambda^2 n + \lambda n^{3/2} \sqrt{\log(n)}\ . 
\]
Let us turn to the second result. For a fixed $(i,j)$, it follows from a standard control of the Gaussian tail distribution that  $(\hat{X}_{th})_{ij}= X_{ij}$ with probability at least $1-1/n^2$. Hence, $\mathbb{E}[((\hat{X}_{th})_{ij}- X_{ij})^2]=\lambda^2/n^2$. The result follows by linearity. 
\end{proof}

\begin{proof}[Proof of Proposition~\ref{prp:upper:IT:seriation}]
$\|Y - \hat{X}_{\hat{\pi}}\|_F^2 \leq \|Y- X\|^2_F$ implies that $\|X - \hat{X}_{\hat{\pi}}\|_F^2+ 2 \langle E , X - \hat{X}_{\hat{\pi}}\rangle\leq 0$. Let us apply an union bound of the random variable $\langle E ,\hat{X}_{\pi}-X\rangle$  over all $n!$ matrices of the form $X_{\pi}$. 
Hence, with probability higher than $1-e^{-x}$, we have $-\langle E , X - \hat{X}_{\hat{\pi}}\rangle\leq \sqrt{2[n\log(n)+x]}\|X -\hat{X}_{\hat{pi}}\|_F$. On this event, we have $\|X - \hat{X}_{\hat{\pi}}\|^2_F\leq 8[n\log(n) + x]$. The result follows by integrating with respect to $x$. 
\end{proof}

\begin{proof}[Proof of Lemma~\ref{lem:reduction}]
The proof closely follows that of Lemma 8 in~\cite{berenfeld2024seriation}. Still, we provide here a complete proof as 
the notation and the conventions slightly differ from~\cite{berenfeld2024seriation}. 
Without loss of generality, we assume that $\pi^*$ is the identity permutation. 
If $|\hat{\pi}(i)- \hat{\pi}(j)|\leq 2\rho$, we have $|\hat{X}_{ij}-X_{ij}|= \lambda/2$. Hence, it follows that 
\begin{equation}\label{eq:upper_bound_l2_pi3}
\|\hat{X}-X\|_F^2\leq \lambda^2 \rho n + \lambda^2 \left|\{(i,j): X_{ij}= \lambda \text{ and } |\hat{\pi}(i)- \hat{\pi}(j)|>2\rho\}\right| 
\end{equation}
We bound the RHS of \eqref{eq:lower_bound_risk} using the loss $\ell^2_2(\hat{\pi})$. For that purpose, let us introduce $\mathcal{R}'_n$ the collection of $n\times n$ matrices whose rows are unimodal and achieve their maximum on the diagonal. Obviously $\mathcal{R}_n\subset \mathcal{R}'_n$ and so $\ell^2_2(\hat{\pi})\geq \inf_{R\in \mathcal{R}'_n}\| X - P(\hat{\pi}) R P(\hat{\pi})^{T} \|_F^2$. Besides, we introduce $\mathcal{R}'_n(\lambda)$ as the subset of $\mathcal{R}'_n$ that only take its values in $\{0,\lambda\}$ outside the diagonal. For $R\in \cR_n'$, let us introduce $R_0\in \cR_n(\lambda)$ defined by $\cro{R_0}_{ij}=\lambda\1\ac{R_{ij}\leq \lambda/2}$. Since $X$ takes its values on $\ac{0,\lambda}$, we have that $\| X - P(\hat{\pi}) R P(\hat{\pi})^{T} \|_F^2\geq \frac{1}{4}\| X - P(\hat{\pi}) R_0 P(\hat{\pi})^{T} \|_F^2$. We deduce
\begin{equation}\label{eq:lower_bound_l2_pi}
\ell^2_2(\hat{\pi})\geq  \inf_{R\in \mathcal{R}'_n}\| X - P(\hat{\pi}) R P(\hat{\pi})^{T} \|_F^2 \geq  \frac{1}{4}\inf_{R\in \mathcal{R}'_n(\lambda)}\| X - P(\hat{\pi})R P(\hat{\pi})^{T}\|_F^2 .
\end{equation}
Let  $R$ denote any matrix in $\mathcal{R}'_n(\lambda)$ that achieves the above infimum.  Fix any $i\in [n]$. We claim that  
\begin{align}
\frac{\|[ X - P(\hat{\pi}) RP(\hat{\pi})^T ]_{i}\|^2}{\lambda^2}&  \geq  \left|\{j: |i-j|\leq \rho \text{ and } |\hat{\pi}(i)-\hat{\pi}(j)|> 2\rho \}\right|   , \label{eq:lower_row}  
\end{align}
where $[ X - P(\hat{\pi}) RP(\hat{\pi})^T ]_{i}$ stands for the $i$-th row ot the matrix. Before showing this claim, let us conclude the proof. Since $$\left|\{j: |i-j|\leq \rho \text{ and } |\hat{\pi}(i)-\hat{\pi}(j)|> 2\rho \}\right|= \left|\{j: X_{i,j}= \lambda \text{ and } |\hat{\pi}(i)- \hat{\pi}(j)|> 2\rho\}\right|\enspace,$$ it follows from 
\eqref{eq:upper_bound_l2_pi3} and~\eqref{eq:lower_row} that 
\[
\|\hat{X}-X\|_F^2\leq \lambda^2 \rho n + 4\ell^2_2(P(\hat{\pi}))\ , 
\]
which is precisely the desired bound.

Let us now show~\eqref{eq:lower_row}. Without loss of generality, we can assume that  the $j$'s such that $|i-j|\leq  \rho$ and $ |\hat{\pi}(i)-\hat{\pi}(j)|> 2\rho $ are such that $\hat{\pi}(j)> \hat{\pi}(i)+ 2\rho$. We consider two cases: (i)
$R_{\hat{\pi}(i), \hat{\pi}(i) + 2\rho + 1 }=0$.  For all $j$ such that $|i-j|\leq  \rho$ and $ |\hat{\pi}(i)-\hat{\pi}(j)|> 2\rho$, we have $[ X - P(\hat{\pi}) RP(\hat{\pi})^T ]_{ij}=\lambda$. Then, this straightforwardly implies~\eqref{eq:lower_bound_l2_pi}. Now, suppose (ii) that  $R_{\hat{\pi}(i), \hat{\pi}(i) + 2\rho + 1 }=\lambda$. There exists at least $\left|\{j: |i-j|\leq \rho \text{ and } |\hat{\pi}(i)-\hat{\pi}(j)|\geq 2\rho \}\right|$ indices $j$ such that $\hat{\pi}(j)\in (\hat{\pi}(i), \hat{\pi}(i)+2\rho]$ and $|i-j|>\rho$. Since, for a such a $j$, we have  $X_{i,j}=0$ and $R_{\hat{\pi}(i),\hat{\pi}(j)}= \lambda$ , we again derive~\eqref{eq:lower_row}.
\end{proof}

\subsection{Proof of Proposition \ref{prop:hardnessclustering}}\label{prf:hardnessclustering}

This proof is adapted from the proof of Proposition A.3 of \cite{Even25a}. Let us consider a regime where $n,K$ go to infinity (and possibly $p$) and where $K=o(n)$. Given a partition $G=G_1,\ldots, G_K$ such that $|G_k|=n/K$ for all $k\in [K]$, we define the partnership matrix $\Gamma^G\in \R^{n\times n}$ by $\Gamma^G_{i,j}=\sum_{k\in [K]}\1\ac{i,j\in G_k}$. We write $\Gamma^*$ for $G^*$. Then, 
\begin{equation*}
n(n-1) MMSE_{poly}=\E\cro{\|\Gamma^*\|^2_F}-corr_{poly}^2=n^2/K-corr_{poly}^2,
\end{equation*}
where we define $$corr_{poly}^2:= \sup_{\hat \Gamma\ poly-time,\ \E\cro{\|\hat \Gamma\|^2_{F}}=1} \E\cro{\langle \Gamma^*,\hat \Gamma\rangle_{F}}^2\enspace.$$ 
Since $MMSE_{poly}= 1/K(1+o(1))$, it follows that   $corr_{poly}^2=o(n^2/K)$. For $\hat G$ a polynomial time estimator of $G^*$ with prescribed group size $n/K$, we have 
$$\E\left[\<\Gamma^{\hat G}, \Gamma^*\>\right]\leq \sqrt{\E\cro{\|\Gamma^{\hat{G}}\|^2_F}}\ corr_{poly}=\sqrt{n^2\over K}\ corr_{poly}$$ 
since $\|\Gamma^{\hat{G}}\|^2_F=n^2/K$ almost surely. We deduce 
\begin{align*}
\E\cro{\|\Gamma^{\hat{G}}-\Gamma^*\|^2_F}=&2n^2/K-2\mathbb{E}\left[\<\Gamma^{\hat{G}}, \Gamma^*\>\right]\\
\geq& 2n^2/K-\sqrt{n^2/K}\ corr_{poly}\\
\geq& 2n^2/K(1+o(1))\enspace.
\end{align*}
We conclude the proof of the proposition with the following lemma.
\begin{lem}\label{lem:clusteringhardness} For any $G$ with prescribed group size $n/K$, we have 
$\cro{1-err(G,G^*)}^2\leq 1-\frac{K\|\Gamma^G-\Gamma^*\|_F^2}{2n^2}$.
\end{lem}

\begin{proof}[Proof of Lemma \ref{lem:clusteringhardness}]
With no loss of generality, we assume that the permutation $\psi$ in the definition of $err(G,G^*)$ is the identity. Thus,
\[
err(G,G^*)= \frac{1}{2n}\sum_{k=1}^K |G_k \Delta G^*_k|= 1 - \frac{1}{n}\sum_{k=1}^K
|G_k\cap G^*_k|\ , 
\]
which implies that $\sum_{k=1}^K|G_k\cap G^*_k| = n[1- err(G,G^*)]$. We have, for $G,G^*$ with group size $n/K$,  that 
\begin{equation*}
\|\Gamma^G-\Gamma^*\|_F^2=  \|\Gamma^G\|^2_F+\|\Gamma^*\|_F^2-2\<\Gamma^G, \Gamma^*\>_F = 2n^2/K-2\<\Gamma^G, \Gamma^*\>_F \ .
\end{equation*}
Furthermore,
$$\<\Gamma^G, \Gamma^*\>_F=\sum_{ij\in [n]}\sum_{k,k'\in [K]}\1\ac{i,j\in G_k}\1\ac{i,j\in G_k^*} \enspace.$$
For $k\in [K]$, let us denote $N_k=\left|G_k\cap G^*_k\right|$. Then,
\begin{equation*}
\<\Gamma^G, \Gamma^*\>_F\geq\sum_{k\in [K]}N_k^2\geq \frac{1}{K}\cro{\sum_{k\in [K]}N_k}^2\geq n^2\cro{1-err(G,G^*)}^2/K\enspace,
\end{equation*}
where we used the Cauchy-Schwarz inequality for the second inequality. We conclude that $$\|\Gamma^G-\Gamma^*\|_F^2\leq 2{n^2\over K}\pa{1-\cro{1-err(G,G^*)}^2}\enspace.$$
\end{proof}

\subsection{Proof of Proposition \ref{prop:featurematchinghardness}}\label{prf:featurematchinghardness}

This proof is also adapted from the proof of Proposition A.3 of \cite{Even25a} and is similar to the proof of Proposition \ref{prop:hardnessclustering}. Let us consider a regime where $M,K$ go to infinity (and possibly $p$). Given a $M$-tuple of permutations $\pi=\pi_1,\ldots, \pi_M$, we define the partnership matrix $\Gamma^\pi\in \R^{(K\times M)^2}$ by $\Gamma^\pi_{(k,m), (k',m')}=\1\ac{\pi_m(k)=\pi_{m'}(k')}$. We write $\Gamma^*$ for $\pi^*$. Then, 
\begin{equation*}
(M(M-1))K^2 MMSE_{poly}=\E\cro{\|\Gamma^*\|^2_F}-corr_{poly}^2=KM^2-corr_{poly}^2,
\end{equation*}
where we define $$corr_{poly}^2:= \sup_{\hat \Gamma\ poly-time,\ \E\cro{\|\hat \Gamma\|^2_{F}}=1} \E\cro{\langle \Gamma^*,\hat \Gamma\rangle_{F}}^2\enspace.$$ 
Since $MMSE_{poly}= 1/K(1+o(1))$, it follows that   $corr_{poly}^2=o(M^2K)$. For $\hat{\pi}\in \pa{\mathcal{S}_K}^M$ a polynomial time estimator of $\pi^*$, we have 
$$\E\left[\<\Gamma^{\hat{\pi}}, \Gamma^*\>\right]\leq \sqrt{\E\cro{\|\Gamma^{\hat{\pi}}\|^2_F}}\ corr_{poly}=\sqrt{KM^2}\ corr_{poly}$$ 
since $\|\Gamma^{\hat{\pi}}\|^2_F=KM^2$ almost surely. We deduce \begin{align*}
\E\cro{\|\Gamma^{\hat{\pi}}-\Gamma^*\|^2_F}=&2KM^2-2\mathbb{E}\left[\<\Gamma^{\hat{\pi}}, \Gamma^*\>\right]\\
\geq& 2KM^2-\sqrt{KM^2}\ corr_{poly}\\
\geq& 2KM^2(1+o(1))\enspace.
\end{align*}
We conclude the proof of the proposition with the following lemma.
\begin{lem}\label{lem:featurematchinghardness} For any $\pi\in\pa{\mathcal{S}_{K}}^M$, we have 
$\cro{1-err(\pi,\pi^*)}^2\leq 1-\frac{\|\Gamma^\pi-\Gamma^*\|_F^2}{2KM^2}$.
\end{lem}

\begin{proof}[Proof of Lemma \ref{lem:featurematchinghardness}]
With no loss of generality, we can suppose that the permutation $\psi$ minimizing the definition of the error $err(\pi,\pi^*)$ is the identity and thus $$err(\pi,\pi^*)=\frac{1}{KM}\sum_{k,m}\1\ac{\pi_m(k)\neq \pi^*_m(k)}=1-\frac{1}{KM}\sum_{k,m}\1\ac{\pi_m(k)= \pi^*_m(k)}\enspace.$$
And, we also have $$\|\Gamma^\pi-\Gamma^*\|_F^2=\|\Gamma^\pi\|_F^2+\|\Gamma^*\|_F^2-2\<\Gamma^\pi, \Gamma^*\>\enspace.$$ We can directly compute $\|\Gamma^\pi\|_F^2=\|\Gamma^*\|_F^2=KM^2$. For the crossed-product term, we expand $$\<\Gamma^\pi, \Gamma^*\>=\sum_{(k,m), (k',m')}\1\ac{\pi_m(k)=\pi_{m'}(k')}\1\ac{\pi^*_m(k)=\pi^*_{m'}(k')}\enspace.$$
For $k\in [K]$, let us denote $N_k=\left|\ac{(k',m)\in [K]\times [M], \pi_m(k')=k\\ \quad\text{and}\quad \pi^*_m(k')=k}\right|$. We get \begin{align*}
\<\Gamma^\pi, \Gamma^*\>\geq&\sum_{k\in [K]}N_k^2
\geq \frac{1}{K}\cro{\sum_{k\in [K]}N_k}^2\geq KM^2\cro{1-err(\pi,\pi^*)}^2\enspace,
\end{align*}
where we used Cauchy-Schwarz inequality for the second inequality. We conclude that $$\|\Gamma^\pi-\Gamma^*\|^2_F\leq 2KM^2-2KM^2\cro{1-err(\pi,\pi^*)}^2\enspace.$$
\end{proof}

\subsection{Proof of Corollary \ref{cor:featurematchingupperbound}}\label{prf:featurematchingupperbound}

We write $G^*$ the partition of $[K]\times [M]$ defined by $G_{k}^*=\ac{(k',m)\in [K]\times [M],\enspace \pi^*_m(k')=k}$. We remark that $G^*$ is a balanced partition of $[K]\times [M]$ into $K$ groups and that $Y^{(m)}_{k'}\sim \cN\pa{\mu_k, \sigma^2I_p}$. 

\begin{enumerate}
\item Using Proposition 4 of \cite{Even24}, we know that a single linkage hierarchical procedure is able to recover exactly $G^*$ with high probability as soon as $\Delta^2\gtrsim \log(KM)+\sqrt{p\log(KM)}$;
\item Using Theorem 1 of \cite{giraud2019partial}, we know that a SDP relaxation procedure is able to recover exactly $G^*$ in high dimension $p\geq MK$ as soon as $\Delta^2\geq_{\log}\sqrt{\frac{pK}{M}}$;
\item Using Proposition 3.2 of \cite{Even25a}, we deduce that, with a method based on spectral projections together with the low-dimensional clustering procedure from \cite{LiuLi2022}, we are able to recover exactly $G^*$ with high probability when $M\geq K^c$, for some numerical constant $c$, and $\Delta^2\geq_{\log}1+\sqrt{\frac{pK}{M}}$;
\item Still using Proposition 3.2 of \cite{Even25a}, we know that, with a method based on spectral projections together with a single linkage hierarchical procedure, we are able to recover exactly $G^*$ with high probability when $M\gtrsim K$, $MK\geq p\geq M$ and $\Delta^2\geq_{\log}1+\sqrt{\frac{pK}{M}}$.
\end{enumerate}

All the methods that are used for these four points are computable in polynomial time. Hence, combining those results, we deduce that, except when both $M\in [K,K^c]$ and $p\in [\log(KM), \frac{p}{KM}]$, there exists an efficient algorithm $\hat{G}$ which recovers exactly with high probability $G^*$ as soon as $\Delta^2\geq_{\log} 1 +\min\pa{\sqrt{p}, \sqrt{\frac{pK}{M}}}$. We define the estimator $\hat{\pi}$ by $\hat{\pi}_m(k')=k$ if $(k',m)\in \hat{G}_k$. On the event where $\hat{G}=G^*$, $\hat{\pi}$ is a $M$-tuple of permutations on $[K]$ which satisfies $err(\hat{\pi}, \pi^*)=err(\hat{G},G^*)=0$. This concludes the proof of the corollary.

 \subsection{Sketch of proof of Proposition \ref{prop:upperboundinformationalfeature}}\label{prf:upperboundinformationalfeature}

This proof is directly adapted from the proof of Theorem 2 of \cite{Even24} for exact $K$-means and we mainly explain the key steps.  We view $Y$ as a matrix in $\R^{KM\times p}$. Given $\pi\in \pa{\cS_K}^M$, let $B(\pi)$ the normalized partnership matrix $B(\pi)\in \R^{KM\times KM}$ defined $\pa{B(\pi)}_{(k,m),(k',m')}=\frac{1}{M}\1\ac{\pi_m(k)=\pi_{m'}(k')}$. As noticed by \cite{PengWei07}, minimizing Criterion \eqref{eq:critpi} is equivalent to maximizing the quantity $\<YY^T, B(\pi)\>$. Let $\cB:=\{B(\pi), \pi\in \pa{\cS_K}^M\}$ and let $$\hat{B}:=\argmax_{B\in \cB}\<YY^T, B\>\enspace.$$
$\cB$ is a subset of the set $\cC=\{ B\in S_{n}(\mathbb{R})^{+}:\ B_{ij}\geq 0, \tr(B)=K, B1=1, B^{2}=B \}$ which is the set of normalized partnership matrices associated to partitions $G$ of $[KM]$ into $K$ groups. In particular, all the inequalities from \cite{Even24} which are valid on $\cC$ are in fact valid on $\cB$. In the following, we write $B^*$ for $B(\pi^*)$. Let us consider;\begin{itemize}
\item $A\in\R^{KM\times K}$ the membership matrix defined by $A_{(k',m),k}=\1_{\pi^*_m(k')=k}$, 
\item $E\in\R^{KM\times p}=Y-\E[Y]$ the noise matrix,
\item $\mu \in\R^{K\times p}$ whose $k$-th row is $\mu_{k}$.
\end{itemize}
By definition of $\hat{B}$, we have $\<YY^T, \hat{B}-B^*\>\geq 0$. Decomposing $Y$ with the relation $Y=A\mu +E$, we get
$$\<A\mu(A\mu)^{T},B^{*}-\hat{B}\>\leq \<A\mu E^{T}+E(A\mu)^{T},\hat{B}-B^{*}\>+\<EE^{T}-pI_{MK},\hat{B}-B^{*}\>\enspace.$$
Given a matrix $A$, we write $\|A\|_1$ for the sum of the absolute values of its entries. 
For any $M$-tuple of permutation, the error $err(\pi,\pi^*)$ is linked to the norm $\|B^*-B^*B(\pi)\|_1$ 
through the inequality $err(\pi,\pi^*)\lesssim \frac{\|B^*-B^*B(\pi)\|_1}{KM}$ (see Lemma 9 of \cite{Even24}). Building on this bound;\begin{enumerate}
\item One can lower-bound the signal term $\<A\mu(A\mu)^{T},B^{*}-B\>$ with respect to $\frac{\|B^*-B^*B\|_1}{KM}$. Using Lemma 4 from \cite{giraud2019partial} which is valid on a set containing $\cC$ and thus $\cB$, we get that, for all $\pi\in \pa{\cS_K}^M$, $\<A\mu(A\mu)^{T},B^{*}-B(\pi)\>\gtrsim \Delta^2\|B^*-B^*B(\pi)\|_1\gtrsim KM\Delta^2err(\pi,\pi^*)$.
\item Directly plugging Lemma 11 of \cite{Even24}, we deduce that with high probability, simultaneously on all $B\in \cB$, \begin{align*}
    |\<EE^T-pI_{KM}, B^*-B\>|\lesssim& \|B^*-B^*B\|_{1}\pa{\log(K)+\sqrt{\frac{p}{M}\log(K)}}\\
    &+\|B^*-B^*B\|_{1}\pa{\sqrt{\frac{p}{M}\log(\frac{MK^4}{ \|B^*-B^*B\|_{1}})}+\log(\frac{MK^4}{ \|B^*-B^*B\|_{1}})}\enspace.
\end{align*}
In particular, the first term of this some is smaller than one half of the signal term when $\Delta^2\gtrsim \log(K)+\sqrt{\frac{p}{M}\log(K)}$. 
\item Finally, using Lemma 12 of \cite{Even24}, we get that, with high probability, simultaneously on all $B\in \cB$, $$\langle A\mu E^T +E(A\mu)^T  ,B-B^{*}\rangle \lesssim \sqrt{\<A\mu(A\mu)^{T},B^{*}-B\>}\sqrt{ \|B^*-B^*B\|_{1}\log(\frac{MK^4}{\|B^*-B^*B\|_{1}})}\enspace . $$
\end{enumerate}
In the following, let us denote $\hat{\delta}=\|B^*-B^*\hat{B}\|_{1}$. Combining those three points together, we get that, whenever $\Delta^2\gtrsim \log(K)+\sqrt{\frac{p}{M}\log(K)}$, with high probability,
\begin{align*}
\<A\mu(A\mu)^{T},B^{*}-\hat{B}\>\lesssim& \hat{\delta}\pa{\sqrt{\frac{p}{M}\log(\frac{MK^4}{ \hat{\delta}})}+\log(\frac{MK^4}{ \hat{\delta}})}+\sqrt{\<A\mu(A\mu)^{T},B^{*}-\hat{B}\>}\sqrt{ \hat\delta\log(\frac{MK^4}{\hat\delta})}\enspace,
\end{align*}
which, in turn, implies,
\begin{align*}
\Delta^2\hat{\delta}\lesssim&\hat\delta\pa{\log(\frac{MK^4}{\hat\delta})+\sqrt{\frac{p}{M}\log(\frac{MK^4}{\hat\delta})}}\enspace,
\end{align*}
and so 
$$\hat{\delta}\lesssim MK^4\exp\pa{-c\pa{\Delta^2\wedge\frac{M\Delta^4}{p}}}\enspace,$$
with $c$ some numerical constant. Finally, we end up with 
$$err(\hat{\pi}, \pi^*)\lesssim \frac{\hat{\delta}}{KM}\lesssim K^3\exp\pa{-c\pa{\Delta^2\wedge\frac{M\Delta^4}{p}}}\leq  \exp\pa{-c'\pa{\Delta^2\wedge \frac{\Delta^4 M}{p}}}\enspace, $$ where the last inequality holds for some numerical constant $c'$ if $\Delta^2\gtrsim \log(K)+\sqrt{\frac{p\log(K)}{M}}$. This concludes the (sketch of) proof.

\subsection{Proof of Theorem \ref{thm:lowerboundinformationalfeature}; Perfect Recovery}\label{prf:lowerboundexact}

This section is dedicated to proving the first point of Theorem \ref{thm:lowerboundinformationalfeature}. Without loss of generality, we suppose throughout this proof that $\sigma=1$. We suppose that $K\geq K_{0}$ with $K_{0}$ a constant that we will choose large enough. 

To prove the first point of Theorem \ref{thm:lowerboundinformationalfeature}, we will distinguish two cases. First, we will prove that there exist positive numerical constants $c_{1}$ and $C$ such that, when $\Bar{\Delta}^{2}\leq c_{1}\log(KM)$, we have $$\inf_{\hat{\pi}:\R^{K\times M\times p}\to (\mathcal{S}_{K})^{M}} \sup_{\mu\in\Theta_{\bar{\Delta}}}\sup_{\pi\in (\mathcal{S}_{K})^{M}}\P_{\mu,\pi}(err(\hat{\pi}, \pi)\neq 0)>C\enspace.$$
Second, we will show that there exists a numerical constant $c_{2}$ such that this still holds when $c_{1}\log(KM)\leq \bar{\Delta}^{2}\leq c_{2}\sqrt{\frac{p}{M}\log(KM)}.$

\subsubsection{Case $\bar{\Delta}^{2}\leq c_{1}\log(KM)$.}

Let us suppose that $\bar{\Delta}^{2}\leq c_{1}\log(KM)$ for $c_{1}$ a positive numerical constant that we will choose small enough later. Let $e$ be a unit vector of $\R^{p}$ and define $\mu_{k}=\sqrt{2}k\Bar{\Delta} e$ for $k\in[1,K]$. To prove our statement, we will use Fano's Lemma that we recall here (see e.g p 57 of \cite{HDS2}).
\begin{lem}[Fano's Lemma]\label{lem:fano}
    Let $(\P_{s})_{s\in[\ell]}$ be a set of probability distributions on some set $\mathcal{Y}$. For any probability distribution $\mathbb{Q}$ such that for all $s\in[\ell]$, $\P_{s}<<\mathbb{Q}$, $$\min_{\hat{J}:\mathcal{Y}\to [\ell]}\frac{1}{\ell}\sum_{s=1}^{\ell}\P_{s}\pa{\hat{J}(Y)\neq s}\geq 1-\frac{1+\frac{1}{\ell}\sum_{s=1}^{\ell}KL(\P_{s},\mathbb{Q})}{\log(\ell)}\enspace , $$ where we recall that $KL(\P,\mathbb{Q})=\int \log\pa{\frac{d\P}{d\mathbb{Q}}}d\P$ stands for the Kullback-Leibler divergence between $\P$ and $\mathbb{Q}$. 
\end{lem}

\medskip

Let us consider the distribution $\P_{\mu,id}$ as reference distribution for applying Lemma \ref{lem:fano}. Our goal will be to find different $M$-tuples of permutations $\pi^{(1)},...,\pi^{(\ell)}$, with $\ell$ as large as possible, such that $KL(\P_{\mu, \pi^{(s)}},\P_{\mu,id})$ remains small for all $s\in[\ell]$. To do so, for $s\leq \lfloor \frac{K}{2}\rfloor$ and $m\in[M]$, we define $\pi^{(s,m)}$ the $M$-tuple of permutations defined as follows. If $m'\neq m$, then $\pi^{(s,m)}_{m'}=id$, and $\pi^{(s,m)}_{m}=(2s-1,2s)$ where $(2s-1,2s)$ stands for the transposition of $2s-1$ and $2s$.  

For $s\leq \lfloor \frac{K}{2}\rfloor$ and $m\in[M]$, let us compute $KL(\P_{\mu, \pi^{(s,m)}}, \P_{\mu,id})$. Given $m'\in[M]$ and $k\in[K]$, we denote $\P_{\mu, \pi^{(s,m)};k,m'}$ the marginal distribution of the vector $Y_{k}^{(m')}$ under  $\P_{\mu, \pi^{(s,m)}}$. Using the independence of the $Y_{k}^{(m')}$'s for $k\in[K]$ and $m'\in[M]$, we get 

\begin{align*}
    KL(\P_{\mu, \pi^{(s,m)}}, \P_{\mu,id})=&\sum_{k=1}^{i=K}\sum_{m'=1}^{M}KL(\P_{\mu, \pi^{(s,m)};k,m'}, \P_{\mu,id;k,m'})\\
    =& KL\pa{\P_{\mu, \pi^{(s,m)};2s-1,m},\P_{\mu, id;2s-1,m}}+KL\pa{\P_{\mu, \pi^{(s,m)};2s,m},\P_{\mu, id;2s,m}}\\
    =&KL\pa{\mathcal{N}(\mu_{2s-1},I_p),\mathcal{N}(\mu_{2s},I_{p})}+KL\pa{\mathcal{N}(\mu_{2s},I_p),\mathcal{N}(\mu_{2s-1},I_{p})}\\
    =& 2\bar{\Delta}^2   \leq 2c_1 \log(KM)\enspace .
\end{align*}

For any estimator $\hat{\pi}$, we associate $\hat{j}$ the estimator that gives $(s,m)$ if $\hat{\pi}=\pi^{(s,m)}$ and $(1,1)$ otherwise. Fano's Lemma (Lemma~\ref{lem:fano}) implies that, for all estimators $\hat{\pi}$, the corresponding estimator $\hat{j}$ satisfies $$\frac{1}{M\lfloor \frac{K}{2}\rfloor}\sum_{(s,m)}\P_{\mu,\pi^{(s,m)}}(\hat{j}\neq (s,m))\geq 1-\frac{1+2c_{1} \log(KM)}{\log(M\lfloor \frac{K}{2}\rfloor)}\enspace.$$ If $c_{1}$ is small enough and $K_{0}$ large enough, there exists a constant $C>0$ such that, for all integers $K\geq K_{0}$ and $M\geq 2$, we have $1-\frac{1+2c_{1} \log(KM)}{\log(M\lfloor \frac{K}{2}\rfloor)}\geq C.$

Since, for any estimator $\hat{\pi}$ and its corresponding estimator $\hat{j}$, for any $s\leq \lfloor K/2\rfloor$ and $m\leq M$, we have $\P_{\mu,\pi^{(s,m)}}(\hat{\pi}=\pi^{(s,m)})\leq \P_{\mu,\pi^{(s,m)}}(\hat{j}=(s,m))$, we get $$\frac{1}{M\lfloor \frac{K}{2}\rfloor}\sum_{(s,m)}\P_{\mu,\pi^{(s,m)}}(\hat{\pi}\neq \pi^{(s,m)})\geq C\enspace .$$
This, with the fact that, for all estimator $\hat{\pi}$, $$\sup_{\mu\in\Theta_{\bar{\Delta}}}\sup_{\pi\in(\mathcal{S}_{K})^{M}}\P_{\mu,\pi}(\hat{\pi}\neq \pi)\geq \frac{1}{M\lfloor \frac{K}{2}\rfloor}\sum_{(s,m)}\P_{\mu,\pi^{(s,m)}}(\hat{\pi}\neq \pi^{(s,m)})$$ concludes the proof of our statement in the regime where $\bar{\Delta}^{2}\leq c_{1}\log(KM)$.

\subsubsection{Case $c_{1}\log(KM)\leq \bar{\Delta}^{2}\leq c_{2}\sqrt{\frac{p}{M}\log(KM)}$.}

We suppose that 
\begin{equation}\label{eq:constraint:delta_bar}
c_{1}\log(KM)\leq \bar{\Delta}^{2}\leq c_{2}\sqrt{\frac{p}{M}\log(KM)}\enspace ,
\end{equation}
with $c_{1}$ defined just above and $c_{2}$ that we will choose small enough. Given $\rho$ a probability distribution on $(\R^{p})^{K}$ and a $M$-tuple of permutations $\pi\in(\mathcal{S}_{K})^{M}$, we define the probability distribution on $\R^{K\times M\times p}$ by $$\P_{\rho,\pi}(A)=\int \P_{\mu,\pi}(A)d\rho(\mu)\enspace.$$ In this regime, we shall use the following lemma.

\begin{lem}\label{lem:reductionfano}
    We suppose that there exists a probability distribution $\rho$ on $(\R^{p})^{K}$ and $a>0$ such that $$\inf_{\hat{\pi}:\R^{K\times M\times p}\to (\mathcal{S}_{K})^{M}}\sup_{\pi\in(\mathcal{S}_{K})^{M}}\P_{\rho,\pi}(err(\hat{\pi}, \pi)\neq 0)-\rho(\R^{p}\setminus\Theta_{\bar{\Delta}})>a\enspace.$$
    Then, we have $$\inf_{\hat{\pi}:\R^{K\times M\times p}\to (\mathcal{S}_{K})^{M}}\sup_{\mu\in\Theta_{\Bar{\Delta}}}\sup_{\pi\in(\mathcal{S}_{K})^{M}}\P_{\mu,\pi}(err(\hat{\pi}, \pi)\neq 0)>a\enspace.$$
\end{lem}

We refer to Section \ref{prf:reductionfano} for a proof of this lemma. We will consider the following distribution on $(\R^{p})^{K}$. Define $\eps=2\sqrt{\frac{1}{p}}\bar{\Delta}$ and define $\rho$ the uniform distribution on the hypercube $\mathcal{E}=\{-\eps,\eps\}^{p\times K}$. We will use Fano's Lemma to lower bound $\inf_{\hat{\pi}:\R^{K\times M\times p}\to (\mathcal{S}_{K})^{M}}\sup_{\pi\in(\mathcal{S}_{K})^{M}}\P_{\rho,\pi}(err(\hat{\pi}, \pi)\neq 0)$. To do so, we need to find many $M$-tuple of permutations $\pi^{(1)},\ldots,\pi^{(\ell)}\in(\mathcal{S}_{K})^{M}$, with $\ell$ large, such that $KL(\P_{\rho,\pi^{(s)}},\P_{\rho,id})$ remains small for all $s\in [\ell]$. 

Given $m\in[M]$, $k< k'\in [K]$, we define $\pi^{(m,k,k')}$ the $M$-tuple of permutations that satisfies; for all $m'\neq m$, $\pi^{(m,k,k')}_{m'}=id$, and $\pi^{(m,k,k')}_{m}$ is the transposition $(k,k')$. We denote $V$ the set of all these $M$-tuple of permutation; $V=\{\pi^{(m,k,k')},\enspace m\in[M-1],\enspace k< k'\in [K]\}$. For a $M$-tuple of permutations $\pi\in V$, let us compute the quantity $KL(\P_{\rho,\pi},\P_{\rho,id})=\int\log (\frac{d\P_{\rho,\pi}}{d\P_{\rho,id}})d\P_{\rho,\pi}$. We refer to Section \ref{prf:calculcompliquéexactrecovery} for a proof of the next lemma.

\begin{lem}\label{lem:calculcompliquéexactrecovery}
We suppose that  the numerical constant $c_2$ is small enough with respect to $c_1$, where both $c_1$ and $c_2$ arise in~\eqref{eq:constraint:delta_bar}. Then, there exists a numerical constant $c>0$ such that, for all $m,k,k'$ with $m\in [M]$, $k<k'\in [K]$, we have  
\[
KL(\P_{\rho,\pi},\P_{\rho,id})=\int\log (\frac{d\P_{\rho,\pi^{(m,k,k')}}}{d\P_{\rho,id}})d\P_{\rho,\pi^{(m,k,k')}}\leq cc_2^2\log(KM)\enspace . 
\]
\end{lem}

This inequality combined with Lemma \ref{lem:fano} applied to the set $V$ of all the $M$-tuple of permutations $\pi^{(m,k,k')}$ leads to $$\inf_{\hat{\pi}:\R^{^{K\times M\times p}}\to (\mathcal{S}_{K})^{M}}\frac{1}{|V|}\sum_{\pi\in V}\P_{\rho,\pi}[\hat{\pi}\neq \pi]\geq 1-\frac{1+cc_{2}^{2}\log(Km)}{\log(|V|)}\enspace.$$ For all estimator $\hat{\pi}:\R^{K\times M\times p}\to (\mathcal{S}_{K})^{M}$, let us define $\psi_{\hat{\pi}}\in \mathcal{S}_K$ the permutation on $[K]$ such that $\psi_{\hat{\pi}}\circ \hat{\pi}_M=id$. We remark that, for any $\pi\in V$, having $err(\hat{\pi}, \pi)=0$ is equivalent to having $\pi=\psi_{\hat{\pi}}$. Thus, $$\inf_{\hat{\pi}:\R^{K\times M\times p}\to (\mathcal{S}_{K})^{M}}\frac{1}{|V|}\sum_{\pi\in V}\P_{\rho,\pi}[err(\hat{\pi}, \pi)\neq 0]\geq 1-\frac{1+cc_{2}^{2}\log(Km)}{\log(|V|)}\enspace,$$
and so
$$\inf_{\hat{\pi}:\R^{K\times M\times p}\to (\mathcal{S}_{K})^{M}}\sup_{\pi\in(\mathcal{S}_{K})^{M}}\P_{\pi,\rho}[err(\hat{\pi}, \pi)\neq 0]\geq 1-\frac{1+cc_{2}^{2}\log(Km)}{\log(|V|)}\enspace.$$
Finally, the fact that $\log(|V|)= \log(K(K-1)M/2)$ implies that, if we choose $c_{2}$ small enough and if $K_{0}$ is large enough, we will dispose of a numerical constant $b>0$ satisfying $$\inf_{\hat{\pi}:\R^{K\times M\times p}\to (\mathcal{S}_{K})^{M}}\frac{1}{|V|}\sum_{\pi\in V}\P_{\rho,\pi}[err(\hat{\pi}, \pi)\neq 0]\geq b\enspace.$$ In order to apply Lemma \ref{lem:reduction}, it remains to control $\rho(\pa{\R^{p}}^K\setminus \Theta_{\bar{\Delta}})$. From Lemma 14 of \cite{Collier16}, we get that $\rho(\pa{\R^{p}}^K\setminus \Theta_{\Bar{\Delta}})\leq \frac{K(K-1)}{2}e^{\frac{-p}{8}}$. Moreover, $c_{1}\log(KM)\leq c_{2}\sqrt{\frac{p}{M}\log(KM)}$ implies $p\geq \frac{c_{1}^{2}}{c_{2}^{2}}\log(KM)$. Thus, $\rho(\R^{p}\setminus \Theta_{\Bar{\Delta}})\leq \frac{b}{2}$, provided $c_{2}$ is small enough compared to $c_{1}$. Hence, Lemma \ref{lem:reductionfano} gives $$\inf_{\hat{\pi}:\R^{K\times M\times p}\to (\mathcal{S}_{K})^{M}}\sup_{\pi\in(\mathcal{S}_{K})^{M}}\P_{\rho,\pi}[\hat{\pi}\neq \pi]\geq \frac{b}{2}\enspace.$$ This concludes the proof of the first point of Theorem \ref{thm:lowerboundinformationalfeature} (perfect recovery).

\subsubsection{Proof of Lemma \ref{lem:calculcompliquéexactrecovery}}\label{prf:calculcompliquéexactrecovery}

By symmetry, it is sufficient to bound the Kullback-Leibler discrepansy for $\pi=\pi^{(1,1,2)}=\pa{(1,2), id,\ldots,id}$. In the following, we denote $\rho'$ the probability distribution on $(\R)^{p\times K}$ that satisfies; if $(\mu_{1},\ldots,\mu_{K})\sim \rho'$, all the $\mu_{i}$'s are independent, $\mu_{1}=\mu_{2}=0$ and all the other $\mu_{i}$'s are drawn uniformly on the set $\{\eps,-\eps\}^{p}$. 

First, we will compute the quantity \begin{equation}\label{eq:vraisemblance}
    \frac{d\P_{\rho,\pi}}{d\P_{\rho,id}}=\frac{\frac{d\P_{\rho,\pi}}{d\P_{\rho',\pi}}}{\frac{d\P_{\rho,id}}{d\P_{\rho',\pi}}}=\frac{\frac{d\P_{\rho,\pi}}{d\P_{\rho',\pi}}}{\frac{d\P_{\rho,id}}{d\P_{\rho',id}}}\enspace,
\end{equation} where the second equality comes from the fact that $\P_{\rho',id}=\P_{\rho',\pi}$. Given a probability law $\P$ on some Euclidean space, we denote $d\P$ the density of this law with respect to the Lebesgue measure on this space (when $\P$ is absolutely continuous with respect to the Lebesgue measure).  For the numerator in \eqref{eq:vraisemblance}, we have,
\begin{align*}
    \frac{d\P_{\rho,\pi}}{d\P_{\rho',\pi}}(Y)=&\frac{\E_{\rho}\cro{d\P_{\mu,\pi}(Y)}}{\E_{\rho'}\cro{d\P_{\mu,\pi}(Y)}}\\
    =&\frac{\E_{\rho}\cro{\prod_{k\in[1,K]}\prod_{m\in[M]}\exp\pa{-\frac{1}{2}\|Y^{(m)}_{\pi_{m}^{-1}(k)}-\mu_{k}\|^{2}}}}{\E_{\rho'}\cro{\prod_{k\in[1,K]}\prod_{m\in[M]}\exp\pa{-\frac{1}{2}\|Y^{(m)}_{\pi_{m}^{-1}(k)}-\mu_{k}\|^{2}}}}\\
    =&\frac{\E_{\rho}\cro{\prod_{d\in[1,p]}\prod_{k\in[K]}\prod_{m\in[M]}\exp\pa{-\frac{1}{2}\|X^{(m)}_{\pi_{m}^{-1}(k),d}-\mu_{k,d}\|^{2}}}}{\E_{\rho'}\cro{\prod_{d\in[1,p]}\prod_{k\in[K]}\prod_{m\in[M]}\exp\pa{-\frac{1}{2}\|Y^{(m)}_{\pi_{m}^{-1}(k),d}-\mu_{k,d}\|^{2}}}}\enspace.
\end{align*}
Using the independence of the $\mu_{k,d}$'s both for the law $\rho$ and $\rho'$ together with the fact that, when $k>3$, $\mu_{k,d}$ has the same law under $\rho$ and $\rho'$, we get that 
\begin{align*}
    \frac{d\P_{\rho,\pi}}{d\P_{\rho',\pi}}(Y)=&\prod_{d\in[p]}\prod_{k\in[K]}\frac{\E_{\rho}\cro{\prod_{m\in[M]}\exp\pa{-\frac{1}{2}(Y^{(m)}_{\pi_{m}^{-1}(k),d}-\mu_{k,d})^{2}}}}{\E_{\rho'}\cro{\prod_{m\in[M]}\exp\pa{-\frac{1}{2}(Y^{(m)}_{\pi_{m}^{-1}(k),d}-\mu_{k,d})^{2}}}}\\
    =&\prod_{d\in[p]}\prod_{k\in\{1,2\}}\frac{\E_{\rho}\cro{\prod_{m\in[M]}\exp\pa{-\frac{1}{2}(Y^{(m)}_{\pi_{m}^{-1}(k),d}-\mu_{k,d})^{2}}}}{\E_{\rho'}\cro{\prod_{m\in[M]}\exp\pa{-\frac{1}{2}(Y^{(m)}_{\pi_{m}^{-1}(k),d}-\mu_{k,d})^{2}}}}\\
    =&\prod_{d\in[p]}\prod_{k\in\{1,2\}}\E_{\rho}\cro{\prod_{m\in[M]}\exp{\pa{-\frac{1}{2}\pa{(Y^{(m)}_{\pi_{m}^{-1}(k),d}-\mu_{k,d})^{2}-(Y^{(m)}_{\pi_{k}^{-1}(k),d})^{2}}}}}\\
    =&\prod_{d\in[p]}\prod_{k\in\{1,2\}}\E_{\rho}\cro{\prod_{m\in[M]}\exp{\pa{Y^{(m)}_{\pi_{m}^{-1}(k),d}\mu_{k,d}-\frac{\eps^{2}}{2}}}}\\
    =&\prod_{d\in[p]}e^{\frac{-M\eps^{2}}{2}}\cosh\left[\eps \left(Y^{(1)}_{2,d}+ \sum_{m\in[2,M]} Y^{(m)}_{1,d} \right)\right]
    \cdot \prod_{d\in[p]}e^{\frac{-M\eps^{2}}{2}}\cosh\left[\eps \left(Y^{(1)}_{1,d}+ \sum_{m\in[2,M]} Y^{(m)}_{2,d} \right)\right]\enspace,\\
\end{align*}
where the third equality comes from the fact that, for all $d\in[1,p]$ and when $k\leq2$, $\mu_{k,d}=0$ almost surely under the law $\rho'$. Similarly, we get that 
\begin{align*}
    \frac{d\P_{\rho,id}}{d\P_{\rho',id}}(Y)=&\prod_{d\in[p]}e^{\frac{-M\eps^{2}}{2}}\cosh{\pa{\sum_{m\in[M]}\eps Y^{(m)}_{1,d}}}\cdot  \prod_{d\in[p]}e^{\frac{-M\eps^{2}}{2}}\cosh{\pa{\sum_{m\in[M]}\eps Y^{(m)}_{2,d}}}\enspace.\\
\end{align*}
Combining these two equalities in \eqref{eq:vraisemblance} leads to 
$$\frac{d\P_{\rho,\pi}}{d\P_{\rho,id}}=\frac{\prod_{d\in[p]}\cosh{\pa{\eps\sum_{m\in[2,M]}Y^{(m)}_{1,d}+\eps Y^{(1)}_{2,d}}}\cosh{\pa{\eps\sum_{m\in[2,M]}Y^{(m)}_{2,d}+\eps Y^{(1)}_{1,d}}}}{\prod_{d\in[p]}\cosh{\pa{\eps\sum_{m\in[M]}Y^{(m)}_{1,d}}}\cosh{\pa{\eps\sum_{m\in[M]}Y^{(m)}_{2,d}}}}\enspace.$$
We denote $\phi$ the standard Gaussian density $\phi(x)=\frac{1}{\sqrt{2\pi}}e^{\frac{-x^{2}}{2}}$. Under the law $\P_{\rho,\pi}$, conditionally on $\mu\sim\rho$, we have that: 
\begin{itemize}
    \item $\sum_{m\in[M]}Y^{(M)}_{1,d}\sim \mathcal{N}((M-1)\mu_{1,d}+\mu_{2,d},M)$,
    \item $\sum_{M\in[M]}Y^{(M)}_{2,d}\sim \mathcal{N}((M-1)\mu_{2,d}+\mu_{1,d},M)$,
    \item $\sum_{m\in[2,M]}Y^{(m)}_{1,d}+Y^{(1)}_{2,d}\sim \mathcal{N}(M\mu_{1,d},M)$,
    \item $\sum_{m\in[2,M]}Y^{(m)}_{2,d}+Y^{(1)}_{1,d}\sim \mathcal{N}(M\mu_{2,d},M)$.
\end{itemize}
This leads to 
\begin{align*}
    KL(\P_{\rho, \pi}, \P_{\rho,id})=&\E_{\rho,\pi}\cro{\log\pa{\frac{d\P_{\rho,\pi}}{d\P_{\rho,id}}}}\\
    =&\sum_{d\in[p]}\E_{\rho,\pi}\cro{\log\cosh\pa{\eps\sum_{m\in[2,M]}Y^{(m)}_{2,d}+\eps Y^{(1)}_{1,d}}}\\
    &+\sum_{d\in[p]}\E_{\rho,\pi}\cro{\log\cosh\pa{\eps\sum_{m\in[2,M]}Y^{(m)}_{1,d}+\eps Y^{(1)}_{2,d}}}\\
    &-\sum_{d\in[p]}\E_{\rho,\pi}\cro{\log\cosh\pa{\eps\sum_{m\in[M]}Y^{(m)}_{1,d}}}\\
    &-\sum_{d\in[p]}\E_{\rho,\pi}\cro{\log\cosh\pa{\eps\sum_{m\in[M]}Y^{(m)}_{2,d}}}\\
    =& 2p \E_{\rho,\pi}\cro{\log\cosh\pa{\eps \sum_{m\in[2,M]}Y^{(m)}_{1,1}+Y^{(1)}_{2,1}}}-2p \E_{\rho,\pi}\cro{\log\cosh\pa{\eps \sum_{m\in[M]}Y^{(m)}_{1,1}}}\\
    =&2p \E_{\rho}\cro{\int \log\cosh(\eps(M\mu_{1,1}+\sqrt{M}x))\phi(x)dx}\\
    &-2p \E_{\rho}\cro{\int \log\cosh(\eps((M-1)\mu_{1,1}+\mu_{2,1}+\sqrt{M}x))\phi(x)dx}\enspace.    
\end{align*}
First, let us upper-bound the term $$A:=\E_{\rho}\cro{\int \log\cosh(\eps(M\mu_{1,1}+\sqrt{M}x))\phi(x)dx}\enspace.$$ We denote $u=\eps\pa{ (M-1) \mu_{1,1}+\sqrt{M}x}$ and $h=\eps \mu_{1,1}$. Then, $A=\E_{\rho}\cro{\int \log\cosh(u+h)\phi(x)dx}\enspace.$ We will use the Taylor expansion of the function $\log\cosh$ around $u$. We compute the following derivatives: \begin{itemize}
    \item For all $x\in \R$, $\log\cosh'(x)=\tanh(x)$,
    \item For all $x\in\R$, $\log\cosh ''(x)=1-\tanh^{2}(x)$ which is bounded by $2$ in absolute value.
\end{itemize}
Hence, Taylor-Lagrange inequality implies \begin{equation}\label{eq:Taylor}
    |\log\cosh(x+y)- \log\cosh(x)- \tanh(x)y|\leq y^{2}, \enspace \forall (x,y)\in\R^{2}\enspace.
\end{equation}
Plugging this inequality leads to \begin{equation}\label{eq:Taylor_1}
    A=\E_{\rho}\cro{\int \log\cosh(u+h)\phi(x)dx}\leq \E_{\rho}\cro{\log\cosh(u)}+\E_{\rho}\cro{\int\tanh{(u)}h\phi(x)dx}+ \E_{\rho}(h^{2})\enspace.
\end{equation}

First, since $h^{2}=\eps^{4}$, we have $\E_{\rho}(h^{2})=\eps^{4}$. Now, we need to upper bound $\E_{\rho}\cro{\int\tanh{(u)}h\phi(x)dx}$. For any $y\in\R$, we have $\tanh'(y)=1-tanh^{2}(y)$ and $\tanh''(y)=-2\tanh(y)(1-\tanh(y)^{2})$, which is bounded by $4$ in absolute value. Hence, Taylor-Lagrange inequality taken at $0$ leads to $$ |\tanh(y)-y|\leq 2y^{2}\enspace, \forall y\in\R\enspace.$$

This leads us to $\E_{\rho}\cro{\int \tanh{(u)}h\phi(x)dx}\leq \E_{\rho}\cro{\int uh\phi(x)dx}+2\E_{\rho}\cro{\int u^{2}|h|\phi(x)dx}$. On the one hand, 
\begin{align*}
    \E_{\rho}\cro{\int uh\phi(x)dx}=&\eps^{2}\E_{\rho}\cro{\int((M-1)\mu_{1,1}+\sqrt{M}x)\mu_{1,1}\phi(x)dx}\\
    =&\eps^{2}\E_{\rho}\cro{(M-1)\mu_{1,1}^{2}}\\
    =&(M-1)\eps^{4}\enspace.    
\end{align*}
On the other hand, we have 
\begin{align*}
    \E_{\rho}[\int u^{2}|h|\phi(x)dx]=&\eps^{4}\E_{\rho}[\int((M-1)\mu_{1,1}+\sqrt{M}x)^{2}\phi(x)dx]\\
    =&\eps^{4}\pa{(M-1)^{2}\eps^{2}+M}\enspace.
\end{align*}
Plugging all these inequalities in \eqref{eq:Taylor_1} leads us to 
\begin{equation}\label{eq:KL_2}
   A\leq \E_{\rho}\cro{\log\cosh(u)}+(M-1)\eps^{4}+2\eps^{4}\pa{(M-1)^{2}\eps^{2}+M}+ \eps^{4}\enspace.
\end{equation}
Now, let us lower-bound the term $$B:=\E_{\rho}\cro{\int \log\cosh(\eps(m\mu_{1,1}+\mu_{2,1}+\sqrt{m+1}x))\phi(x)dx}=\E_{\rho}\cro{\int \log\cosh(u+h')\phi(x)dx}\enspace,$$
where we define $h'=\eps \mu_{2,1}$ which is independent of $u$. Using inequality \eqref{eq:Taylor} together with the independence of $u$ and $h'$ leads to $$\E_{\rho}\cro{\int \log\cosh(u+h')\phi(x)dx}\geq \E_{\rho}\cro{\log\cosh{u}}+\E_{\rho}\cro{\int \tanh{(u)}\phi(x)dx}\E_{\rho}\cro{h'}- \E_{\rho}\cro{h'^{2}}\enspace.$$
Since $\E_{\rho}\cro{h'}=0$ and $\E_{\rho}\cro{h'^{2}}=\eps^{4}$, we have \begin{equation}\label{eq:KL_1}
    \E_{\rho}\cro{\int \log\cosh(u+h')\phi(x)dx}\geq \E_{\rho}\cro{\log\cosh{u}}-\eps^{4}\enspace.
\end{equation} 
Hence, plugging inequalities \eqref{eq:KL_2} and \eqref{eq:KL_1} in the expression of $KL(\P_{\rho, \pi}, \P_{\rho,id})$ leads us to 
\begin{align*}
    KL(\P_{\rho, \pi}, \P_{\rho,id})=&2p(A-B)\leq 2\pa{2\eps^{4}+(M-1)\eps^{4}+2\eps^{4}((M-1)^{2}\eps^{2}+M)}\\
    \leq& c\pa{pM\eps^{4}\pa{1+M\eps^{2}}}\enspace,
\end{align*}
for some numerical constant $c$. Since $\eps^{2}=\frac{4}{p}\bar{\Delta}^{2}$, for some numerical constant $c$ that may differ, we have $$KL(\P_{\rho, \pi}, \P_{\rho,id})\leq c\pa{ \bar{\Delta}^{4}\frac{M}{p}(1+\frac{M}{p}\bar{\Delta}^{2})}\enspace.$$ 
The hypothesis $c_{1}\log(KM)\leq\bar{\Delta}^{2}\leq c_{2}\sqrt{\frac{p}{M}\log(KM)}$ gives us 
$$\bar{\Delta}^{4}\frac{M}{p}\leq c_{2}^{2}\log(Km)\enspace ,$$ and  $$\frac{M}{p}\bar{\Delta}^{2}=\frac{M}{p}\bar{\Delta}^{4}\frac{1}{\bar{\Delta}^{2}}\leq \frac{c_{2}^{2}\log(KM)}{c_{1}\log(KM)}\leq 1\enspace ,$$ 
provide that we chose $c_{2}\leq \sqrt{c_1}$. Thus, there exists a numerical constant $c$ such that for all $\pi\in V$, we have $$KL(\P_{\rho, \pi}, \P_{\rho,id})\leq c c_{2}^{2}\log(Km)\enspace.$$

\subsubsection{Proof of Lemma \ref{lem:reductionfano}}\label{prf:reductionfano}

Let us suppose that there exists a probability distribution $\rho$ on $\R^{K\times M\times p}$ and $a>0$ such that $$\inf_{\hat{\pi}:R^{K\times M\times p}\to (\mathcal{S}_{K})^{M}}\sup_{\pi\in(\mathcal{S}_{K})^{M}}\P_{\rho,\pi}(err(\hat{\pi}, \pi)\neq 0)-\rho(\pa{\R^{p}}^K\setminus\Theta_{\bar{\Delta}})\geq a\enspace.$$
Given an estimator $\hat{\pi}:\R^{K\times M\times p}\to (\mathcal{S}_{K})^{M} $, the previous hypothesis directly implies that there exists $\pi\in(\mathcal{S}_{K})^{M}$ such that $\P_{\rho,\pi}(err(\hat{\pi}, \pi)\neq 0)-\rho(\pa{\R^{p}}^K\setminus\Theta_{\bar{\Delta}})\geq a.$ 

By definition, $\P_{\rho,\pi}(err(\hat{\pi}, \pi)\neq 0)=\int \P_{\mu, \pi}(err(\hat{\pi}, \pi)\neq 0)d\rho(\mu).$ The quantity $\P_{\mu, \pi}(err(\hat{\pi}, \pi)\neq 0)$ being bounded by $1$, we have $\P_{\rho,\pi}(err(\hat{\pi}, \pi)\neq 0)\leq \int_{\Theta_{\bar{\Delta}}} \P_{\mu, \pi}(err(\hat{\pi}, \pi)\neq 0)d\rho(\mu)+\rho((\R^{p})^{K}\setminus \Theta_{\bar{\Delta}}).$ Therefore, $\int_{\Theta_{\bar{\Delta}}} \P_{\mu, \pi}(err(\hat{\pi}, \pi)\neq 0)d\rho(\mu)\geq a.$ This implies the existence of $\mu\in\Theta_{\bar{\Delta}}$ such that $\P_{\mu, \pi}(\hat{\pi}\neq \pi)\geq a$. This being true for all estimator $\hat{\pi}$, we get the following inequality that concludes the Lemma $$\inf_{\hat{\pi}:\R^{K\times M\times p}\to (\mathcal{S}_{K})^{M}}\sup_{\mu\in\Theta_{\bar{\Delta}}}\sup_{\pi\in(\mathcal{S}_{K})^{M}}\P_{\mu,\pi}(err(\hat{\pi}, \pi)\neq 0)\geq a\enspace.$$

\subsection{Proof of Theorem \ref{thm:lowerboundinformationalfeature}; Partial Recovery}
\label{prf:lowerboundpartial}

In this section, we prove the second point of Theorem \ref{thm:lowerboundinformationalfeature}. Let us suppose that $p\geq c \log(K)$ and $K\geq K_{0}$, with $c$ and $K_{0}$ two numerical constants that we will choose large enough later. Without loss of generality, we suppose throughout this proof that $\sigma=1$.

To prove our result, we will apply the following lemma that we prove in Section \ref{prf:numberpermutations} together with a reduction lemma.

\begin{lem}\label{lem:numberpermutations}
    Suppose that $K\geq K_0$ where $K_0$ is a numerical constant large enough. There exists $V$ a set of $M$-tuple of permutations such that the following properties hold.
    \begin{itemize}
        \item For all $\pi\in V$ and $M\leq \lceil \frac{M}{2}\rceil$, $\pi_{m}=id$,
        \item There exists a numerical constant $a$ such that for any two different $M$-tuple of permutations $\pi$ and $\pi'$ in $V$, $\frac{1}{MK}\sum_{m=1}^{M}\sum_{k=1}^{K}\1_{\pi'_{m}(k)\neq \pi_{m}(k)}\geq a$,
        \item There exists a numerical constant $c$ such that $\log(|S|)\geq cMK\log(K)$.
    \end{itemize}
\end{lem}

In the following, we suppose that $K_{0}\geq 4$ and we recall that we consider $K\geq K_0$. Here is the reduction lemma that we will apply in this proof. This lemma plays the same role in this proof as does Lemma \ref{lem:reductionfano} in Section \ref{prf:lowerboundexact}.

\begin{lem}\label{lem:reductionfano2}
    Suppose that there exists a probability distribution $\rho$ on $(\R^{p})^K$ and $C>0$ satisfying $$\inf_{\hat{\pi}:\R^{K\times M\times p}\to (\mathcal{S}_{K})^{M}}\sup_{\pi\in(\mathcal{S}_{K})^{M}}\E_{\rho,\pi}\cro{err(\hat{\pi},\pi)}-\rho((\R^{p})^{K}\setminus \Theta_{\bar{\Delta}})\geq C\enspace.$$ Then, we have $$\inf_{\hat{\pi}:\R^{K\times M\times p}\to (\mathcal{S}_{K})^{M}}\sup_{\mu\in\Theta_{\bar{\Delta}}}\sup_{\pi\in(\mathcal{S}_{K})^{M}}\E_{\mu,\pi}\cro{err(\hat{\pi},\pi)}\geq C\enspace.$$
\end{lem}

We will also need the following lemma which is a consequence of Fano's lemma. We refer to Section \ref{sec:prfFano2} for its proof.

\begin{lem}\label{lem:Fano2}
    Let $\rho$ be a probability distribution on $(\R^{p})^{K}$. For any finite set $A$ of $(\mathcal{S}_{K})^{M}$, we have the following inequality $$\inf_{\hat{\pi}:\R^{K\times M\times p}\to (\mathcal{S}_{K})^{M}}\frac{2}{|A|}\sum_{\pi\in A}\E_{\rho,\pi}\cro{err(\hat{\pi},\pi)}\geq \min_{\pi\neq \pi'\in A}err(\pi,\pi')\pa{1-\frac{1+\frac{1}{|A|}\sup_{\pi\in A}KL(\P_{\rho,\pi},\P_{\rho,id})}{\log(|A|)}}\enspace.$$
\end{lem}
As in Section \ref{prf:lowerboundexact}, we distinguish two cases. In the first one, we suppose that there exists a constant $c_{1}$ that we will choose small enough such that $\bar{\Delta}^{2}\leq c_{1}\log(K)$. In the second one, we will suppose that $c_{1}\log(K)\leq \bar{\Delta}^{2}\leq c_{2}\sqrt{\frac{p}{M}\log(K)}$, for $c_{2}$ a numerical constant that we will also choose small enough.

\subsection{Case $\bar{\Delta}^{2}\leq c_{1}\log(K)$}

In this section, we suppose that $\bar{\Delta}^{2}\leq c_{1}\log(K)$. We suppose that the constant $c$ such that $p\geq c\log(K)$ is large enough for having $\frac{p}{4}\log(2)\geq \log(K)$. Our choice of the centers $\mu_{1},\ldots, \mu_{K}$ relies on the following lemma, proved in Section \ref{sec:prfchoicepoints}.

\begin{lem}\label{lem:choicepoints}
    If $\frac{p}{4}\log(2)\geq \log(K)$, there exists $\mu_{1},\ldots ,\mu_{K}$ in $\R^{p}$ such that $\frac{1}{2}\min_{k\neq l}\|\mu_{k}-\mu_{l}\|^2\geq \bar{\Delta}^2$ and $\frac{1}{2}\max_{k\neq l}\|\mu_{k}-\mu_{l}\|^2\leq 4\bar{\Delta}^2$.
\end{lem}

We consider $\mu=(\mu_{1},\ldots ,\mu_{K})$ in $(\R^{p})^{K}$ that satisfies the conditions of this lemma. Given $\pi$ taken in the set $V$ defined in Lemma \ref{lem:numberpermutations}, let us compute $KL(\P_{\mu,\pi},\P_{\mu,id})$. We denote $\P_{\mu,\pi;k,m}$ the marginal distribution of $Y_{k}^{(m)}$ under the joint law $\P_{\mu,\pi}$. By independance of all the features, we have that $$KL(\P_{\mu,\pi},\P_{\mu,id})=\sum_{k=1}^{K}\sum_{m=1}^{M}KL(\P_{\mu,\pi;k,m},\P_{\mu,id;k,m})\enspace.$$ Given $k\in[K]$ and $m\in[M]$, $KL(\P_{\mu,\pi;k,m},\P_{\mu,id;k,m})=\frac{\|\mu_{k}-\mu_{\pi_{m}(k)}\|^{2}}{2}$. Hence, we have $$KL(\P_{\mu,\pi},\P_{\mu,id})\leq 4KM \bar{\Delta}^{2}\leq 4c_{1} KM\log(K)\enspace.$$ Combining this with Lemma \ref{lem:Fano2} and Lemma \ref{lem:numberpermutations} leads us to $$\inf_{\hat{\pi}:\R^{K\times M\times p}\to (\mathcal{S}_{K})^{M}}\frac{2}{|V|}\sum_{\pi\in V}\E_{\mu,\pi}\cro{err(\hat{\pi},\pi)}\geq a\pa{1-\frac{1+2c_{1}^{2}KM\log(K)}{cKM\log(K)}}\enspace,$$ where $a$ and $c$ are positive numerical constants. The quantity $a\pa{1-\frac{1+2c_{1}^{2}KM\log(K)}{cM\log(K)}}$ being larger than $\frac{a}{2}$ if $c_{1}$ is small enough and $K_{0}$ large enough, we conclude that there exists a constant $C$ such that $$\inf_{\hat{\pi}:\R^{K\times M\times p}\to(\mathcal{S}_{K})^{M}}\sup_{\mu\in\Theta_{\bar{\Delta}}}\sup_{\pi\in(\mathcal{S}_{K})^{M}}\E_{\mu,\pi}(err(\hat{\pi},\pi))\geq C\enspace.$$

\subsection{Case $c_{1}\log(K)\leq \bar{\Delta}^{2}\leq c_{2}\sqrt{\frac{p}{M}\log(K)}$}

In this section, we suppose that $c_{1}\log(K)\leq \bar{\Delta}^{2}\leq c_{2}\sqrt{\frac{p}{M}\log(K)}$ with $c_{2}$ a numerical constant that we will suppose small enough compared $c_1$ --$c_1$ now being considered as fixed. We still suppose that $K\geq K_0$, for $K_0$ a large enough numerical constant. As in Section \ref{prf:lowerboundexact}, we denote $\rho$ the uniform distribution on $\mathcal{E}=\{\eps, -\eps\}^{pK}$, where $\eps=2\sqrt{\frac{1}{p}}\bar{\Delta}$.

\begin{lem}\label{lem:calculcompliquépartialrecovery}
We suppose that $c_2$ is small with respect to $c_1$ and that $c_{1}\log(K)\leq \bar{\Delta}^{2}\leq c_{2}\sqrt{\frac{p}{M}\log(K)}$. Then, there exists a numerical constant $c>0$ such that, for all $\pi\in V$, we have $$KL\pa{\P_{\rho, \pi},\P_{\rho,id}}\leq cc_2^2KM\log(K)$$
\end{lem}

We refer to Section \ref{prf:calculcompliquépartialrecovery} for a proof of this lemma. Combining Lemma \ref{lem:calculcompliquépartialrecovery} and Lemma \ref{lem:Fano2} implies that 
$$\inf_{\hat{\pi}:\R^{K\times M\times p}\to (\mathcal{S}_{K})^{M}}\frac{2}{|V|}\sum_{\pi\in V}\E_{\rho,\pi}\cro{err(\hat{\pi},\pi)}\geq \min_{\pi\neq \pi'\in V}err(\pi,\pi')\pa{1-\frac{1+c c_{2}^{2}KM\log(K)}{c'Km\log(K)}}\enspace,$$ 
where $c$ and $c'$ are two numerical constants. Given $\pi\neq \pi'\in V$, for $m\in[1,\lceil \frac{M}{2}\rceil]$, $\pi_{m}=\pi'_{m}$. Let $\psi$ be a permutation of $[K]$. Then, for $k\in[K]$, if $\psi(k)\neq k$, we get $\sum_{m=1}^{M}\1_{\psi(\pi_{m}(k))\neq \pi'_{m}(k)}\geq \lceil M/2\rceil \geq \sum_{m=1}^{M}\1_{\pi_{m}(k)\neq \pi'_{m}(k)}$. Hence, from the definition \eqref{eq:errorperm}, we have $err(\pi,\pi')=\frac{1}{MK}\sum_{i=1}^{K}\sum_{j=1}^{m}\1_{\pi_{j}(i)\neq \pi'_{j}(i)}$. From the definition of $V$ in Lemma \ref{lem:numberpermutations}, we deduce that $err(\pi,\pi')\geq a$, where $a$ is a numerical constant. Hence, $$\inf_{\hat{\pi}:\R^{K\times M\times p}\to (\mathcal{S}_{K})^{M}}\frac{2}{|V|}\sum_{\pi\in V}\E_{\rho,\pi}\cro{err(\hat{\pi},\pi)}\geq a\pa{1-\frac{1+cc_{2}^{2}Km\log(K)}{c'KM\log(K)}}\enspace.$$

This last quantity is larger than $\frac{a}{2}$ provided $c_{2}$ is small enough and $K_0$ large enough. We can conclude using Lemma \ref{lem:reductionfano2} together with the fact that $\rho(\pa{\R^{p}}^K\setminus \theta_{\Delta})\leq \frac{K(K-1)}{2}e^{\frac{-p}{8}} \leq \frac{a}{4}$ provided the constant $c$ such that $p\geq c\log(K)$ is large enough. We refer to Lemma 14 of \cite{Collier16} for the inequality $\rho(\R^{p}\setminus \theta_{\Delta})\leq \frac{K(K-1)}{2}e^{\frac{-p}{8}}$.

\subsection{Proof of Lemma \ref{lem:calculcompliquépartialrecovery}}\label{prf:calculcompliquépartialrecovery}

For $\pi\in V$, let us compute $KL(\P_{\rho, \pi},\P_{\rho,id})$. We recall that $KL(\P_{\rho, \pi},\P_{\rho,id})=\int \log\pa{\frac{d\P_{\rho,G}}{d\P_{\rho,\overline{G}}}}d\P_{\rho,G}$. We write $\P_{0,\pi}$, or equivalently $\P_{0,id}$, the distribution of $Y\in \R^{p\times [K]\times [M]}$ when $\mu_{1}=\ldots=\mu_{K}=0$ almost surely. Under this distribution, the entries of $Y$ are drawn independently according $\cN\pa{0,I_p}$. First, we will compute the quantity 
\begin{equation}\label{eq:vraisemblancepartial}
    \frac{d\P_{\rho,\pi}}{d\P_{\rho,id}}=\frac{\frac{d\P_{\rho,\pi}}{d\P_{0,\pi}}}{\frac{d\P_{\rho,id}}{d\P_{0,\pi}}}=\frac{\frac{d\P_{\rho,\pi}}{d\P_{0,\pi}}}{\frac{d\P_{\rho,id}}{d\P_{0,id}}}\enspace,
\end{equation} 
where the second equality comes from the fact that $\P_{0,\pi}=\P_{0,id}$. Given a probability distribution  $\P$ on some Euclidean space, which is absolutely continuous with respect to the Lebesgue measure, we write $d\P$ for the density of this distribution with respect to the Lebesgue measure.
  For the numerator in \eqref{eq:vraisemblancepartial}, we have
\begin{align*}
    \frac{d\P_{\rho,\pi}}{d\P_{0,\pi}}(Y)=&\frac{\E_{\rho}\cro{d\P_{\mu,\pi}(Y)}}{\E_{0}\cro{d\P_{\mu,\pi}(Y)}}\\
    =&\frac{\E_{\rho}\cro{\prod_{k\in[K]}\prod_{m\in [M]}\exp\pa{-\frac{1}{2}\|Y_{k}^{(m)}-\mu_{\pi_m(k)}\|^{2}}}}{\E_{0}\cro{\prod_{k\in[K]}\prod_{m\in [M]}\exp\pa{-\frac{1}{2}\|Y_{k}^{(m)}-\mu_{\pi_m(k)}\|^{2}}}}\\
    =&\frac{\E_{\rho}\cro{\prod_{d\in[p]}\prod_{k\in[K]}\prod_{m\in [M]}\exp\pa{-\frac{1}{2}\pa{Y_{k,d}^{(m)}-\mu_{\pi_m(k),d}}^{2}}}}{\prod_{d\in[p]}\prod_{k\in[K]}\prod_{m\in [M]}\exp\pa{-\frac{1}{2}\pa{Y_{k,d}^{(m)}}^2}}\enspace.
\end{align*}
Using the independence of the $\mu_{k,d}$'s under the law $\rho$, we get that 
\begin{align*}
     \frac{d\P_{\rho,\pi}}{d\P_{0,\pi}}(Y)=&\prod_{d\in[p]}\prod_{k\in[K]}\frac{\E_{\rho}\cro{\prod_{m\in [M]}\exp\pa{-\frac{1}{2}\pa{Y_{\pi^{-1}(k),d}^{(m)}-\mu_{k,d}}^{2}}}}{\prod_{m\in [M]}\exp\pa{-\frac{1}{2}\pa{Y_{\pi^{-1}(k),d}^{(m)}}^2}}\\
    =&\prod_{d\in[p]}\prod_{k\in[K]}\E_{\rho}\cro{\prod_{m\in [M]}\exp{\pa{-\frac{1}{2}\pa{(Y_{\pi^{-1}(k),d}^{(m)}-\mu_{k,d})^{2}-(Y_{\pi^{-1}(k),d}^{(m)})^{2}}}}}\\
    =&\prod_{d\in[p]}\prod_{k\in[K]}\E_{\rho}\cro{\prod_{m\in [M]}\exp{\pa{Y_{\pi^{-1}(k),d}^{(m)}\mu_{k,d}-\frac{\eps^{2}}{2}}}}\\
    =&\prod_{d\in[p]}\prod_{k\in[K]}e^{\frac{-M\eps^{2}}{2}}\cosh{\pa{\sum_{m\in [M]}\eps Y_{\pi^{-1}(k),d}^{(m)}}}\enspace.
\end{align*}
Similarly, we have
\begin{equation*}
    \frac{d\P_{\rho,id}}{d\P_{0,id}}(Y)=\prod_{d\in[p]}\prod_{k\in[K]}e^{\frac{-M\eps^{2}}{2}}\cosh{\pa{\sum_{m\in [M]}\eps Y_{k,d}^{(m)}}}\enspace.
\end{equation*}
Combining these two equalities in \eqref{eq:vraisemblancepartial}, we end up with
\begin{equation}\label{eq:vraisemblance2partial}
    \frac{d\P_{\rho,\pi}}{d\P_{\rho,id}}(Y)=\prod_{d=1}^{p}\prod_{k=1}^{K}\frac{\cosh{\pa{\sum_{m\in [M]}\eps Y_{\pi_m^{-1}(k), d}^{(m)}}}}{\cosh{\pa{\sum_{m\in [M]}\eps Y_{k,d}^{(m)}}}}\enspace.
\end{equation}
We denote $\phi$ the standard Gaussian density $\phi(x)=\frac{1}{\sqrt{2\pi}}e^{\frac{-x^{2}}{2}}$. Under the law $\P_{\rho, \pi}$, conditionally on $\mu_{1},\ldots \mu_{K}\sim\rho$, we have that: 
\begin{itemize}
    \item $\sum_{m\in [M]}\eps Y_{\pi_m^{-1}(k), d}^{(m)}\sim \mathcal{N}(M\mu_{k,d},M)$,
    \item $\sum_{m\in [M]}\eps Y_{k, d}^{(m)}\sim \mathcal{N}(\sum_{m\in[M]}\mu_{\pi_m(k),d}, M)$.
\end{itemize}
Plugging these two points, together with equality \eqref{eq:vraisemblance2partial}, in the definition of the Kullback-Leibler divergence leads to
\begin{align}\nonumber
    \frac{1}{p}KL(\P_{\rho, \pi},\P_{\rho,id})=&\sum_{k\in[K]}\E_{\rho}\cro{\int\log\cosh(\eps(M \mu_{k,1}+\sqrt{M}x))\phi(x)dx}\\
    &-\sum_{k\in[K]}\E_{\rho}\cro{\int\log\cosh(\eps(\sum_{m\in [M]} \mu_{\pi_{m}(k),1}+\sqrt{M}x))\phi(x)dx}\enspace.\label{eq:KLpartial}
\end{align}
By symmetry, it is sufficient to upper-bound the term corresponding to $k=1$ in the sum above. We call $S_1=S_{11}-S_{12}$ the corresponding terms with 
$S_1=S_{11}-S_{12}$, with 
\begin{align*}
    S_{11}&=\E_{\rho}\cro{\int\log\cosh(\eps(M \mu_{1,1}+\sqrt{M}x))\phi(x)dx},\\
     \text{and}\quad S_{12}&=\E_{\rho}\cro{\int\log\cosh(\eps(\sum_{m\in [M]} \mu_{\pi_{m}(1),1}+\sqrt{M}x))\phi(x)dx}.
\end{align*}
First, we upper-bound the term $S_{11}$. We will use the following inequality 
\begin{equation}\label{eq:Taylor2}
    \frac{x^{2}}{2}-\frac{x^4}{12}\leq  \log\cosh(x)\leq \frac{x^2}{2},\enspace \forall x\in\R\enspace.
\end{equation}
\begin{proof} [Proof of inequality \eqref{eq:Taylor2}]
    Let us first prove the upper-bound $\log\cosh(x)\leq \frac{x^{2}}{2}$. For $x\in\R$, $\cosh(x)=\sum_{t\in\N}\frac{x^{2t}}{(2t)!}\leq \sum_{t\in\N}\frac{x^{2t}}{t!2^{t}}=\exp\pa{\frac{t^{2}}{2}}$. Applying the logarithmic function leads us to $\log\cosh(x)\leq \frac{x^2}{2}$.

    Let us now prove the lower-bound $\log\cosh(x)\geq \frac{x^{2}}{2}-\frac{x^{4}}{12}$. To do so, we write, for $x\geq 0$, $f(x)=\tanh(x)$ and $g(x)=x-\frac{x^{3}}{3}$. We have $f'(x)=1-(f(x))^{2}$. Besides, $g'(x)=1-x^{2}\leq 1-(g(x))^{2}$. Together with the fact that $f(0)=g(0)=0$, this implies 
    $$f(x)\geq g(x),\enspace \forall x\geq 0\enspace.$$ 
    Integrating these function leads to $$\log\cosh(x)\geq \frac{x^{2}}{2}-\frac{x^{4}}{12},\enspace \forall x\geq 0\enspace.$$ By parity of these functions, this last inequality is satisfied for all $x\in\R$. 
\end{proof}
Inequality \eqref{eq:Taylor2}, together with the independence of $x$ and $\mu_{1,1}$, imply that 
\begin{equation}
    S_{11}\leq\frac{1}{2} \E_{\rho}\cro{\int(\eps(M\mu_{1,1}+\sqrt{M}x))^{2}\phi(x)dx}\leq \frac{1}{2}\eps^{2}\pa{M+M^{2}\eps^{2}}\enspace.
\end{equation}
We arrive at 
\begin{equation}\label{eq:S11}
    S_{11}\leq \frac{1}{2}M\eps^{2}(1+M \eps^{2})\enspace.
\end{equation}
Let us now lower-bound $S_{12}$. For $k\in [K]$, let us write $m_k=|\ac{m\in [M], \pi_m(1)=k}|$. Inequality \eqref{eq:Taylor2} induces
\begin{align*}
    S_{12}\geq&\frac{1}{2}\E_{\rho}\cro{\int\eps^2\bigg(\sum_{k\in[K]} m_{k}\mu_{k,1}+\sqrt{M}x\bigg)^{2}\phi(x)dx}\\
    &-\frac{1}{12}\E_{\rho}\cro{\int\eps^4\bigg(\sum_{k\in[K]} m_{k}\mu_{k,1}+\sqrt{M}x\bigg)^{4}\phi(x)dx}\\
    \geq & \frac{1}{2}\eps^{2}\pa{\sum_{k\in [K]}m_{k}^{2}\eps^{2}+M}
    -\frac{1}{12}\eps^{4}\E_{\rho}\left[\bigg(\sum_{k\in[K]}m_{k}\mu_{k,1}\bigg)^4\right]-\frac{1}{12}\eps^{4}3M^{2}\\
    &-\frac{1}{2}\eps^{4}M\E_{\rho}[(\sum_{k\in[K]}m_{k}\mu_{k,1})^2]\\
    \geq & \frac{1}{2}\eps^{2}\pa{\sum_{k\in [K]}m_{k}^{2}\eps^{2}+M}
    -\frac{1}{12}\eps^{8}\pa{\sum_{k\in[K]}m_{k}^{4}+6\pa{\sum_{k\in[K]}m_{k}^{2}}^{2}}-\frac{1}{4}\eps^{4}M^{2}\\
    &-\frac{1}{2}\eps^{6}M\sum_{k\in[K]}m_{k}^2\enspace.
\end{align*}
We end up with
\begin{equation}\label{eq:S12}
    S_{12}\geq \frac{1}{2}M\eps^{2}-\frac{1}{4}M^{2}\eps^{4}-\frac{1}{2}\eps^{6}M\sum_{k\in [K]}m_{k}^{2}-\frac{1}{12}\eps^{8}\pa{\sum_{k\in [K]}m_{k}^{4}+6\pa{\sum_{k\in [K]}m_{k}^{2}}^{2}}\enspace.
\end{equation}
Combining inequalities \eqref{eq:S11} and \eqref{eq:S12}, together with $M=\sum_{k\in [K]}m_{k}$, we get
\begin{align*}
    S_{1}\leq &\frac{1}{2}M\eps^{2}(1+M\eps^{2})-\frac{1}{2}M\eps^{2}
    +\frac{1}{4}M^{2}\eps^{4}+\frac{1}{2}\eps^{6}M\sum_{k\in [K]}m_{k}^{2}
    +\frac{1}{12}\eps^{8}\pa{\sum_{k\in[K]}m_{k}^{4}+6\pa{\sum_{k\in[K]}m_{k}^{2}}^{2}}\\
    \leq &\frac{3}{4}M^{2}\eps^{4}+\frac{7}{12}\eps^{8}M^{4}+\frac{1}{2}\eps^{6}M^{3}\\
    \leq &cM^2\eps^{4}\pa{1+\eps^{2}M+\eps^{4}M^2}\enspace,
\end{align*}
where $c$ is a numerical constant. Summing over $k\in[K]$ in \eqref{eq:KLpartial} leads to 
$$KL(\P_{\rho, \pi},\P_{\rho,id})\leq cpKM^2\eps^{4}\pa{1+\eps^{2}M+\eps^{4}M^2}\enspace.$$
From the definition of $\eps$, we deduce that $pM^2K\eps^{4}=\frac{16M^2K}{p}\Bar{\Delta}^{4}$. The hypothesis $\Bar{\Delta}^{2}\leq c_2\sqrt{\frac{p}{M}\log(K)}$ leads to $pMK\eps^{4}\leq 16c_{2}^{2}MK\log(K)$. Moreover, the hypothesis $c_{1}\log(K)\leq\Bar{\Delta}^{2}\leq c_2\sqrt{\frac{p}{M}\log(K)}$ implies $\eps^{2}M=\frac{4M}{p}\frac{\Bar{\Delta}^{4}}{\Bar{\Delta}^{2}}\leq 4\frac{c_{2}^{2}}{c_{1}}\leq 1$ provided that  $c_{2}$ is chosen small enough so that $c_2\leq \sqrt{c_1}$. Thus, there exists a numerical constant $c>0$ such that, when $c_2$ is small enough with respect to $c_1$, 
$$KL(\P_{\rho, \pi},\P_{\rho,id})\leq cc_{2}^2MK\log(K)\enspace.$$
This concludes the proof of the lemma.

\subsection{Proof of Lemma \ref{lem:numberpermutations}}\label{prf:numberpermutations}

To prove Lemma \ref{lem:numberpermutations}, we will use Lemma 16 from \cite{Collier16} that we recall here.
\begin{lem}
    For any $K\geq 4$. There exists $\psi^{(1)},\ldots ,\psi^{(\ell)}\in\mathcal{S}_{K}$ such that the following points hold: \begin{enumerate}
        \item $\ell\geq (\frac{K}{24})^{\frac{K}{6}}$, 
        \item For any $t\neq t'\in[\ell]$, $\frac{1}{K}\sum_{k\in[K]}\1_{\psi^{(t)}(k)\neq \psi^{(t')}(k)}\geq \frac{3}{8}$.
    \end{enumerate}
\end{lem}
If $M\leq 3$, we simply define $V$ of size $\ell$ made of $\pi$'s of the form $\pi_1=\pi_2=id$ and $\pi_3= \psi^{(l)}$ for some $l=1,\ldots, \ell$. Obviously, this set $V$ satisfies the desired properties of Lemma~\ref{lem:numberpermutations}.

Let us now consider the general case where $M\geq 4$. We construct $V$ as follows. Denote $\overline{V}$ the subset of $(\mathcal{S}_{K})^{M}$ made from the $M$-tuples $\pi$ satisfying the two conditions; \begin{enumerate}
    \item if $m\leq \lceil \frac{M}{2}\rceil$, $\pi_{m}=id$;
    \item if $m> \lceil \frac{M}{2}\rceil$, $\pi_{m}$ is one of the $\psi^{(1)},\ldots ,\psi^{(\ell)}$. 
\end{enumerate}
We take for $V$ a maximal family in $\overline{V}$ such that for $\pi$ and $\pi'$ in $V$, there are at least $\frac{M}{4}$ index such that $\pi_{m}\neq \pi'_{m}$. Since for $t\neq t'\in[\ell]$, $\frac{1}{K}\sum_{k\in[K]}\1_{\psi^{(t)}(k)\neq \psi^{(t')}(k)}\geq \frac{3}{8}$, we have that, for $\pi\neq\pi'\in V$, $\frac{1}{KM}\sum_{k\in [K]}\sum_{m\in [M]}\1\ac{\pi_m(k)\neq \pi'_m(k)}\geq \frac{3}{32}.$

It remains to lower-bound the cardinality of $V$. First, we know that $|\overline{V}|= \ell^{\lfloor\frac{M}{2}\rfloor}$. For $\pi\in V$, denote by $B_{\pi}$ the subset of $\overline{V}$ made from all the $M$-tuple $\pi'\in |\overline{V}|$ that have less than $\frac{M}{4}$ indices such that $\pi_{m}\neq \pi'_{m}$. The definition of $V$ implies that $\overline{V}\subset \bigcup_{\pi\in V}B_{\pi}$. A fortiori, since the cardinality of $B_{\pi}$ is constant over $V$, we have $\ell^{\lfloor\frac{M}{2}\rfloor}\leq |V||B_{\pi}|$ for any $\pi \in V$. 

Let us upper-bound $|B_{\pi}|$ for some $\pi \in V$. Any $\pi'\in B_{\pi}$ has at least $\frac{M}{4}$ indice such that $\pi_{m}=\pi'_{m}$ among the indices $m> \lceil\frac{M}{2}\rceil$. Fixing $\lceil\frac{M}{4}\rceil$ indices where the permutations are equal, there are at most $\ell^{\frac{M}{4}}$ possibilities for choosing an element in $\overline{V}$ that is equal to $\pi'$ on these index. Thus, we get that $|B_{\pi}|\leq \binom{\lfloor\frac{M}{2}\rfloor}{\lceil \frac{M}{4}\rceil}\ell^{\frac{M}{4}}$. Since $\binom{x}{y}\leq (xe/y)^y$. 
 We end up with \begin{align*}
    \log|V|\geq& \log(\ell^{\lfloor\frac{m}{2}\rfloor})-\log|B_{\pi}|\\
    \geq&\lfloor\frac{M}{2}\rfloor\log(\ell)-\frac{M}{4}\log(\ell)-\lceil \frac{M}{4}\rceil \log(e\frac{\lfloor\frac{M}{2}\rfloor}{\lceil \frac{M}{4}\rceil})\\
    \geq& \frac{KM}{18}\log(\frac{K}{24})-\frac{KM}{24}\log(\frac{K}{24})-\lceil \frac{M}{4}\rceil \log(\frac{e\lfloor\frac{M}{2}\rfloor}{\lceil \frac{M}{4}\rceil})\enspace.
\end{align*} 
When $K_{0}$ is large enough, these inequalities imply $\log|V|\gtrsim MK\log(K)$ whenever $K\geq K_0$. This concludes the proof of the Lemma.

\subsection{Proof of Lemma \ref{lem:reductionfano2}}\label{prf:reductionfano2}

We suppose that there exists a probability distribution $\rho$ on $(\R^{p})^{K}$ and $C>0$ satisfying $$\inf_{\hat{\pi}:\R^{K\times M\times p}\to (\mathcal{S}_{K})^{M}}\sup_{\pi\in(\mathcal{S}_{K})^{M}}\E_{\rho,\pi}\cro{err(\hat{\pi},\pi)}-\rho((\R^{p})^{K}\setminus \Theta_{\bar{\Delta}})\geq C\enspace.$$
Given $\hat{\pi}:(\R^{p})^{KM}\to (\mathcal{S}_{K})^{M}$, there exists a $M$-tuple of permutation $\pi$ such that $\E_{\rho,\pi}\cro{err(\hat{\pi}\neq \pi)}-\rho((\R^{p})^{K}\setminus \Theta_{\bar{\Delta}})\geq C$. Since $\E_{\rho,\pi}\cro{err(\hat{\pi}, \pi)}=\int_{(\R^{p})^{K}}\E_{\mu,\pi}\cro{err(\hat{\pi}, \pi)}d\rho(\mu)$ and $\E_{\mu,\pi}\cro{err(\hat{\pi}, \pi)}$ is upper bounded by $1$, we end up with $\int_{\Theta_{\Delta}}\E_{\mu,\pi}\cro{err(\hat{\pi}, \pi)}d\rho(\mu)\geq C$.

This implies the existence of $\mu\in\Theta_{\Delta}$ such that $\E_{\mu,\pi}\cro{err(\hat{\pi}, \pi)}\geq C$. A fortiori, $$\sup_{\mu\in\Theta_{\Delta}}\sup_{\pi\in(\mathcal{S}_K)^M}\E_{\mu,\pi}\cro{err(\hat{\pi}, \pi)}\geq C\enspace .$$ 
This last inequality being true for all estimator $\hat{\pi}$, we get the desired result $$\inf_{\hat{\pi}:\R^{K\times M\times p}\to (\mathcal{S}_{K})^{M}}\sup_{\mu\in\Theta_{\bar{\Delta}}}\sup_{\pi\in(\mathcal{S}_{K})^{M}}\E_{\mu,\pi}\cro{err(\hat{\pi},\pi)}\geq C\enspace.$$

\subsection{Proof of Lemma \ref{lem:Fano2}}\label{sec:prfFano2}

Let $A$ be a subset of $(\mathcal{S}_K)^M$. Denote $\pi^{(1)},\ldots ,\pi^{(|A|)}$ the elements of $A$. Given any estimator $\hat{\pi}:\R^{K\times M\times p}\to (\mathcal{S}_{K})^{M}$, we denote $\hat{j}$ an index that minimises $err(\hat{\pi},\pi^{(j)})$.

For any $j\in[1,|A|]$, using the definition of $\hat{j}$ together with the fact that the function $err$ satisfies the triangular inequality, we have \begin{align*}
    min_{j\neq j'}err(\pi^{(j)},\pi^{(j')})\1_{\hat{j}\neq j}&\leq err(\pi^{(\hat{j})},\pi^{(j)})\\
    &\leq err(\hat{\pi},\pi^{(\hat{j})})+err(\hat{\pi},\pi^{(j)})\\
    &\leq 2err(\hat{\pi},\pi^{(j)})\enspace.
\end{align*}
Taking the expectation over this last inequality, we get, for $j\in[1,|A|]$, $$2\E_{\rho,\pi^{(j)}}\cro{err(\hat{\pi},\pi^{(j)})}\geq min_{\pi\neq \pi'\in A}err(\pi,\pi') \P_{\rho,\pi^{(j)}}\cro{\hat{j}\neq j}\enspace. $$ Summing over $\pi\in A$ leads to $$\frac{2}{|A|}\sum_{\pi\in A}\E_{\rho,\pi}\cro{err(\hat{\pi},\pi)}\geq \min_{\pi,\pi'\in A}err(\pi,\pi')\min_{\hat{\pi}:\R^{K\times M\times p}\to A}\frac{1}{|A|}\sum_{\pi\in A} \P_{\rho,\pi}\cro{\hat{\pi}\neq \pi}.$$Applying Lemma \ref{lem:fano}, we arrive at the desired conclusion $$\inf_{\hat{\pi}:\R^{K\times M\times p}\to (\mathcal{S}_{K})^{M}}\frac{2}{|A|}\sum_{\pi\in A}\E_{\rho,\pi}\cro{err(\hat{\pi},\pi)}\geq \min_{\pi,\pi'\in A}err(\pi,\pi')\pa{1-\frac{1+\frac{1}{|A|}\sum_{\pi\in A}KL(\P_{\rho,\pi},\P_{\rho,id})}{\log(|A|)}}\enspace.$$

\subsection{Proof of Lemma \ref{lem:choicepoints}}\label{sec:prfchoicepoints}

We suppose that $\frac{p}{4}\log(2)\geq \log(K)$ and we want to find vectors $\mu_{1},\ldots ,\mu_{K}$ in $\R^{p}$ such that $\frac{1}{2}\min_{k\neq l}\|\mu_{k}-\mu_{l}\|^2\geq \bar{\Delta}^2$ and $\frac{1}{2}\max_{k\neq l}\|\mu_{k}-\mu_{l}\|\leq 4\bar{\Delta}^2.$ Denote $\mathcal{H}=\bar{\Delta} \sqrt{\frac{2}{p}}\{-1,1\}^{p}$. There are $2^{p}$ elements in $\mathcal{H}$. Consider $H$ a maximal subset of $\mathcal{H}$ such that the following holds. For $\mu$ and $\mu'$ two distinct elements of $H$, there exist at least $p/4$ indices such that $\mu_{i}\neq \mu'_{i}$.

If $\mu$ and $\mu'$ are two distinct elements of $H$, we have from the definitions of $\mathcal{H}$ and $H$ that $$\bar{\Delta}^2\leq \frac{1}{2}\|\mu-\mu'\|^2\leq 4\bar{\Delta}^2\enspace.$$
It remains to lower-bound the cardinality of $H$ and to prove $|H|\geq K$. Given $\mu\in H$, denote $B_{\mu}$ the subset of $\mathcal{H}$ made of the elements $\mu'$ that have at least $\lfloor\frac{3p}{4}\rfloor$ indices such that $\mu_{i}=\mu'_{i}$. From the definition of $H$, we deduce that $\mathcal{H}\subseteq\bigcup_{\mu\in H}B_{\mu}$.

Let us upper-bound $|B_{\mu}|$ for any $\mu\in H$. Given a fixed set of index of cardinal $\lfloor\frac{3p}{4}\rfloor$, there are at most $2^{\frac{p}{4}}$ points in $\mathcal{H}$ that are equal to $\mu$ on these indices. Hence, the cardinal of $B_{\mu}$ is upper-bounded by $\binom{p}{\lfloor \frac{3p}{4}\rfloor}2^{\frac{p}{4}}$. 

Combining this with $\mathcal{H}\subseteq \bigcup_{\mu\in H}B_{\mu}$, we arrive at $\log(|H|)\geq p\log(2)-\log\binom{p}{\lfloor \frac{3p}{4}\rfloor}-\frac{p}{4}\log(2)$. We have  
\[
\log\binom{p}{\lfloor \frac{3p}{4}\rfloor}\leq \lceil \frac{p}{4}\rceil \log(ep/\lfloor \frac{3p}{4}\rfloor)\leq  \frac{p}{2}\log(2)
\ ,
\]
since $p\log(2)/4\geq \log(K)\geq \log(K_0)$ is large enough. Thus, we end up with $\log(|H|)\geq \frac{p}{4}\log(2)\geq \log(K)$. This concludes the proof.

\end{document}